\documentclass[%
  paper=A4, 
  DIV=calc, 
  smallheadings,
  10pt,
]{scrbook}  

\DeclareRobustCommand*{\onlyattoc}[1]{}
\newcommand*{\activateonlyattoc}{%
  \DeclareRobustCommand*{\onlyattoc}[1]{##1}%
}
\AtBeginDocument{\addtocontents{toc}{\protect\activateonlyattoc}}

\usepackage{t1enc}      
\usepackage[utf8]{inputenc}

\usepackage{longtable} 
\usepackage{latexsym} 
\usepackage[format=plain,labelsep=period]{caption}
\usepackage{graphicx}
\usepackage{typearea}
\usepackage{amsmath, amssymb}
\usepackage{color}
\usepackage{multirow}
\usepackage{icomma}
\usepackage{dsfont}
\usepackage[pdftex,%
        colorlinks=false,%
        linkcolor=true,%
        bookmarksopen=true,%
        pdfborder=0,
        bookmarksnumbered=true,
        plainpages=false,
        pdfpagelabels]{hyperref}
\usepackage{amssymb} 
\usepackage{amsthm}
\usepackage{enumerate}
\usepackage{bm}
\usepackage{booktabs}
\usepackage{mathrsfs}
\usepackage{float}
\usepackage{tikz}
\usepackage{makeidx}
\usepackage{array}
\newcolumntype{L}[1]{>{\raggedright\let\newline\\\arraybackslash\hspace{0pt}}m{#1}}
\addtokomafont{sectioning}{\rmfamily}

\renewcommand{\bibname}{References}

\usepackage{scrpage2}
\ohead[\pagemark]{\pagemark}
\ofoot{}

\automark{chapter}

\automark[section]{chapter}
\newcommand{\T}{\textstyle}

\newcommand{\SIR}{\mathrm{SIR}}
\newcommand{\esssup}{\mathop{\mathrm{ess~sup}}}
\newcommand{\essinf}{\mathop{\mathrm{ess~inf}}}
{\left\lbrace\begin{array}{@{}l@{}}}%
{\end{array}\right.}

\input{epsf.sty}
\usepackage{hyperref}
\usepackage{amsmath, amssymb}
\usepackage{subfigure} 

\theoremstyle{plain}
\newtheorem {thm}{Theorem}[chapter]
\newtheorem {prop}[thm]{Proposition}
\newtheorem {claim}[thm]{Claim}

\newtheorem {lem}[thm]{Lemma}
\newtheorem {cor}[thm]{Corollary}
\theoremstyle{definition}
\newtheorem {defi}[thm]{Definition}

\newtheorem {ass}[thm]{Assumption}
\theoremstyle{remark}
\newtheorem {rem}[thm]{Remark}
\newtheorem {exam}[thm]{Example}

\newcommand\numberthis{\addtocounter{equation}{1}\tag{\theequation}}

\def\four{{\lbrace 1,\ldots,4 \rbrace}}
\def\d{\mathrm d}

\def\phi{\varphi}

\def\T{\T}

\def\VV|{{\vvvert}}
\def\V|{{\Vert}}

\newcommand{\tensor}{\mathop{\otimes}}
\def\Finite{0<\Fmin\leq\Fmax<\infty}
\def\ellmin{\ell_{\min}}
\def\ellmax{\ell_{\max}}
\def\Fmax{F_{\max}}
\def\Fmin{F_{\min}}
\def\alphaminus{\alpha_{-}}

\def\Fdelta{[\Fmin, \Fmax]_\delta}

\makeindex

\begin{document}

\begin{titlepage}
\titlehead{\vspace{-2.5cm}} 
\subject{\vspace*{-3cm}\centering
\large{Master's Thesis
					}}
\title{Highly dense mobile communication networks with random fadings}

\author{
				\textbf{András József Tóbiás}
				}

\publishers{\centering\includegraphics*[width=5cm]{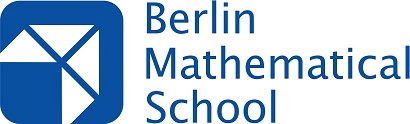} \vspace{50pt} \\ \includegraphics*[width=5cm]{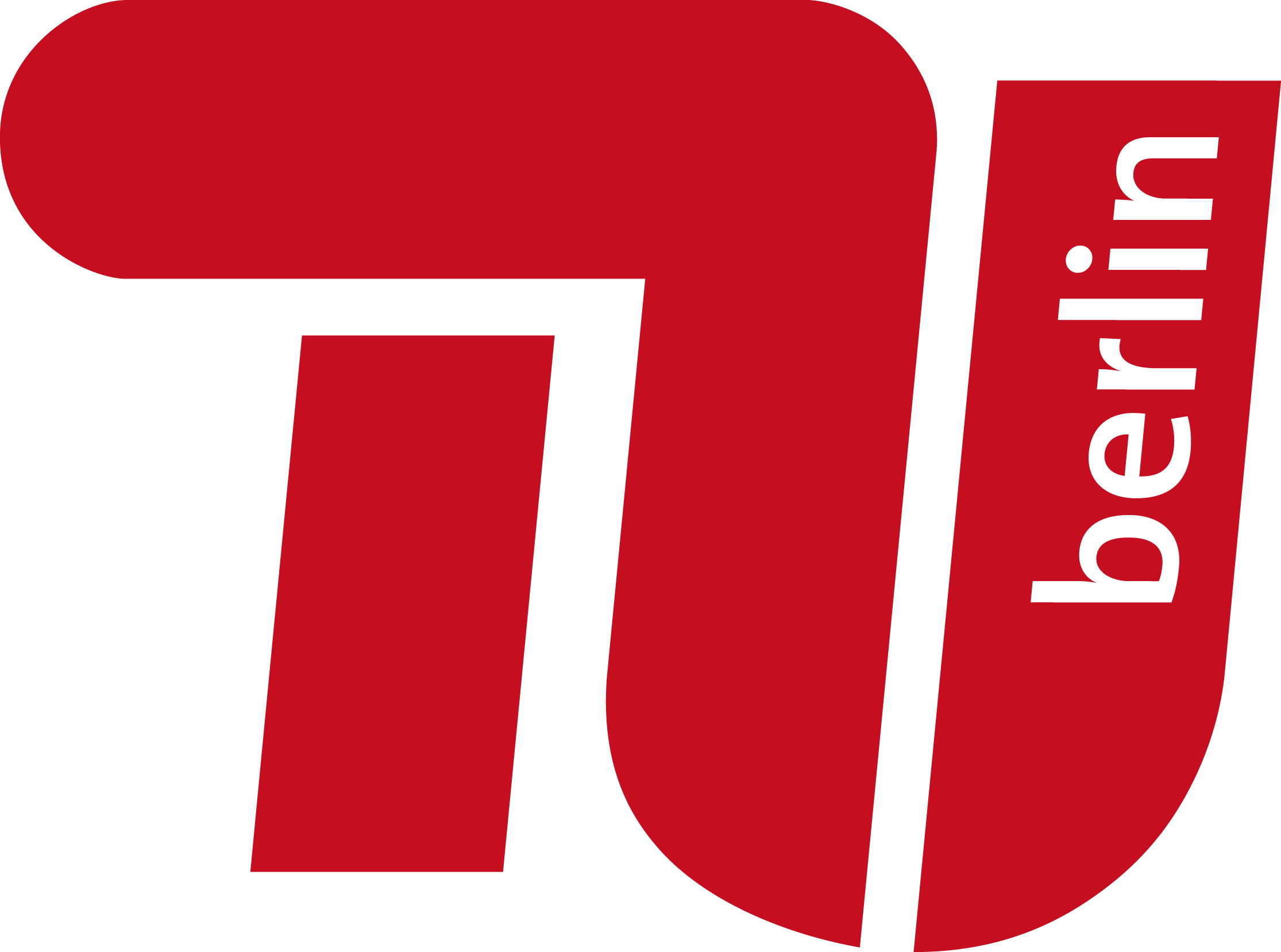} \\
		\large{Technische Universität Berlin \\
		Fakultät II \\
		Institut für Mathematik \\
		AG Stochastik und Finanzmatematik}
		\vspace{1cm}
		\begin{tabbing}
	\large{Supervisor:} \quad \quad \= \large{Prof.} \large{Dr. Wolfgang König}\\
	\large{Second reader:} \quad \> \= \large{~~~~~~~~Dr. Christian Hirsch}\\
		\end{tabbing}
		
		\vspace{-3cm}
}
\date{		Submitted on April 18, 2016 \\ Revised on \today}


\maketitle
\end{titlepage}
\pagenumbering{roman}
\setcounter{page}{1}
\pagestyle{scrheadings}
\color{white} u \color{black}
\vspace{200pt}

\section*{Erklärung}
\begin{flushleft}
Hiermit erkläre ich, dass ich die vorliegende Arbeit selbstständig und eigenhändig
sowie ohne unerlaubte fremde Hilfe und ausschließlich unter Verwendung
der aufgeführten Quellen und Hilfsmittel angefertigt habe. 
\vspace{5pt}

Die selbständige und eigenständige Anfertigung versichert an Eides statt:
\vspace{10pt}

Berlin, den 18. April 2016 
\end{flushleft}
\vspace{50pt}
András Tóbiás
$\vspace{3cm}$


\newpage
\section*{Acknowledgements}
I would like to thank my supervisor Wolfgang König for the opportunity to write a Master's thesis about this topic and to work together with his colleagues, and also for his support, advice and discussions about my thesis and RTG application.

I would like to thank Christian Hirsch for his ideas about how to proceed with my thesis, discussing about the topic, answering my questions many times and proofreading various parts of my writing.

I would like to thank the Berlin Mathematical School for funding my studies and providing me an office space and a computer pool for working in the BMS Lounge. I would like to thank my fellow students at BMS, especially Adrián Hinojosa Calleja, for interesting discussions. I would also like to thank Stanley Schade for proofreading the German summary of the thesis.

\newpage

\tableofcontents
\newpage
\thispagestyle{empty}
\color{white} . \color{black}
\newpage

\pagenumbering{arabic}
\setcounter{page}{1}

\pagestyle{scrheadings}

\chapter{Introduction}\label{Introduction}
\begin{quote}
He had always been enthralled by the methods of natural science, but the ordinary subject matter of that science had seemed to him trivial and of no import. And so he had begun by vivisecting himself, as he had ended by vivisecting others. Human life—that appeared to him the one thing worth investigating. \\
Oscar Wilde: \emph{The Picture of Dorian Gray,} Chapter 4. 
\end{quote}

In this Master's thesis, we consider a mathematical model for a wireless network. We investigate the large deviation asymptotics of interference in the high-density limit, when the number of users tends to infinity. In the same time, the communication area where the users are situated remains the same compact subset $W \subset \mathbb R^d$. The thesis is based on the current paper \cite{cikk} of C. Hirsch, B. Jahnel, H. P. Keeler and R. I. A. Patterson. Particular in my thesis, the main goal is to add \emph{random fadings} to the model. The fadings are positive random variables, which are interpreted as loudnesses of the users in the system. 

In the new model with random fadings, one has to encounter two sources of randomness. One is the spatial positions of the users, which also accounts for the number of users, who form a Poisson point process on $W$. The other is the realization of the fadings of the users. The main question is which one is the more dominant source of \emph{bad connection}. That is, if the interference in the system is too high, is it caused by spatial effects, i.e. that too many users gather at the same place, or fading effects, i.e. that some users are too loud? In particular, how do the configurations that exhibit too many users with unsatisfactory connection depend on the distribution of the fadings?

In order to answer these questions, we proceed as follows. First, in Chapter \ref{előzés}, we provide the reader with the most important preliminaries that are needed in order to start our own investigations. These contain basic notions and results about large deviations (following the book \cite{LDP} of Dembo and Zeitouni) and about Poisson point processes (following the book \cite{kingman} of Kingman). Moreover, we sketch the original model definition and main results of \cite{cikk}, in order to provide a comparison with our generalized model. 

In Chapter \ref{eleje}, we extend the model of this paper with independent and identically distributed (i.i.d.) random fadings, and we follow the large deviation approach of \cite[Sections 2--6]{cikk}. This way, we obtain results that are generalizations of the theorems of the original fading-free setting. One of these results implies that the probability of having an unlikely number of users with bad connection decays exponentially. The other result is that we obtain the most likely configuration that exhibits a certain number of users with bad connection (in other words, we will call these users \emph{frustrated}). The methods of \cite{cikk} can be generalized to this case as long as the fadings are bounded and bounded away from 0.

Knowing these general results, in Chapter \ref{effect} we investigate the effects coming from random fadings more delicately. In particular, in this chapter we encounter several mathematical problems that do not have an analogue from the fading-free setting of \cite{cikk}. In the first part of this chapter, we answer the questions above and several other ones. Some of them can be answered directly by the formulas obtained in Chapter \ref{eleje}, some other ones we can solve theoretically in special settings exhibiting certain symmetries, and for some other ones we use numerical computations and simulations to conjecture the answer. In the second part of Chapter \ref{effect}, we generalize the model of Chapter \ref{eleje} in two ways, where one can obtain analogous results about the frustration probabilities. On the one hand, we observe the fadings of a user do not have to be i.i.d., they can depend on the spatial position of the user. As long as the fading distribution of a given users is independent of both the spatial positions and the loudnesses of all other users, we can deduce analogues of the results of Chapter \ref{eleje} conveniently. On the other hand, we show that the approach of Chapter \ref{eleje} is still valid in the case when the fading value of the base station $o$, with which the users communicate in the model, is also random.

In Chapter \ref{summary}, we summarize and conclude the thesis. We also sketch several open questions and possible directions along which related research can be continued. 
\chapter{Preliminaries} \label{előzés}
In this chapter, we summarize the basic notions about large deviations and about Poisson point processes in order to start the own investigations of this thesis. We also set up some general notation, which we summarize in Section \ref{indexofnotations} in the Appendix. The last section of the chapter sums up the model description and the main results of the paper \cite{cikk}. This article serves as a basis for our work: after this chapter, we construct modified versions of the model of this paper, incorporating random fadings, which serve as loudnesses of the users.
\section{Large deviations} \label{ldpcske}
In this section, we enumerate the definitions and results that we will need for our large deviation (LD) approach, following the book of Dembo and Zeitouni \cite{LDP} directly. We only mention those ones that we will need in the following parts of the thesis, and we omit the proofs of the results.

Throughout this section, let $X$ be a Hausdorff topological space, and let $\mathcal B \subseteq \mathcal{P}(X)$ be a $\sigma$-algebra. Unless we explicitly say the opposite, $\mathcal B$ will be contain the Borel $\sigma$-algebra of $X$. We note that in many cases that one considers in probability theory, $X$ is a Polish space, i.e., the topology of $X$ can be induced by a metric that makes $X$ a separable, complete metric space.
\begin{defi} \index{rate function}
A \emph{rate function} is a lower semicontinuous mapping $I: X \to [0,\infty]$. I.e., for all $\alpha \in [0,\infty)$, the level set $\Psi_I(\alpha)=\lbrace x \in X \vert~I(x) \leq \alpha \rbrace$ is a closed subset of $X$. A \emph{good rate} function is a rate function for which all the level sets $\Psi_I(\alpha)$ are compact.
\end{defi}
Lower semicontinuity can be defined for general topological spaces $X$ this way, but if $X$ is a metric space, then lower semicontinuity can be checked on sequences. I.e., $I$ is lower semicontinuous if and only if $\liminf_{x_n \to x} I(x_n) \geq I(x)$ for all $x \in X$. A consequence of a rate function being good is that the infimum is achieved over closed sets. 

A function $f: X \to \mathbb R$ is called upper semicontinuous if $-f$ is lower semicontinuous. We will write l.s.c. for lower semicontinuous and u.s.c. for upper semicontinuous. \index{semicontinuous function}

In the following, for any set $A \in \mathcal B$, we write $A^o$ for the interior of $A$, $\overline{A}$ for the closure of $A$ and $A^c$ for the complement of $A$. The infimum of a function over the empty set is interpreted as $\infty$.
\begin{defi} \index{large deviation principle}
A family of probability measures $\lbrace \mu_\varepsilon \rbrace$ on $(X,\mathcal B)$ satisfies the \emph{large deviation principle} with a rate function $I$, if for all $A \in \mathcal B$
\begin{equation} \label{ldpdef}  -\inf_{x \in A^o} I(x) \leq \liminf_{\varepsilon \to 0} \varepsilon \log \mu_\varepsilon (A) \leq \limsup_{\varepsilon \to 0} \varepsilon \log \mu_\varepsilon (A) \leq -\inf_{x \in \overline{A}} I(x). \end{equation}
\end{defi}
The left and right hand sides of \eqref{ldpdef} are referred to as lower and upper bounds, respectively. We write LDP for large deviation principle. 
\begin{lem}
$\lbrace \mu_\varepsilon \rbrace$ on $(X,\mathcal B)$ satisfies the LDP with a rate function $I$ if and only if the following two conditions are satisfied:
\begin{enumerate}[(i)]
\item \text{(Lower bound)} For any open set $G \subseteq X$,
\[ \liminf_{\varepsilon \to 0} \varepsilon \log \mu_\varepsilon (G) \geq -\inf_{x \in G} I(x). \]
\item \text{(Upper bound)} For any closed set $F \subseteq X$,
\[ \limsup_{\varepsilon \to 0} \varepsilon \log \mu_\varepsilon (F) \leq -\inf_{x \in F} I(x). \]
\end{enumerate}
\end{lem}
The following notion of exponential tightness is used in large deviation theory as a tool which makes it possible to conclude statements by proving weaker statements, see e.g. Proposition \ref{4.2.7}.
\begin{defi} \index{exponential tightness}
Suppose that the compact subsets of $X$ belong to $\mathcal B$. A family of probability measures $ \lbrace \mu_{\varepsilon} \rbrace$ is \emph{exponentially tight} if for all $\alpha>0$ there exists a compact subset $K_\alpha \subset X$ such that
\[ \limsup_{\varepsilon \to 0} \varepsilon \log \mu_\varepsilon (K_\alpha^c) < -\alpha. \] 
\end{defi}
For the first large deviation results that we describe here, we consider random variables that take only finitely many values in $\mathbb R^d$ for some $d \geq 1$. Let $\Sigma=\lbrace a_1, \ldots, a_{\vert \Sigma \vert} \rbrace$ be the \emph{underlying alphabet}, in which the independent and identically distributed (i.i.d.) random variables $Y_1,~Y_2,~\ldots$ take values, with $a_i \in \mathbb R^d$ for all $i \in \lbrace 1,\ldots,\vert \Sigma \vert \rbrace$, and $a_i \neq a_j$ if $i \neq j$. The next theorem can be seen as a simple \emph{level-2} large deviation result, in the case of finite alphabets, which means that it corresponds to the \emph{empirical measures} of the i.i.d. random variables $\lbrace Y_i \rbrace_{i \in \mathbb N}$.\footnote{ We write $\mathbb N=\lbrace 1,2,\ldots \rbrace$ and $\mathbb N_0=\lbrace 0,1,2,\ldots \rbrace$ About definitions of different levels of LDP, see \cite[Section 7]{level}.} For stating this theorem, we define the types and the relative entropy in this finite setting. Throughout this thesis, for any measurable space $(Z,\mathcal{B})$, let $\mathcal{M}(Z)$ denote the set of finite measures on $(Z,\mathcal{B})$, and $\mathcal{M}_1(Z)$ the set of probability measures on $(Z,\mathcal{B})$.\footnote{ If $X$ is any topological space and $\mathcal{B}$ is the Borel $\sigma$-algebra of $X$, then $\mathcal{M}(Z)$ is the set of finite Borel measures of $X$.}
\begin{defi}
The \emph{type} $L^{\mathbf y}_n$ if a finite sequence $\mathbf y=(y_1,\ldots,y_n) \in \Sigma^n$ is the empirical measure (law) induced by this sequence. Explicitly, $L^{\mathbf y}_n=\left( L^{\mathbf y}_n(a_1),\ldots,L^{\mathbf y}_n(a_{\vert \Sigma \vert}) \right)$ is the element of $\mathcal{M}_1(\Sigma)$ defined as \[ L^{\mathbf y}_n (a_i)=\sum_{j=1}^n \delta_{a_i}(y_j), \quad i=1,\ldots,\vert \Sigma \vert. \]
\end{defi}
We write $L_n^{\mathbf Y}$ for the \emph{random} type associated with the sequence $\mathbf Y=(Y_1,\ldots,Y_n)$.
\begin{defi}
Let $\nu=(\nu_1,\ldots,\nu_n),~\mu=(\mu_1,\ldots,\mu_n) \in \mathbb R^n$ be two vectors such that $\nu_i \geq 0$ and $\mu_i >0$ for all $i=1,\ldots,n$. The \emph{relative entropy} of a vector $\nu$ with respect to $\mu$ is defined as
\begin{equation} \label{diszkrétrelatíventrópia}\index{relative entropy!discrete} h(\nu \vert \mu)=\sum_{i=1}^{\vert \Sigma \vert} \nu(a_i) \log \frac{\nu(a_i)}{\mu(a_i)}-\nu(a_i)+\mu(a_i). \end{equation}
\end{defi}
Let now $n=\vert \Sigma \vert$ and let $\mu$ be a probability vector in $\mathbb R^n$. By applying Jensen's inequality to the convex function $x \log x$, it follows that the function $h(\cdot \vert \mu)$ is nonnegative. Moreover, $h(\cdot \vert \mu)$ is finite and continuous on the compact set $K=\lbrace \nu \in \mathcal{M}_1(\Sigma) \vert~\Sigma_\nu \subseteq \Sigma_\mu \rbrace$, because $x \log x$ is continuous on $[0,1]$. Here $\mathcal{M}_1(\Sigma)$ denotes the set of probability measures on $\Sigma$, and $\Sigma_\mu=\lbrace a_i : \mu(a_i)>0 \rbrace$ denotes the support of the probability vector $\mu$. Moreover, $h(\cdot \vert \mu)=\infty$ outside $K$, because of division by 0, and hence $h(\cdot \vert \mu)$ is a good rate function. Now we state the theorem.
\begin{thm}[Sanov] \index{Sanov's theorem}
For every set $A$ of probability vectors in $\mathcal{M}_1(\Sigma)$, we have
\[ -\inf_{\nu \in A^o} h(\nu \vert \mu) \leq \liminf_{n \to \infty} \frac{1}{n} \log \mathbb P_{\mu}(L_n^{\mathbf Y} \in A) \leq \limsup_{n \to \infty} \frac{1}{n} \log \mathbb P_{\mu}(L_n^{\mathbf Y} \in A) \leq -\inf_{\nu \in A} h(\nu \vert \mu). \]
\end{thm}
Note that in the case of alphabets, it is not necessary to apply the closure operator in the upper bound.

We continue with a large deviation result corresponding to empirical means of $\mathbb R$-valued i.i.d. random variables, which is therefore called level-1 LD result. Let $\lbrace X_i \rbrace_{i \in \mathbb N}$ be a sequence of i.i.d. random variables with common law $\mu \in \mathcal{M}_1(\mathbb R^d)$. Let us write $\tilde{S}_n=\frac{1}{n} \sum_{i=1}^{n} X_i$ and $\nu_n$ for the law of $\tilde{S}_n$, for $n \in \mathbb N$. Moreover, let us define the \emph{logarithmic moment generating function} associated with the law $\mu$ as
\[ \Lambda_{X_1}(\lambda)=\log \mathbb E \left[e^{\lambda X_1} \right]. \] We define the \emph{Fenchel--Legendre transform} of $\Lambda(\lambda)$ as
\[ \Lambda_{X_1}^\ast(x)=\sup_{\lambda \in \mathbb R} \lbrace \lambda x-\Lambda(\lambda) \rbrace. \] Then the following large deviation theorem holds.
\begin{thm}[Cramér's theorem] \label{nagycramér} \index{Cramér's theorem}
$\Lambda_{X_1}$ is a convex function and $\Lambda_{X_1}^\ast$ is a convex rate function. The sequence of measures $\lbrace \mu_n \rbrace$ satisfies the LDP with the convex rate function $\Lambda^\ast(\cdot)$, i.e.
\begin{enumerate}[(i)]
    \item For all closed set $F \subset \mathbb R$,
    \[ \limsup_{n \to \infty} \frac{1}{n} \log \mu_n(F) \leq -\inf_{x \in F} \Lambda_{X_1}^{\ast}(x). \]
    \item For all open set $G \subset \mathbb R$,
    \[ \liminf_{n \to \infty} \frac{1}{n} \log \mu_n(G) \geq -\inf_{x \in G} \Lambda_{X_1}^{\ast}(x). \]
\end{enumerate}
If $\lbrace \lambda: \Lambda(\lambda)<\infty \rbrace=\lbrace 0 \rbrace$, then $\Lambda^\ast$ is identically zero. If $\Lambda(\lambda)<\infty$ for some $\lambda>0$, then $\mathbb E[X_1]$ exists as an element of $[-\infty,\infty)$, and for all $x \geq \mathbb E[X_1]$ we have
\[ \Lambda_{X_1}^\ast(x)=\sup_{\lambda \geq 0} \lbrace \lambda x-\Lambda(\lambda) \rbrace.\]
\end{thm}

The following notion of regular topological spaces will be important when we consider the LDP of probability measures on function spaces in Chapter \ref{eleje}. We note that all metric spaces are regular. 
\begin{defi} \index{regular topological space}
A Hausdorff topological space $X$ is called regular if for all closed sets $F \subset X$ and points $x \in X \setminus F$ there exist disjoint open sets $G,~H$ such that $F \subseteq G$ and $x \in H$.\footnote{ Note that assuming the Hausdorff condition is redundant here. It is sufficient to assume that the space is $T_1$, i.e., all singletons $\lbrace x \rbrace \subseteq X$ are closed. This, together with the separation of closed sets and single points, implies the Hausdorff condition.}
\end{defi}
\begin{lem}
A family of probability measures $\lbrace \mu_{\varepsilon} \rbrace$ on a regular topological space $X$ can have at most one rate function associated with its LDP.
\end{lem}
Knowing the existence of an LDP on some topological space, the following theorem provides the existence of LDP of continuous transformations of the corresponding probability measures.
\begin{thm}[Contraction principle] \label{contractionprinciple} \index{contraction principle}
Let $X$ and $Y$ be Hausdorff spaces and $f: X \to Y$ a continuous function. Consider a good rate function $I: X \to [0,\infty]$. For each $y \in Y$, define
\[ I'(y)=\inf \lbrace I(x) \vert ~ x \in X,~ y=f(x) \rbrace,\]
where again the infimum taken over $\emptyset$ is defined as $\infty$. Then we have:
\begin{enumerate}[(i)]
\item $I'$ is a good rate function on $Y$.
\item If the family of probability measures $\lbrace \mu_\varepsilon \rbrace$ satisfies an LDP on $X$ with good rate function $I$, then $\lbrace \mu_\varepsilon \circ f^{-1}\rbrace$ satisfies an LDP on $Y$ with good rate function $I'$.
\end{enumerate}
\end{thm}
Now, we state Varadhan's lemmas, which we will directly use in Section \ref{kettőkettő}.
Let $X$ be a regular topological space, $(Z_\varepsilon)$ be a family of random variables with values in $X$ defined on a probability space $(\Omega,\mathcal F,\mathbb P)$, with $\lbrace \mu_\varepsilon \rbrace= \lbrace \mathbb P \circ Z_\varepsilon^{-1} \rbrace$.
\begin{thm}[Varadhan] 
Suppose that $\lbrace \mu_\varepsilon \rbrace$ satisfies an LDP with a good rate function $I: X \to [0,\infty]$, and let $\Phi: X \to \mathbb R$ be any continuous function. Assume further either the tail condition
\begin{equation} \label{farkacska}  \lim_{M \to \infty} \limsup_{\varepsilon \to 0} \varepsilon \log \mathbb E \left( \exp(\Phi(Z_\varepsilon)/\varepsilon) \mathds 1 \lbrace \Phi(Z_\varepsilon) \geq M \rbrace \right)=-\infty, \end{equation}
or the following moment condition for some $\gamma>0$,
\begin{equation} \label{momentecske} \limsup_{\varepsilon \to 0} \varepsilon \log \mathbb E \left( \exp(\gamma \Phi(Z_\varepsilon) / \varepsilon )\right)<\infty. \end{equation}
Then \[ \lim_{\varepsilon \to 0} \varepsilon \log \mathbb E \left( \exp(\Phi(Z_\varepsilon)/\varepsilon) \right) = \sup_{x \in X} \lbrace \Phi(x)-I(x) \rbrace. \]
\end{thm}
\begin{lem}
The moment condition \eqref{momentecske} implies the tail condition \eqref{farkacska}.
\end{lem}
\begin{lem}[Varadhan's lemma, lower bound] \index{Varadhan's lemmas} \label{varadhanlower} If $\Phi: X \to \mathbb R$ is lower semicontinuous and the large deviations lower bound holds with rate function $I: X \to [0,\infty]$, then
\[ \liminf_{\varepsilon \to 0} \varepsilon \log \mathbb E \left( \exp(\Phi(Z_\varepsilon)/\varepsilon) \right) \geq \sup_{x \in X} \lbrace \Phi(x)-I(x) \rbrace. \] \end{lem}
\begin{lem}[Varadhan's lemma, upper bound] \label{varadhanupper} If $\Phi: X \to \mathbb R$ is upper semicontinuous for which the tail condition \eqref{farkacska} holds, and the large deviations upper bound holds with the good rate function $I: X \to [0,\infty]$, then
\[ \limsup_{\varepsilon \to 0} \varepsilon \log \mathbb E \left( \exp(\Phi(Z_\varepsilon)/\varepsilon) \right) \leq \sup_{x \in X} \lbrace \Phi(x)-I(x) \rbrace. \] \end{lem}
The following result \cite[Exercise 4.2.7]{LDP} will be used to derive large deviation principles for sums of independent random variables in this thesis, similarly as in \cite{cikk}.
\begin{prop} \label{4.2.7}
Let $ \lbrace X_\varepsilon \rbrace$ and $ \lbrace Y_\varepsilon \rbrace$ be two collections of random variables defined on a probability space $(\Omega, \mathcal F, \mathbb P)$ with values in the separable, regular topological space $X$ equipped with its Borel $\sigma$-algebra $\mathcal B_X$ such that for all $\varepsilon>0$, $X_\varepsilon$ is independent of $Y_\varepsilon$. Suppose that $\lbrace \mu_\varepsilon \rbrace=\lbrace \mathbf P \circ X_\varepsilon^{-1} \rbrace $ satisfies an LDP with the good rate function $I_X(\cdot)$ and $\lbrace \nu_\varepsilon \rbrace=\lbrace \mathbf P \circ Y_\varepsilon^{-1} \rbrace $ satisfies an LDP with the good rate function $I_Y(\cdot)$, moreover both $\lbrace \mu_\varepsilon \rbrace$ and $\lbrace \nu_\varepsilon \rbrace$ are exponentially tight. Then, for any continuous $F: X \times X \to X$, the family of laws $\lbrace \mathbb P \circ (F(X_\varepsilon, Y_\varepsilon))^{-1} \rbrace$ satisfies an LDP with good rate function \[ I(z)=\inf_{\lbrace (x,y) :~ z=F(x,y) \rbrace} I_X(x)+I_Y(y). \]
\end{prop}
The general \emph{level-2} large deviation principles correspond to the empirical measures of sequences of i.i.d. random variables. To be more precise, let $\Sigma$ be a Polish space and $\lbrace Y_1,Y_2,\ldots \rbrace$ be a sequence of independent, $\Sigma$-valued random variables, identically distributed according to the same law $\mu \in \mathcal{M}_1(\Sigma)$. Then the \emph{empirical law} of $(Y_1,\ldots,Y_n)$ is given as
\[ L_n^{\mathbf Y}=\sum_{i=1}^n \delta_{Y_i} \in \mathcal{M}_1(\Sigma). \]
Let us write $M(\Sigma)$ for the set of all finite signed measures on $\Sigma$. Note that by a standard measure-theoretic result, namely the Hahn decomposition theorem, for all $\nu \in M(\Sigma)$, $\exists! ~\nu_+,~\nu_- \in \mathcal{M}(\Sigma)$ such that $\nu=\nu_+-\nu_-$. Then the \emph{weak topology}\index{weak topology} on $M(\Sigma)$ is generated by the collection
\[ U_{\phi,x,\delta}=\lbrace \nu \in M(\Sigma) \vert~\vert \langle \phi, \nu \rangle-x \vert < \delta \rbrace, \] using the conventional notation $\langle \phi, \nu \rangle=\int_\Sigma \phi \d \nu$, where $x \in \mathbb R$, $\delta>0$, $\phi \in B(\Sigma)$, and $B(\Sigma)$ is the set of all bounded measurable functions on $\Sigma$. The convergence w.r.t. the restriction of this topology to $\mathcal{M}_1(\Sigma)$ is equivalent to the weak convergence of probability measures\index{weak convergence of measures}. The level-2 LDP is established in a topology on $\mathcal{M}_1(\Sigma)$ that is finer than the weak topology. This is called the $\tau$-topology, which is generated by the collection
\[ W_{\phi,x,\delta}=\lbrace \nu \in \mathcal{M}_1(\Sigma) \vert~\vert \langle \phi, \nu \rangle-x \vert < \delta \rbrace, \]
where again $x \in \mathbb R$, $\delta>0$ and $\phi \in B(\Sigma)$. Further, for $\phi \in B(\Sigma)$, we define $p_\phi: \mathcal{M}_1(\Sigma) \to \mathbb R$ as $p_\phi(\nu)=\langle \phi, \nu \rangle$. The cylinder $\sigma$-algebra on $\mathcal{M}_1(\Sigma)$, denoted by $\mathcal{B}^{cy}$, is defined as the $\sigma$-algebra generated by $( p_\phi)_{\phi \in B(\Sigma)}$. Furthermore, for a finite Borel measure $\nu \in \mathcal{M}(\Sigma)$, we define the relative entropy of $\nu$ with respect to $\mu \in \mathcal{M}(\Sigma)$ as follows
\begin{equation} \label{relatíventrópia}\index{relative entropy}  h(\nu \vert \mu)=\begin{cases} \int_\Sigma f \log f \d \mu-\nu(\Sigma)+\mu(\Sigma) \quad \text{ if } f=\frac{\d \nu}{\d \mu} \text{ exists } \\ \infty \quad \text{ otherwise. }\end{cases} \end{equation} Using these notations, we state Sanov's theorem, the level-2 LDP result, which we use in Section \ref{kettőkettő}.
\begin{thm}[Sanov's theorem] \label{nagysanov}\index{Sanov's theorem}
The empirical measures $L_n^{\mathbf Y}$ satisfy an LDP on the measure space $(\mathcal{M}_1(\Sigma), \mathcal B^{cy})$ with respect to the $\tau$-topology with the good, convex rate function $h(\cdot \vert \mu)$.
\end{thm}
Hence, the contraction principle\index{contraction principle} (Theorem \ref{contractionprinciple}) implies that Sanov's theorem is also true w.r.t. the weak topology instead of the $\tau$-topology.
\section{Poisson point processes} \label{kingmanke}
The next main source of prerequisite for this thesis is the theory of Poisson point processes. In order to introduce this theory, in this section we follow the book of Kingman \cite{kingman}. However, sometimes in our setting it will be useful to think about the Poisson point processes as random measures, especially when the domain of the process is an infinite-dimensional path space, e.g. in the setting of \cite{cikk}. For this purpose, we start with the notion of \emph{Poisson random measures}, according to the definitions in \cite[Section 2.2]{lévyprocesses}, which corresponds to a slightly more general setting than Kingman's definition of Poisson process. In this section we omit the proofs, but our Chapter \ref{eleje} contains some examples of typical calculations on Poisson processes.
\begin{defi} \label{PRM}\index{Poisson random measure}
Let $(S,\mathcal S, \nu)$ be an arbitrary $\sigma$-finite measure space, and $(\Omega, \mathcal{F}, \mathbb P)$ a probability space. Let $N: \mathcal S \to \lbrace 0,1,2,\ldots \rbrace \cup \lbrace \infty \rbrace$ be such that the family $\lbrace N(A): A \in \mathcal{S} \rbrace$ are random variables defined on $(\Omega, \mathcal{F}, \mathbb P)$. Then $N$ is called a \emph{Poisson random measure} (PRM) on $S$ with intensity $\nu$ if 
\begin{enumerate}[(i)]
\item for any $n \in \mathbb N$, for mutually disjoint $A_1, A_2,\ldots,A_n \in \mathcal{S}$ the random variables $N(A_1),\ldots,N(A_n)$ are independent,
\item for any $A \in \mathcal{S}$, $N(A)$ is Poisson distributed with parameter $\nu(S) \in [0,\infty]$, where a Poisson(0) distributed random variable is defined to be almost surely equal to 0, and a Poisson($\infty$)-distributed random variable is defined to be almost surely equal to $\infty$, 
\item $\mathbb P$-almost surely, $N$ is a measure.
\end{enumerate}
\end{defi}
Now it is easy to see that if for some $s \in S$ we have that $\nu(\lbrace s \rbrace)=\varepsilon>0$, then 
\[ \mathbb P(N(\lbrace x \rbrace) \geq 2) = 1-\mathrm{e}^{-\varepsilon}-\varepsilon \mathrm{e}^{-\varepsilon}>0, \]
while if $\nu$ is non-atomic, i.e., if for all $s \in S$ we have $\nu(\lbrace s \rbrace)=0$, then for all $s \in S$ we have that $\mathbb P(N(\lbrace x \rbrace) \geq 2) \leq \mathbb P(N(\lbrace x \rbrace) \geq 1)=0$. 

Now, we consider the setting when $\nu$ is non-atomic, and we assume that $\mathcal S$ is large enough to distinguish individual points, which can be ensured by assuming that the diagonal
\[ D=\lbrace (x,x) \vert ~ x \in S \rbrace \subseteq S \times S \] is a measurable subset of $S \times S$, i.e., $D \in \mathcal{S} \tensor \mathcal{S}$. In this case we can think about the Poisson random measure $N$ as a random set of points. That is, there exists a countable random subset $\Pi=\lbrace X_i \vert~ i=1,~\ldots,~N(S) \rbrace \subseteq S$ such that for all $A \in \mathcal S$ we have
\begin{equation} \label{prmppp} N(A)=\sharp \lbrace \Pi \cap A \rbrace, \end{equation} where throughout this thesis, $\sharp$ denotes cardinality of a countable set, writing $\infty$ for countably infinite. Indeed, the following theorem \cite[p.~23]{kingman} provides the existence in this case:
\begin{thm}[Existence Theorem]\index{Existence Theorem}
Let $(S,\mathcal S, \nu)$ and $D$ be as in Definition \ref{PRM}. If $\nu$ is non-atomic and $D \in \mathcal S \tensor \mathcal S$, then there exists a PRM on $S$ with intensity $\nu$.
\end{thm}
In this case, according to \cite[p.~11]{kingman} we call $\Pi$ from \eqref{prmppp} a \emph{Poisson process}\index{Poisson process!on general state space} on $S$ with \emph{intensity measure} $\nu$.\footnote{ It is also valid to call the random measure $N$ a Poisson process, but in the wireless communication literature it is more common to consider the random set of points. This is often interpreted as the set of users in the wireless network, see Section \ref{explanation}.} Note that for all $A \in \mathcal S$ we have $\nu(A)=\mathbb E [N(A)]$. Instead of 'mean measure' one can equivalently use 'intensity' or 'intensity measure'.\footnote{ The original definition of Poisson process in \cite[Section 2.1]{kingman} does not include the $\sigma$-finiteness of the mean measure $\nu$, but since the Existence Theorem requires this condition, we include this assumption in the definition of Poisson process in this thesis.}

A \emph{Poisson point process}\index{Poisson point process} is a special case of a Poisson random measure, the case when $S$ is a topological space with Borel $\sigma$-algebra $\mathcal S$, and $\nu$ is a measure on $S$, which is not only $\sigma$-finite but also locally finite, see e.g. \cite[Section 1.1]{stochgeom}. When $S=\mathbb R^d$, this automatically implies that $\nu(K)<\infty$, and therefore $\mathbb P$-a.s. $N(K)<\infty$, for any bounded $K \subset \mathbb R^d$, and this turns out to be the suitable setting in the case of wireless networks for communication with static users (see Section \ref{Anfang}). A Poisson point process on $\mathbb R^d$ is called \emph{homogeneous} if the mean measure $\mu$ has a constant density w.r.t. the $d$-dimensional Lebesgue measure, i.e., $\mu(\d x)=\lambda \d x$ for some $\lambda \geq 0$. A homogeneous Poisson process on $[0,\infty)$ with intensity $\lambda \geq 0$ is just the usual Poisson process with intensity $\lambda$.

Having established these definitions, in the rest of this section we follow \cite{kingman} directly. We describe the behaviour of general Poisson processes under various transformations: restriction, superposition, mapping, colouring (which is also called thinning), marking and random displacement, which will be used later on in this thesis. We also state Campbell's theorem about real-valued sums over Poisson processes.

We start with the Restriction Theorem which implies that the restrictions of a Poisson process to a measurable subset of the state space is still a Poisson process.
\begin{thm}[Restriction Theorem] \label{restriction}\index{Restriction Theorem}
Let $\Pi$ be a Poisson process with mean measure $\mu$ on $S$, and let $S_1 \subseteq S$ be a measurable subset. Then the random countable set
$\Pi_1=\Pi \cap S_1$
can be regarded either as a Poisson process on $S$ with mean measure
$\mu_1(A)=\mu(A \cap S_1)$
or as a Poisson process on $S_1$ with intensity $\mu \vert_{S_1}$, where $\mu \vert_{S_1}$ denotes the restriction of $\mu$ to $S_1$.
\end{thm}
Next, we present the Superposition Theorem, which ensures that the sum of independent Poisson processes on the same space is still a Poisson process.
\begin{thm}[Superposition Theorem]  \label{superposition}\index{Superposition Theorem}
Let $\Pi_1,\Pi_2,\ldots$ be a countable collection of independent Poisson processes on $S$, and let $\mu_n$ denote the mean measure of $\Pi_n$ for all $n \in \mathbb N$. Then their superposition
$\Pi=\sum_{i=1}^{\infty} \Pi_i$
is a Poisson process on $S$ with mean measure $\mu=\sum_{i=1}^{\infty} \mu_i$.\footnote{ Hence, if $\mu$ is also locally finite, then $\Pi$ is a Poisson point process. Similarly, in the following transformation theorems for Poisson processes in this section it is always true that whenever all the conditions of the theorem are satisfied and the image measure is locally finite, the image process will not only be a Poisson process but also a Poisson point process.}
\end{thm}
We continue with the Mapping Theorem, which claims that the Poisson property is preserved under \emph{deterministic} transformation of the points of a Poisson process, in case that the image measure is non-atomic.
\begin{thm}[Mapping Theorem] \label{mapping}\index{Mapping Theorem}
Let $\Pi$ be a Poisson process on $(S,\mathcal S)$ with mean measure $\mu$ and let $(T,\mathcal T)$ be another measurable space. If $f: S \to T$ is a measurable function such that the image measure $\mu^\ast=\mu \circ f^{-1}$ has no atoms, then $f(\Pi)$ is a Poisson process on $T$ with mean measure $\mu^\ast$.
\end{thm}
After deterministic transformations of Poisson processes, we turn to random transformations. A generalization of the Restriction Theorem for the randomized setting is the Colouring Theorem.
\begin{thm}[Colouring Theorem] \label{colouring}\index{Colouring Theorem}
Let $\Pi$ be a Poisson process on $S$ with mean measure $\mu$. Let the points of $\Pi$ be randomly coloured by $k$ colours, the probability that a point receives the $i$th colour being $p_i$ (such that $p_i \geq 0$, $\sum_{i=1}^k p_i=1$), and the colours of different points of $\Pi$ being independent (of one another and of the position of the points). Let $\Pi_i$ the set of the points that have the $i$th colour. Then $\Pi_1,\ldots,\Pi_k$ are independent Poisson processes on $S$, and $\Pi_i$ has mean measure $\mu_i=p_i \mu$ for all $i \in \lbrace 1,\ldots,k \rbrace$.
\end{thm}
The set $\Pi_i$ is often called a \emph{thinning}\index{thinning} of $\Pi$ with survival probability $p_i$, see e.g. \cite[Section 1.3.2]{stochgeom}. The interpretation for this is that one can consider i.i.d. Bernoulli random variables $\lbrace I^i_X \rbrace_{ X \in \Pi}$ with common survival probability $p_i$, and then $\Pi_i$ equals $\lbrace X \in \Pi \vert I_X^i = 1 \rbrace$ in distribution, for all $i=1,\ldots,k$.

Generalizing the idea of colouring, we set up the notation of \emph{marked Poisson processes}\index{Poisson point process!marked}. Let $\Pi$ be a Poisson process on $S$ with mean measure $\mu$. Suppose that to each point $X$ of the random countable set $\Pi$, we associate a random variable $m_X$ (the \emph{mark} of $X$) taking values in some measurable space $(M, \mathcal{M})$. The distribution of $m_X$ may depend on the realization of $X$ but not on the other points of $\Pi$, and the marks of different points are independent. More precisely, there exists a probability kernel\index{probability kernel} $p(\cdot, \cdot)$ such that for all $X \in \Pi$ and $x \in S$, the conditional distribution of the mark $m_X$ knowing $X=x$ is given by $p(x,\cdot)$. For an example, the colours of $m_X$ of $X \in \Pi$ from the Colouring Theorem are marks taking values in $\lbrace 1,\ldots,k \rbrace$. In this case, the marks of any two points of $\Pi$ are identically distributed.

In the general case, if $m_X$ denotes the mark of $X \in \Pi$, then the pair $(X,m_X)$ can be considered as a random point in the product space $S \times M$. The totality of points $X^\ast$ forms a random countable subset
\begin{equation} \label{márkázás} \Pi^\ast= \lbrace (X, m_X) \vert~X \in \Pi \rbrace \end{equation}
of $S \times M$. The Marking Theorem below states that $\Pi^\ast$ is a Poisson process on the product space $S \times M$.


\begin{thm}[Marking Theorem] \label{marking}\index{Marking Theorem}
The random subset $\Pi^\ast$ is a Poisson process on $S \times M$ with mean measure $\mu^\ast$ given by
\begin{equation} \label{márkamérték} \mu^\ast(C)=\iint\limits_{(x,m) \in C} \mu(\d x) p(x,\d m). \end{equation}
\end{thm}
Now let us consider the special case when $S$ is a Polish space, $M=\mathbb R$ and $p(x,\cdot)=\zeta(\cdot)$ for some fixed probability distribution $\zeta$ on $\mathbb R$, i.e. conditional on the realization of $\Pi$ the marks are i.i.d.\index{Poisson point process!marked!with i.i.d. marks}\index{Poisson point process!marked!separable construction} Then, since $\Pi$ is a.s. countable, it is sufficient to have countably many independent instances of $\zeta$-distributed random variables in order to construct a product probability space where the marked Poisson process $\Pi^\ast$ is defined. In particular, if the points of $\Pi$ are located in a Polish space $S$, and also the marks take values in a Polish space $M$, then also $\Pi^\ast$ can be defined on a separable space. Indeed, assume that the Poisson random measure $N(\cdot)=\sharp \lbrace \Pi \cap \cdot \rbrace$ is defined on the separable probability space $(\Omega_0, \mathcal F_0, \mathbb P_0)$, with $\mathbb P_0 \circ N^{-1}(\cdot)$ distributed as Poisson($\mu(\cdot)$) satisfying the PRM conditions. Let $(\Omega,\mathcal F,\mathbb P)$ be a separable probability space on which a $\zeta$-distributed real-valued random variable is defined. Then, by the Ionescu-Tulcea theorem (see \cite[p.~31--32]{ionescutulcea}), we have that an i.i.d. sequence of $\zeta$-distributed random variables is defined on the probability space $(\Omega_1,\mathcal F_1,\mathbb P_1)=(\Omega^{\mathbb N}, \mathcal{F}^{\tensor \mathbb N}, \mathbb P^{\tensor \mathbb N})$. Now consider the product probability space \[ (\Omega_2,\mathcal{F}_2,\mathbb P_2)=(\Omega_0 \times \Omega_1, \mathcal{F}_0 \tensor \mathcal{F}_1, \mathbb P_0 \tensor \mathbb P_1), \] which is still a separable space, but it is large enough to determine the distribution of the marked Poisson process, by writing 
\[ \mathbb P_2(A, B_1 \times B_2 \times \ldots)=\sum_{n \in \mathbb N} \underbrace{\mathbb P_0(N(A)=n)}_{=\mu(A)} \prod_{k=1}^n \zeta(B_k),\quad \forall A \in \mathcal S,~\forall B_1,B_2,\ldots \in \mathcal{F} \] 
for the generators of the $\sigma$-algebra $\mathcal{F}_0 \tensor \mathcal{F}_1^{\tensor \mathbb N}$. In particular, this construction shows that marked Poisson can be seen as a process taking values on $(S \times M)^\mathbb N=\lbrace a=\lbrace a_n \rbrace: \mathbb N \to S \times M \rbrace$, which is separable.\footnote{ In more detail, under the event $\sharp \Pi < \infty$, we can interpret the random finite subset $\Pi^\ast=\lbrace (s_1,m_1),\ldots,(s_n,m_n) \rbrace$ as an infinite sequence $\Pi^\ast=\lbrace (s_1,m_1),\ldots,(s_n,m_n), (s_n,m_n),(s_n,m_n),\ldots \rbrace$. This interpretation yields a surjective map $\phi$ from $(S \times M)^{\mathbb N} \cup (\bigcup_{n=1}^\infty (S \times M)^n)$ to $(S \times M)^\mathbb N$. 
The law of $\Pi^\ast$ under $\mathbb P_2$, which is the same as the law of $\phi(\Pi^\ast)$ under $\mathbb P_2 \circ \phi^{-1}$ and the latter process takes values in the Polish space $(S \times M)^\mathbb N$.}

Probability kernels also play an important rôle in the Displacement Theorem, which is a generalization of the Mapping Theorem to the case when the displacement of the points of a Poisson process is random. This theorem requires the state space $S$ of the Poisson process $\Pi$ to be equal to $\mathbb R^d$ for some $d \geq 1$. 
\begin{thm}[Displacement Theorem] \label{displacement}\index{Displacement Theorem}
Let $\Pi$ be a Poisson process on $\mathbb R^d$ with intensity function $\lambda(\cdot)$. Assume that the points of $\Pi$ are randomly displaced, in such a way that the displacements of different points are independent, and suppose that the distribution of the displaced position of a point $\Pi$ at $X=x$ has probability density $\varrho(x,\cdot)$. Then the displaced points form a Poisson process $\Pi'$ with rate function $\lambda'$ given by
\[ \lambda'(y)=\int_{\mathbb R^d} \lambda(x) \varrho(x, y) \mathrm d x. \]
In particular, if $\lambda(x)$ is constant $\lambda$ and $\varrho(x,y)$ is a function of $x-y$, then $\lambda'(y)=\lambda$ for all $y$.  
\end{thm}

Campbell's theorem indicates the way how sums (integrals) over Poisson processes can be computed. Before stating the theorem, we recall that a Poisson random variable $X$ with parameter $\lambda$ has characteristic function \[ \mathbb E[\exp(i \theta X)]=\exp \left( \lambda (e^{i \theta}-1) \right) .\]
\begin{thm}[Campbell's theorem] \label{campbell}\index{Campbell's theorem} Let $\Pi$ be a Poisson process on $(S,\mathcal S)$ with intensity $\mu$, and let $f: ~ S \to \mathbb R$ be measurable. Then the sum
\[ \Sigma=\sum_{X \in \Pi} f(X) \]
is absolutely convergent with probability 1 if and only if \[ \int_S \min (\vert f(x) \vert, 1) \mu(\d x) < \infty. \] If this condition holds, then we have
\[ \mathbb E[\mathrm{exp}(\theta \Sigma)]=\exp \left(\int_S e^{\theta f(x)}-1 \mu(\d x) \right)  \]
for any complex $\theta$ for which the right hand side converges, in particular always if $\theta$ is purely imaginary. Moreover, 
\begin{equation} \label{Windows3.11} \mathbb E(\Sigma)=\int_S f(s) \mu(\d s), \end{equation}
in the sense that the expectation exists if and only if the integral converges, and then they are equal. If \eqref{Windows3.11} converges, then \[ \mathrm{Var}(X)=\int_S f^2(s) \mu(\d s), \]
finite or infinite.
\end{thm}
Finally, we present a simple but very useful result about the homogeneous Poisson process on the positive real half-line, which will also appear in Section \ref{ocsú}.
\begin{thm}[Law of Large Numbers for Poisson Processes] \label{poissonlln}\index{Poisson law of large numbers}
Let $\Pi$ be a Poisson process of constant rate $\lambda \geq 0$ on $(0,\infty)$. Then the number $N(t)$ of points of $\Pi$ in $(0,t]$ satisfies
\[ \lim_{t \to \infty} \frac{N(t)}{t}=\lambda \]
almost surely.
\end{thm}
\section{Main results of \cite{cikk}} \label{explanation}\index{fading!-free case}
According to the Introduction, our goal in this thesis is to generalize the approach of the paper \cite{cikk} into a setting when the users have random fadings. In this section, we summarize the model definition and the main results from \cite[Section 1]{cikk}. We will see that this original model allows the users to move in the communication area along time. However, we omit mobility from our extended model defined in Section \ref{Anfang}, because it is usually hard to handle concrete examples where the effect of mobility can be detected, both analytically and numerically, even in the fading-free case, cf. \cite[Section 7]{cikk}. In this section we also prove Lemma \ref{deltadiscretization} about spatial discretization, which has been tacitly used in \cite{cikk}.

The paper investigates the effect of \emph{relaying} in a compact communication area in the asymptotic setting of a high density of mobile users. The users are given by a Poisson point process\index{Poisson point process} $X^\lambda$ of trajectories with intensity function $\lambda \overline{\mu}(\cdot)$, where $\lambda>0$. We assume that the distribution of the initial points of trajectories is absolutely continuous w.r.t. the Lebesgue measure. Moreover, $\overline{\mu}$ is assumed to be a finite Borel measure on the set of Lipschitz continuous trajectories $\mathcal{L}=\mathcal{L}_{J_1}(I,W)$, with Lipschitz parameter $J_1$, from the interval $I=[0,T)$ to the window $W=[-r,r]^d \subseteq \mathbb R^d$ for some integer $r \geq 1$. 

For the network model, the paper follows a classical approach based on the \emph{signal-to-interference ratio} (SIR). Here $\ell: (0,\infty) \to [0,\infty)$ denotes the \emph{path-loss function}\index{path-loss!function}, which is a Lipschitz continuous function with parameter $J_2$, which describes the decay of the signal strength over distance. Additionally, the ability of a receiver to decode a message is reduced by interference generated by \emph{all} users. 

Formally, we let $X_{i,t}$ denote the $i$th trajectory in $X^\lambda$ at time $t$, and introduce the empirical measures\index{empirical measure}
\[ L_{\lambda}=\frac{1}{\lambda} \sum_{X_i \in X^\lambda} \delta_{X_i}, \quad \and \quad L_{\lambda,t}=\frac{1}{\lambda} \sum_{X_i \in X^\lambda} \delta_{X_{i,t}} \]
respectively as a random element of $\mathcal{M}(\mathcal{L})$, the set of finite Borel measures on $\mathcal{L}$, and in $\mathcal{M}(W)$, the space of finite Borel measures on $W$. As in Section \ref{ldpcske}, let $\mathcal{M}(X)$ be the set of finite Borel measures of the topological space $X$. 
Now, we define the SIR of a message transmitted by $\xi \in W$ and measured, at the same time $t \in I$, at a receiver at $\eta \in W$ as follows
\[ \SIR_\lambda(\xi,~\eta,~ L_{\lambda,t})=\frac{\ell(\vert \xi-\eta \vert)}{\sum_{X_i \in X^\lambda} \ell(\vert X_i - \eta \vert)}. \]\index{SIR}
The denominator of this fraction is called the \emph{interference}\index{interference} at $\eta$. Note that the interference contains even the signal strength coming from the corresponding transmitter, and that the model does not include noise. The reasons for these decisions are explained in Footnote \ref{STIR} on page \pageref{STIR}. 

Then, a connection between\index{connected and disconnected users} $\xi, \eta \in W$ is useful only if
$\SIR_{\lambda}(\xi,~\eta,~L_{\lambda,t}) \geq c$ with some preliminarily given value $c$.
In particular, if the threshold $c$ is given by $c=\lambda^{-1} c'$, then the above requirement can be re-expressed as $\SIR(\xi,~\eta,~L_{\lambda,t}) \geq c')$, where $\SIR(\xi,~\eta,~L_{\lambda,t})=\lambda \SIR_\lambda (\xi,~\eta,~L_{\lambda,t})$.

The paper conducts level-2 large deviation analysis of certain frustration events. In particular, it shows that the most likely option for a rare event to occur can be described by a certain finite Borel measure $\nu \in \mathcal{W}$ that describes the asymptotic configuration of users under conditioning on the rare event. Aiming to prove this, one extends the definition of SIR to arbitrary $\nu \in \mathcal{M}(W)$ and also writes
\[ \SIR(\xi,~\eta,~\nu)=\frac{\ell(\vert \xi-\eta \vert)}{\int_W \ell(x -\eta \vert) \nu(\d x)} \]
for any $\xi,\eta \in W$. In order to keep the model flexible, it is assumed that the \emph{quality of service} (QoS)\index{quality of service (QoS)} of the direct link between $\xi$ and $\eta$ is given by
\[D(\xi,~\eta,~L_{\lambda,t})=g(\SIR(\xi,~\eta,~L_\lambda)),\]
where $g: [0,\infty) \to [0,\infty)$ is a Lipschitz continuous function that is strictly increasing on $[0,\varrho_+)$ and constant equal to $c_+$ on $[\varrho_+,\infty)$ for some $\varrho_+, c_+ >0$. One can easily see that $D(\xi,~\eta,~L_{\lambda,t})=c_+$ if $L_{\lambda,t}(W) \leq \beta_o=\min \lbrace 1, \varrho_+^{-1} \ell_{\min} \ell_{\max}^{-1} \rbrace$, where $\ell_{\min}=\min_{\xi,\eta \in W} \ell(\vert \xi-\eta \vert)$, $\ell_{\max}=\max_{\xi,\eta\in W} \ell(\vert \xi-\eta \vert)$. As SIR, also $D$ is defined for general $\nu \in \mathcal{M}(W)$ via $D(\xi,~\eta,~\nu)=g(\SIR(\xi,~\eta,~\nu))$, and $D(\xi,~\eta,~\nu)=c_+$ if $\nu(W)=0$. For instance, possible choices of $g$ include $g(r)=\min \lbrace r, K \rbrace$ or $g(r)=\min \lbrace \log(1+r), K \rbrace$ for some fixed $K>0$.

If a message is sent out from a user $\xi$ to a user $\eta$ routing via a relay $\zeta$, then the quality of the relayed message depends on both the SIR from $\xi$ to $\zeta$ and the SIR from $\zeta$ to $\eta$. The assumption of the paper is that message transmissions are successful if the SIR of both links are above the certain threshold. In other words, the assumption is that the QoS when relaying from $\xi$ to $\eta$ via $\zeta$ can be expressed as
\[ \Gamma(\xi,~\zeta,~\eta,~L_{\lambda,t})=\min \left\lbrace D(\xi,~\zeta,~L_{\lambda,t}), D(\zeta,~\eta,~L_{\lambda,t}) \right\rbrace. \]
In the following we introduce several characteristics that describe the QoS in a relay setting. In the uplink scenario\index{means of communication}\index{means of communication!uplink!direct}, the destination of a message sent from $X_i \in X^\lambda$ via a relay $X_j \in X^\lambda$ is the origin $o$. Under an optimum relay decision, the QoS for the relayed uplink\index{means of communication!uplink!relayed} communication is defined as
\begin{equation} \label{egy} R(X_{i,t},~o,~L_{\lambda,t})=\max \lbrace D(X_{i,t},~o,~L_{\lambda,t}), \max_{X_{j,t} \in X^\lambda_t} \Gamma(X_{i,t},~X_{j,t},~o,~L_{\lambda,t}) \rbrace. \end{equation}
In other words, in \eqref{egy} the user has the possibility to connect to the base station also directly, but if there exists any user such that relaying via this user offers a better connection, then relaying leads to a higher QoS. Similarly, in the downlink scenario, when messages are sent out from $o$ to a user $X_i \in X^\lambda$, and relaying is again possible, the QoS for the relayed downlink communication\index{means of communication!downlink!direct}\index{means of communication!downlink!relayed} can be expressed as
\[ R(o,~X_{i,t},~L_{\lambda,t})=\max \lbrace D(o,~X_{i,t},~L_{\lambda,t}), \max_{X_{j,t} \in X^\lambda_t} \Gamma(o,~X_{j,t},X_{i,t},~L_{\lambda,t}) \rbrace. \] 
Extending the definition of $R$ further to arbitrary finite Borel measures $\nu \in \mathcal{M}(W)$ we write
\[ R(\xi,~\eta,~\nu)=\max \lbrace D(\xi,~\eta,~\nu),~\nu\text{-}\esssup_{\zeta \in W} \Gamma(\xi,~\zeta,~\eta,~\nu) \rbrace\]
for any given $\xi,~\eta \in W$. Here $\nu\text{-}\esssup$ means essential supremum w.r.t. $\nu$. 

Let $\pi_t: \mathcal{L} \to W$, $x \mapsto x_t$ denote the projection at time $t \in I$. Then for $\overline{\nu} \in \mathcal{M}(\mathcal{L})$ we write $\overline{\nu}_t=\overline{\nu} \circ \pi_t^{-1}$, and for $x,y \in \mathcal{L}$ we define the trajectory of QoS
\[ \overline{\SIR}(x,~y,~\overline{\nu})=(\overline{\SIR}(x_t,~y_t,~\overline{\nu}_t))_{t \in I}, \] 
and similarly $\overline{D}$ and $\overline{R}$. Then $\overline{\SIR}$, $\overline{D}$ and $\overline{R}$ are elements of the space of bounded measurable functions $\mathcal{B}=\mathcal{B}(I, [0,\infty))$ equipped with the supremum norm and the associated Borel $\sigma$-algebra.

The object of interest of the paper is the point process of users $X_i \in X^\lambda$ who are \emph{frustrated}\index{frustration events} due to unsatisfactory $\SIR$. In order to describe the number of frustrated users, we define the following rescaled measure for the uplink
\[ L_{\lambda}^{\text{up}} [\tau]=\frac{1}{\lambda} \sum_{X_j \in X^\lambda} \delta_{X_j} \tau(\overline{R}(X_j,~o,~L_\lambda)), \]
where $\tau: \mathcal{B} \mapsto [0,\infty)$ is a bounded and measurable function. In particular, $L_{\lambda}^{\text{up}} [\tau] \in \mathcal{M}(\mathcal{L})$. More generally, if $\overline{\nu} \in \mathcal{M}(\mathcal{L})$, then $\overline{\nu}^{\text{up}}[\tau]$ is defined as an element of $\mathcal{M}(\mathcal{L})$ via
\[ \frac{\mathrm d \overline{\nu}^{\text{up}}[\tau]}{\d \nu}(x)=\tau(\overline{R}(x,~o,~\overline{\nu})).\]
It is also important to consider those users who have a bad QoS for direct uplink communication separately, because if a large number of users has to communicate via a small number of relays, then communication on full bandwidth cannot be guaranteed. 
For a general $\overline{\nu} \in \mathcal{M}(\mathcal{L})$, the quantity $\overline{\nu}^{\text{up-dir}}[\tau]$ of users who have bad QoS w.r.t. direct communication with the base station is defined via
\[ \frac{\mathrm d \overline{\nu}^{\text{up-dir}}[\tau]}{\d \nu}(x)=\tau(\overline{D}(x,~o,~\overline{\nu})).\]
For the downlink one can now define
\[ \frac{\mathrm d \overline{\nu}^{\text{do}}[\tau]}{\d \nu}(x)=\tau(\overline{R}(o,~x,~\overline{\nu}))\] and analogously for $\overline{\nu}^{\text{do-dir}}[\tau]$. We write $\boldsymbol \tau=(\tau_1,\ldots,\tau_4)$,
\[ L_\lambda[\boldsymbol \tau]=(L_\lambda^{\text{up}}[\tau_1],~L_\lambda^{\text{up-dir}}[\tau_2],~L_\lambda^{\text{do}}[\tau_3],~L_\lambda^{\text{do-dir}}[\tau_4]), \quad \overline{\nu}[\boldsymbol \tau]=(\overline{\nu}^{\text{up}}[\tau_1],~\overline{\nu}^{\text{up-dir}}[\tau_2]),~\overline{\nu}^{\text{do}}[\tau_3],~\overline{\nu}^{\text{do-dir}}[\tau_4]). \]

Let $\mathcal{K}=\mathcal{K}(I,W)$ denote the space of measurable trajectories with values in $W$, equipped with the supremum norm. The objects of interest of the paper are random variables $F(L_\lambda[\boldsymbol \tau])$, where $\tau_i: \mathcal{B} \to [0,\infty),~i \in \four$ is assumed to be decreasing in the sense that for all $\gamma,~\gamma' \in \mathcal{B}$ with $\gamma_t \leq \gamma'_t$ for all $t \in I$, we have $\tau(\gamma) \geq \tau(\gamma')$. Moreover, $F: \mathcal{M}(K)^4 \to [-\infty,\infty)$ is assumed to be increasing in the sense that for all $\overline{\nu},~\overline{\nu}' \in \mathcal{M}(\mathcal{K})$ with $\overline{\nu} \leq \overline{\nu}'$ we have $F(\overline{\nu}) \leq F(\overline{\nu}')$. Here we write $\overline{\nu} \leq \overline{\nu}'$ if $\overline{\nu}(A) \leq \overline{\nu}'(A)$ for all measurable $A \subset \mathcal{K}$. We also put $\overline{\nu} < \overline{\nu}'$ if $\overline{\nu} \leq \overline{\nu}'$ and $\overline{\nu} \neq \overline{\nu}'$.

The most important example is the following: define $F_{\mathbf b}: \mathcal{M}(\mathcal{K})^4 \to [-\infty,\infty)$, 
\begin{equation} \label{kisFfading} (\overline{\nu}_i)_{i \in \lbrace 1,\ldots,4 \rbrace} \mapsto \begin{cases} 0 \quad \quad \text{if } \overline{\nu}_i(\mathcal{K}') > b_i, \forall i \in \lbrace 1,\ldots,4\rbrace \\ -\infty \> \text{ otherwise} \end{cases} \end{equation}
for some $\mathbf b \in \mathbb R^4$, and $\tau_{a,c}: \mathcal{B} \to [0,\infty)$, 
\begin{equation} \label{kistaufading}
\gamma \mapsto \mathds 1 {\lbrace \int_0^T \mathds 1{\lbrace \gamma_t < c \rbrace} \mathrm{d}t>a \rbrace}.
\end{equation}
Then $F_{\mathbf b}$ and $\tau_{a,c}$ are measurable. In particular, for $\boldsymbol \tau_{\mathbf a, \mathbf c}=({\tau}_{a_i,c_i})_{i \in \four}$ we have
\[ \mathbb E (\exp(\lambda F_{\mathbf b} (L_\lambda [\boldsymbol \tau_{\mathbf a, \mathbf c}])))=\mathbb P(L_\lambda[\boldsymbol \tau_{\mathbf a, \mathbf c}](\mathcal L)>\mathbf b), \]
where for vectors $\mathbf a=(a_1,\ldots,a_4),~\mathbf b=(b_1,\ldots,b_4) \in \mathbb R^4$, we write $\mathbf a < \mathbf b$ if $a_i < b_i$ for all $i \in \four$. This describes the probability of the event that more than $\lambda b_i$ users experience a QoS of less than $c_i$ for a period of time longer than $a_i$ for all $i \in \four$.

To prove the frustration results, it is convenient to consider functions $F: \mathcal{M}(K)^4 \to [-\infty,\infty)$ that are compatible with triadic discretizations\index{discretization} of $\mathcal K$. A triadic discretization is choosen to ensure that spatially, the origin is a centre of a sub-cube of $W$, and that $W$ is a union of sub-cubes of the form $\zeta+[-\delta r, \delta r]^d$ with $\zeta \in 2 \delta r \mathbb Z^d$ for some $\delta \in \mathbb B=\lbrace 3^{-m} \vert m \geq 1 \rbrace$. The space and time discretizations are
\[ W_\delta=\delta 2r \mathbb Z^d \cap W \quad \text{and} \quad I_\delta=\delta T \left( \mathbb Z + 1/2 \right) \cap I. \]
The discretized path space is the space $\Pi_\delta=W_\delta^{I_\delta}$ of functions mapping from $I_\delta$ to $W_\delta$. The space of bounded measures $\mathcal{M}(\Pi_\delta)$ can be identified with $[0,\infty)^{\Pi_\delta}$. We define two operations relating $\Pi_\delta=W_\delta^{I_\delta}$ to the continuous path space $\mathcal K$. First, we discretize $x \in \mathcal{K}$ by evaluating $x$ at discrete times $t \in I_\delta$ and spatially moving $x$ to the centres of sub-cubes, denoting the discretized path $\varrho(x) \in \Pi_\delta$ by 
\[ \varrho: \mathcal{K} \rightarrow \Pi_\delta, \quad x \mapsto \lbrace \varrho(x_t) \rbrace_{t \in I_\delta},\]
where $\varrho(x_t)$ denotes the shift of $x \in W$ to the nearest sub-cube centre in $W_\delta$. For $\overline{\nu} \in \mathcal{M}(\mathcal{K})$, the mappings $\varrho$ induce an image measure
$\overline{\nu}^{\varrho}=\overline{\nu} \circ \varrho^{-1} \in \mathcal{M}(\Pi_\delta).$
Second, a discretized path $z \in \Pi_\delta$ can be embedded into $\mathcal K$ as a step function. That is,
\[ \imath: \Pi_\delta \rightarrow \mathcal{K}, \quad u \mapsto \left\lbrace \sum_{i=0}^{\delta^{-1}T-1} \mathds 1 {[i\delta T,(i+1)\delta T]}(t)~u_{(2i+1)\delta T/2~} \right\rbrace_{t \in I_\delta}. \]
Then, for $\overline{\nu} \in \mathcal{M}(\Pi_\delta)$, the mapping $\imath$ induces an image measure 
$\overline{\nu}^{\imath}=\overline{\nu} \times \imath^{-1} \in \mathcal{M}(\mathcal{K}). $

For $\delta \in \mathbb B$, we call a function $F: \mathcal{M}(K)^4 \to [-\infty,\infty)$ is $\delta$-discretized if $F((\overline{\nu})^{\varrho})^{\imath})=F(\overline{\nu})$ holds for all $\overline{\nu} \in \mathcal{M}(\mathcal{K})^4$. E.g., the functions $F_{\mathbf b}$ defined in \eqref{kisFfading} are $\delta$-discretized for all $\delta \in \mathbb B$. The next lemma shows consistence of the triadic discretizations, which makes it useful to discretize $W$ this way. This lemma was tacitly used in \cite{cikk}, and it is easy to see that it remains true if users have random fadings.

\begin{lem} \label{deltadiscretization}\index{discretization!$\delta$-discretized function}
Let $F: \mathcal{M}(\mathcal{K})^4 \to [-\infty,\infty)$ be $\delta$-discretized for some $\delta \in \mathbb B$. Then $F$ is $\delta$-discretized for all $\mathbb B \ni \delta'<\delta$.
\end{lem}
\begin{proof}
Let $\delta', \delta \in \mathbb B$, $\delta'<\delta$, and assume that $F$ is $\delta$-discretized. Then for all $\overline{\nu} \in \mathcal{M}(\mathcal{K})^4$ we have $F((\overline{\nu}^{\varrho_\delta})^{\imath_\delta})=F(\overline{\nu})$. Thus, it suffices to show that $(((\overline{\nu}^{\varrho_{\delta'}})^{\imath_{\delta'}})^{\varrho_{\delta}})^{\imath_{\delta}}=(\overline{\nu}^{\varrho_{\delta}})^{\imath_{\delta}}$ holds for all $\overline{\nu} \in \mathcal{M}(\mathcal{K})^4$, because then we have $F((\overline{\nu}^{\varrho_{\delta'}})^{\imath_{\delta'}})=F((((\overline{\nu}^{\varrho_{\delta'}})^{\imath_{\delta'}})^{\varrho_{\delta}})^{\imath_{\delta}})=F((\overline{\nu}^{\varrho_{\delta}})^{\imath_{\delta}})=F(\overline{\nu})$. In particular, it suffices to prove that $\varrho_\delta(\imath_{\delta'}(\varrho_{\delta'}(\xi)=\varrho_{\delta}(\xi)$ holds for all $\xi \in W$. But this is true because we have $\delta=3^{-m}$ and $\delta'=3^{-n}$ for some $n >m$, and hence the sub-cube centres $x \in W_\delta$ are also contained in $W_{\delta'}$, and thus for all $t \in I_\delta$, the points $\imath_{\delta'}(\varrho_{\delta'} (x_t)=\varrho_{\delta'}(x_t)$ and $x_t$ are in the same sub-cube w.r.t. $\delta$-discretization. 
\end{proof}

The main result \cite[Theorem 1.1]{cikk} is a level-2 large deviation result about the quantities $F(L_\lambda[\boldsymbol \tau])$. Our Theorem \ref{1.1} extends it to settings with random fadings. We define relative entropy $h$ via \eqref{relatíventrópia}\index{relative entropy}. 
\begin{thm} \label{régi 1-es}\index{main results!of \cite{cikk}}
Let $\tau_i: \mathcal{B} \to [0,\infty)$, for $i \in \four$, be bounded, measurable and decreasing functions that map trajectories $\gamma$ to 0 if $\gamma_t \geq {c}_+$ for all $t \in I$. Further, let $F: \mathcal{M}(\mathcal{K})^4 \to [-\infty,\infty) $ be an increasing function that is $\delta$-discretized for some $\delta \in \mathbb B$, bounded from above, and maps the vector of zero measures to $-\infty$. If the $\tau_i \circ \imath$ are u.s.c. as functions on $[0,\infty)^{I_\delta}$ and $\overline{\nu} \mapsto F(\overline{\nu}^{\imath})$ is u.s.c. as a function on $\mathcal{M}(\Pi_\delta)^4$, then we have
\[ \limsup_{\lambda \to \infty} \frac{1}{\lambda} \log \mathbb E \exp (\lambda F(L_\lambda[\boldsymbol \tau])) \leq -\inf\limits_{\overline{\nu} \in \mathcal{M}(\mathcal L)} \left\lbrace h(\overline \nu \vert \overline{\mu}) - F(\overline{\nu}[\boldsymbol \tau]) \right\rbrace, \] 
while if the $\tau_i \circ \imath$ are l.s.c. on $[0,\infty)^{I_\delta}$ and $\overline{\nu} \mapsto F(\overline{\nu}^{\imath})$ is l.s.c. on $\mathcal{M}(\Pi_\delta)^4$, then 
\[ \liminf_{\lambda \to \infty} \frac{1}{\lambda} \log \mathbb E \exp (\lambda F(L_\lambda[\boldsymbol \tau])) \geq -\inf\limits_{\overline{\nu} \in \mathcal{M}(\mathcal L)} \left\lbrace h(\overline \nu \vert \overline{\mu}) - F(\overline{\nu}[\boldsymbol \tau]) \right\rbrace. \] 
\end{thm}
As a special case of Theorem \ref{régi 1-es}, the rate of decay of frustration probabilities\index{frustration probabilities!rate function} $\mathbb P(L_\lambda[\boldsymbol \tau_{\mathbf a, \mathbf c}](\mathcal L)>\mathbf b)$ has been computed, where $\boldsymbol \tau_{\mathbf a, \mathbf c}$ was defined in \eqref{kistaufading}. Let us write $[0,\mathbf T)=[0,T )^4$ and $[0,\mathbf c_+)=[0,c_+)^4$.
\begin{cor}
Let $\mathbf a \in [0,\mathbf T)$, $\mathbf b \in \mathbb R^4$ and $\mathbf c \in [0,\mathbf c_+)$. Then we have
\[ \lim_{\lambda \to \infty} \frac{1}{\lambda} \log \mathbb P (L_{\lambda} [\boldsymbol \tau_{\mathbf{c}}](\mathcal{L})>\mathbf b)= - \inf\limits_{\overline{\nu} \in \mathcal{L}:~\overline{\nu}[\boldsymbol \tau_{\mathbf a, \mathbf c}] (\mathcal{L}) > \mathbf b} h(\overline{\nu} \vert \overline{\mu}). \]
\end{cor}
Finally, the paper concludes that probabilities of frustration events that are unlikely w.r.t. the a priori measure $\mu$ decay at an exponential speed.
\begin{cor} \index{frustration probabilities!exponential decay}
Let $\mathbf a \in [0,\mathbf T),~\mathbf b \in \mathbb R^4,~\mathbf c \in [0,\mathbf c+)$. If $((1+\varepsilon)\overline{\mu})[\boldsymbol \tau_{\mathbf a, \mathbf c}] (\mathcal{L}) \leq \mathbf b$ for some $\varepsilon>0$, then
\[ \limsup_{\lambda \to \infty} \frac{1}{\lambda} \log \mathbb P(L_\lambda[\boldsymbol \tau_{\mathbf a,\mathbf c}](\mathcal L) > \mathbf b ) < 0. \]
\end{cor}
\chapter{Incorporating i.i.d. fading} \label{eleje}
\section{Description of the model with random fadings} \label{Anfang}
In the following, we work out a modified version of the model described in Section \ref{explanation}, adding \emph{random fadings} to the model, interpreted as \emph{loudnesses} of the users.. We introduce a substantial amount of new notation\footnote{ This notation will be consistent with the one introduced in Section \ref{explanation}, with the following quantities redefined. The probability measure $\mathbb P$ and the expectation operator $\mathbb E$ will correspond to the fading variables, not to the Poisson point process $X^\lambda$. Furthermore, random fadings increase the dimension of the model by one. Therefore, the functions $\SIR$, $D$, $\Gamma$ and $R$ will be redefined, having new domains. Moreover, in Section \ref{discretization} we introduce a new notation instead of $\nu [\tau]$ from Section \ref{explanation} for the rescaled measures, in order to emphasize that they depend on the same measure in two ways. Finally, it will be convenient to write $\varrho_1$ instead of $\varrho$ for the spatial discretization operator, cf. Section \ref{discretization}.}, which we summarize in the Appendix of the thesis, in Section \ref{indexofnotations}. We omit mobility from the original model before extending it with random fadings, see the reason for it in the beginning of Section \ref{explanation}. 
\index{mobility of users} 

Hence, we let $X^\lambda$ be a Poisson point process on $W$, with intensity $\lambda \mu$, where $W=[-r,r]^d$ for some integer $r \geq 1$ and $\mu$ is a finite Borel measure on $W$ and $\lambda>0$. We assume that $\mu$ is absolutely continuous with respect to the Lebesgue measure restricted to $W$. However, similarly to the fading-free case, this does not imply that the support of $X^\lambda$ is equal to the whole $W$, it can be as well any closed subset of $W$ with non-empty interior. In particular, for any bounded open subset $U \subset \mathbb R^d$, $\overline{U}$ is a possible support of $X^\lambda$, since such closures are compact and contained in $[-r,r]^d$ for $r \in \mathbb N$ sufficiently large. Similarly to the fading-free setting, we assume that the \emph{path-loss function}\index{path-loss!function} is a Lipschitz continuous function $\ell: [0,\infty) \to (0,\infty)$ with parameter $J_2 \geq 0$. Thus, the propagation of signal strength depends only on the Euclidean distance of points in $W$. Therefore it is reasonable to assume that the straight line segment between two users $X_i,X_j \in X^\lambda$ is a possible route to transmit a message between them. In other words, even if the support of $X^\lambda$, which is the set of possible user locations, is not convex, we may assume that messages can be transferred along straight lines in the whole convex hull of the support of $X^\lambda$. This convex hull is contained in $W$ by construction.

For each user $X_i \in X^\lambda$, we have a real-valued, positive fading variable $F_{X_i}$, which represents the loudness of the user $x$, in such a way that knowing the realization of $X^\lambda$, the fading random variables\index{random fading} $\lbrace F_{X_i} \vert~ X_i \in X^\lambda \rbrace$ are independent and identically distributed (i.i.d.). That is, according to Section \ref{kingmanke}, $\mathbf X^\lambda=\lbrace (X_i, F_{X_i}) \vert~X_i \in X^\lambda \rbrace$ is a marked Poisson point process\index{Poisson point process!marked} with points in $W$ and marks in $(0,\infty)$. By the Marking Theorem\index{Marking Theorem} (Theorem \ref{marking}), $\mathbf X^\lambda$ is a Poisson point process on the enlarged state space $\mathbf W=W \times (0,\infty)$. 

In the scenarios we consider in this thesis, the distribution of the fading variables does not depend on $\lambda$. Thus, we write $F_0$\index{fading!variable} for an arbitrary random variable that equals $F_{X_i}$ in distribution, for any point $X_i \in X^\lambda$ and for any $\lambda>0$. If $F_0$ is defined on the probability space $(\Omega,\mathcal{F},\mathbb{P})$, let $\zeta=\mathbb P \circ F_0^{-1}$ denote the distribution of $F_0$. Using the notation $L_{\lambda}=\frac{1}{\lambda} \sum_{X_i \in X^\lambda} \delta_{X_i}$ for the empirical measure from Section \ref{explanation}, and noting that now $L_{\lambda}$ is a random element of the set of finite measures $\mathcal{M}(W)$ instead of $\mathcal{M}(\mathcal{L})$ since we have no mobility of users, we have that the empirical measure of the marked Poisson point process, given as 
\begin{equation} \label{vastagL}\index{empirical measure} \mathbf L_{\lambda}= \frac{1}{\lambda} \sum_{X_i \in X^\lambda} \delta_{(X_i,F_{X_i})} \end{equation} is a random element of the set of Borel measures of $\mathbf W$. The randomness comes from two sources, one is the realization of the Poisson point process $X^\lambda$, the other is the realization of the fading variables $\lbrace F_{X_i} \vert~X_i \in X^\lambda \rbrace$, knowing $X^\lambda$. According to the Marking Theorem\index{probability kernel}, for all $x \in W$ the conditional distribution of $F_x$ knowing $x \in X^\lambda$ is given by the probability kernel \[ p(x,\cdot)=\zeta(\cdot),~x \in X^\lambda. \] In Section \ref{kernel} we generalize this setting to cases when $p(x,\cdot)$ also depends on $x \in W$.


In order to determine large deviation properties for the marked Poisson point process $\mathbf X^\lambda$ in the high-density limit (i.e., as $\lambda \to \infty$)\index{high-density limit}, we define the measure 
\begin{equation} \label{újmérték} \mu' (\d x, \d u)=\mu(\d x)p(x,\d u)=\mu(\d x) \zeta(\d u), \quad x \in W,~u \in (0,\infty) \end{equation}
which is the intensity of the marked Poisson point process $\lambda \mathbf L^\lambda$, according to Section \ref{kingmanke}. 

We use the following notations for probability measures. $\mathbb E$ means expectation w.r.t. $\mathbb P$, the probability measure which governs $F_0$. Moreover, 
$\mathbb E_2$ means expectation w.r.t. the product probability measure $\mathbb P_2$ given by the marked Poisson point process. Thus, the expression "$\mathbb P_2$-almost surely" is equivalent to "$\tensor\limits_{X_i \in X^\lambda(\omega)} p(X_i,\cdot)$-almost surely for all \emph{fixed realization} of the Poisson point process $X^\lambda$ on $W$". Since both the random countable set $X^\lambda$ is finite for $\mathbb P_2$-almost all realizations of $X^\lambda$ and $F_0$ takes values in $(0,\infty)$, we have that $\mathbf L_{\lambda}$ is an element of $\mathcal{M}(\mathbf W)$ for these realizations of $X^\lambda$, independently of the exact values of $\lbrace F_{X_i} \vert~X_i \in X^\lambda \rbrace$. Note that this notion of $\mathbb P_2$ is in accordance with the separable construction\index{Poisson point process!marked!separable construction} of a marked Poisson process with i.i.d. marks from Section \ref{kingmanke}. In particular, for $\lambda>0$, $\mathbb P_2 \circ (\mathbf X^\lambda)^{-1}$ is the law of the Poisson point process $\mathbf X^\lambda$ on $\mathcal{M}(\mathbf W)$, which endowed with the metric generated by the weak topology is a metric space, hence also a regular topological space. Thus, e.g. Varadhan's lemmas from Section \ref{ldpcske} can be applied to functionals on this space with suitable semicontinuity and integrability properties, in particular to suitable functionals of $\mathbf L_\lambda$.

Furthermore, we introduce the notations $F_{\min}=\mathrm{ess~inf} (F_0) \in [0,\infty)$, $F_{\max}=\mathrm{ess~sup}(F_0) \in (0,\infty]$. Now one can easily see that if $0<\Fmin \leq \Fmax <\infty$, then the sum of the fading variables is $\mathbb P_2$-almost surely finite. In order to be able to generalize the arguments coming from the fading-free setting, we will have to assume that $0<\Fmin \leq \Fmax<\infty$.

Having defined the fading variables, he \emph{random path-loss}\index{path-loss!random} is given by \[ L((X_i,F_{X_i}), x)=\ell(\vert X_i-x) \vert)F_{X_i}\text{,~ for }X_i \in X^\lambda,~x \in W. \]

We introduce the following simplified notation. If $X$ is a Hausdorff topological space, $\nu \in \mathcal{M}(X)$ and $f: X \to \mathbb R$ is measurable, then we write 
\begin{equation} \label{shorthand} \nu(f(\cdot)) = \int_X f(x) \nu(\d x). \end{equation}
Using this notation, in the spirit of Section \ref{explanation}, in the new model the SIR (signal-to-interference ratio) of the transmitter $(X_i,F_{X_i}) \in \mathbf X^\lambda \subset \mathbf W$, measured at the same time at the receiver $x \in W$ can be given as
\begin{equation} \label{sir}\index{SIR}
\SIR_\lambda((X_i,F_{X_i}),~x,~\mathbf L_{\lambda})=\frac{\ell(\vert X_i-x \vert )F_{X_i}}{\sum_{X_j \in X^\lambda}\ell(\vert X_j-x \vert) F_{X_j}}=\frac{L((X_i,F_{X_i}), x)}{\lambda \mathbf L_{\lambda}(L(\cdot, x))}.
\end{equation}
where the $\SIR$ depends on the fading variables through the points that they are attached to. The expression in the denominator of \eqref{sir}, denoted by $\mathcal{I}_\lambda$, is called the \emph{interference}\index{interference}\index{total received power}\footnote{ \label{STIR}\index{STIR}Since the signal of interest is not subtracted from the sum of the signal strengths of all users in the denominator, following the conventional notation (cf. e.g. \cite[p.~2--3]{KB}) we should call the quotient in \eqref{sir} STIR (i.e., signal-to-total received power ratio) instead of SIR, the quantity in the denominator \emph{total received power}, and the difference between the denominator and the numerator would be referred to as interference. However, in the limit as $\lambda \to \infty$, it makes no difference whether we include the signal of interest in the denominator or not, see \cite[Section 1]{cikk}. For the same reason, our model does not include noise. We will always use the notation of \cite{cikk}, writing SIR instead of STIR.}, see e.g. \cite{BKK} for a model that includes fading variables but where the users are static. One can easily see that the quantities $\SIR_\lambda((X_i,F_{X_i}),~x,~\mathbf L_{\lambda})$ ($X_i \neq x$) are all well-defined and positive for all $X_i \in X^\lambda$, $x \in W$, by the positivity of $\ell$ and $F_0$. Let us use the notations $\ell_{\min}=\min\limits_{\xi,\eta \in W} \ell(\vert \xi-\eta \vert)$, $\ell_{\max}=\max\limits_{\xi,\eta\in W} \ell(\vert \xi-\eta \vert)$ from Section \ref{explanation}, then by continuity of the positive function $\ell$ on the compact set $W$, we have that $0 < \ellmin \leq \ell_{\max}<\infty$. If $\mathbb E [F_0]<\infty$, it is immediate by the finiteness of $\mu(W)$ that for almost all realizations of $X^\lambda$, knowing the realization the interference has finite expectation. If even $0<\Fmin \leq \Fmax <\infty$, then even for every realization of $X^\lambda$ such that $1 \leq X^\lambda(W) < \infty$, i.e. when there exists a $\SIR_\lambda$ quantity in the system to encounter, we have that the following inequalities hold $\mathbb P_2$-almost surely
\begin{multline} \label{korlátka} X^\lambda(W) \ell_{\min} \Fmin \leq \mathcal{I}_\lambda \leq X^\lambda(W) \ell_{\max} \Fmax;~~\ell_{\min} \Fmin \leq L((X_i,F_{X_i}),\eta) \leq \ell_{max} \Fmax, \end{multline}
the latter one for all $X_i \in X^\lambda$ and $\eta \in W$, and this shows that for all $X_i,~X_j\in X^\lambda$ ($i \neq j$), $\SIR_\lambda((X_i,F_{X_i}),~X_j,~\mathbf{L}_{\lambda})$ is $\mathbb P_2$-almost surely bounded away from 0 (i.e., bounded from below by some $\varepsilon>0$). Clearly, since the denominator in \eqref{sir} is a sum of positive terms, and one of these equals the numerator, the $\SIR_\lambda$ takes values in $(0,1]$ for all possible pairs of points if $0 < X^\lambda(W) < \infty$, and on the $\mathbb P_2$ nullset $\lbrace \vert X^\lambda \vert = \infty \rbrace$ we can set all the $\SIR$ quantities to be equal to 0. Note that in \eqref{sir} the transmitter has to be an element of the Poisson point process of users, since it has to exhibit a fading value, but the receiver can be any point of $W$.

As in Section \ref{explanation}, we say that a transmission from $(X_i, F_{X_i}) \in \mathbf X^\lambda$ to $x \in W$ is useful if and only if\index{connected and disconnected users} $\SIR_\lambda((X_i,F_{X_i}),~x,~\mathbf L_{\lambda}) \geq \rho$, where usually we consider $\rho=\rho_0 \lambda^{-1}$ for a positive constant $\rho_0$. If that is the case, then usefulness of this transmission is equivalent to $\SIR((X_i,F_{X_i}),~x,~\mathbf L_{\lambda}) \geq \rho_0$, where we use the definition
\begin{equation} \label{sirkán} \SIR((X_i,F_{X_i}),~x,~\mathbf L_{\lambda})=\lambda \SIR_\lambda((X_i,F_{X_i}),~x,~\mathbf L_{\lambda})=\frac{\ell(\vert X_i-x \vert )F_{X_i}}{\frac{1}{\lambda} \sum_{X_j \in X^\lambda}\ell(\vert X_j-x \vert) F_{X_j}}. \end{equation}


Now we extend the definition of SIR to a deterministic setting, depending neither on the Poisson random measure $\lambda L_{\lambda}$ nor on the fading variables. Instead, we consider the points of $W$ with various deterministic loudnesses. Our large deviation results for frustration probabilities will use this general definition of SIR, similarly the main results of \cite[Section 1.2]{cikk} that we described in Section \ref{explanation}.

Let us introduce the notation \[ \nu(\ell(\vert \cdot-\eta \vert\centerdot):=\int\limits_{\mathbf W} \ell(\vert x-\eta \vert) r \nu(\mathrm{d} r,\mathrm{d}x)\] for $\nu \in \mathcal{M}(W)$. The SIR between transmitter $(\xi,s) \in \mathbf W $ and receiver $(\eta,u) \in \mathbf W$ w.r.t. an arbitrary finite measure $\nu \in \mathcal{M}(\mathbf W)$ is defined as\index{SIR!w.r.t. arbitrary measure}
\begin{equation} \label{végretudjukmiasir} \SIR((\xi,s),~(\eta,u),~\nu)=\frac{\ell (\vert \xi-\eta \vert) s}{\int\limits_{\mathbf{W}} \ell(\vert x-\eta \vert) ~r ~\nu(\mathrm{d} r,\mathrm{d}x)}=\frac{\ell (\vert \xi-\eta \vert) s}{\nu(\ell(\vert \cdot-\eta \vert )\centerdot)}, \end{equation} where on the r.h.s. we used the short hand notation \eqref{shorthand} for an integration in two variables. Since $\ell$ is positive and $W$ is compact, $W \times W \to \mathbb R$, $(x,y) \mapsto \ell(|x-y|)$ is bounded away from 0. Thus, if $\nu(\mathbf W)>0$, then both the numerator and the denominator of $\SIR$ in \eqref{végretudjukmiasir} are strictly positive, and thus the quotient is defined as an element of $[0,\infty)$ if we let it be equal to 0 if the denominator is infinite. Moreover, if 
\begin{equation} \label{elegy} \int\limits_{W \times [\Fmin, \Fmax]} u \nu(\mathrm{d} x,\mathrm{d}u)<\infty, \end{equation} then
$\SIR((\xi,s),~(\eta,u),~\nu)$ is strictly positive for all $(\xi,s),(\eta,u) \in \mathbf W$. However, even if the case $\Finite$, it is in general not true that this SIR quantity is an element of $(0,1]$ for all $(\xi,s),~(\eta,u) \in \mathbf W$, which can be seen e.g. in the setting described in Section \ref{I'm a simulant}. The notion \eqref{végretudjukmiasir} of $\SIR$ generalizes \eqref{sirkán}, which can be now written as a random variable in accordance with \eqref{végretudjukmiasir} as
\begin{equation} \label{tessékhátnemlátodhogyezisasir} \SIR((X_i,F_{X_i}),~x,~\mathbf{L}_{\lambda})=\SIR((X_i,F_{X_i}),~(x,u),~\mathbf{L}_{\lambda})) \end{equation}
for any $u \in (0,\infty)$, in particular also for $u=F_{x}$ if $x \in X^\lambda$. In this sense, the condition \eqref{elegy} is a generalization of the condition $\mathbb E[F_0]<\infty$ for satisfying \eqref{korlátka}. 

The characteristics of the quality of service (QoS)\index{quality of service (QoS)} can be defined analogously to \cite[Section 1.1]{cikk}. More precisely, we write $D((\xi,s),~(\eta,u),~\nu)=g(\SIR(\xi,s),~(\eta,u),~\nu))$ for the QoS of the direct link between $(\xi,s)$ and $(\eta,u)$, for a general $\nu \in \mathcal{M}(\mathbf W)$, where $g: [0,\infty) \to [0,\infty)$ is a Lipschitz continuous function, strictly increasing on $[0,\tilde{\varrho}_+)$ and constant with value $\tilde{c}_+$ on $[\tilde{\varrho}_+,\infty)$, for some $\tilde{c}_+,~\tilde{\varrho}_+ ~>0$. About possible choices of $g$, see Section \ref{explanation}. Furthermore, we set $D((\xi,s),~(\eta,u),~\nu)=\tilde{c}_+$ if $\nu(\mathbf W)=0$. In particular, one can use \eqref{tessékhátnemlátodhogyezisasir} for defining $D((X_i,F_{X_i}),~x,~\mathbf L_{\lambda})=D((X_i,F_{X_i}),~(x,u),~\mathbf L_{\lambda})$ for $X_i \in X^\lambda$, $x \in W$ and for an arbitrary fading value $u \in (0,\infty)$. That is, $D((X_i,F_{X_i}),~x,~\mathbf L_{\lambda})=g(\SIR((X_i,F_{X_i}),~x,~\mathbf L_{\lambda}))$ if $L_\lambda(W) >0$ and $D((X_i,F_{X_i}),~x,~\mathbf L_{\lambda})=\tilde{c}_+$ otherwise. Now one can easily see that in the case $\nu=\mathbf L_{\lambda}$ and $0<\Fmin\leq\Fmax<\infty$, we have 
\begin{equation} \label{kisbecslés} D((X_i,F_{X_i}),~ x,~ \mathbf L_{\lambda})=\tilde{c}_+ \quad \text{ if } L_{\lambda}(W) \leq \beta'_0=\min \lbrace 1, \frac{\ell_{\min} F_{\min}}{\tilde{\varrho}_+~ \ell_{\max} \Fmax} \rbrace. \end{equation} We also define 
\begin{equation} \label{gamma} \Gamma((X_i,~F_{X_i}),~(X_j, F_{X_j}),~x,~\mathbf L_{\lambda})=\min \lbrace D((X_i,~F_{X_i}),~X_j,~ \mathbf L_{\lambda}),~D((X_j, F_{X_j}),~x,~ \mathbf L_{\lambda}) \rbrace \end{equation} for the QoS of the connection from $X_i \in X^\lambda$ to $x \in W$ when routing via the relay $X_j \in X^\lambda$. 

With our definition of SIR, $D$ and $\Gamma$, in the following we set up the notations for relayed communication, analogously to \cite[p.~4]{cikk}, with fading variables included. In the uplink scenario\index{means of communication}\index{means of communication!uplink!direct}, messages are sent out from $X_i \in X^\lambda$ to the base station\index{base station}\index{origin, see \emph{base station}}, which is situated at the origin $o$ of $\mathbb R^d$, either directly or routing via a relay $X_j \in X^\lambda$, under an optimal relay decision, thus the QoS for the relayed uplink communication can be written as\index{means of communication!uplink!relayed}
\[ R((X_i,F_{X_i}),~o,~ \mathbf L_{\lambda})=\max \lbrace D((X_{i},F_{X_i}),~o,~\mathbf L_{\lambda}), \max\limits_{X_j \in X^\lambda} \Gamma((X_{i},F_{X_i}),~(X_{j},F_{X_j}),~o,~\mathbf L_{\lambda}) \rbrace. \] In order to define this quantity, it has not been necessary to have a loudness value $F_o$ at the origin (we recall that the receiver does not have to be a point of $X^\lambda$ in \eqref{sir}). However, in case of downlink communication\index{means of communication!downlink!direct}\index{means of communication!downlink!relayed}, when the origin sends a message to $X_i \in X^\lambda$, either directly or routing via a relay $X_j \in X^\lambda$, the situation is different. Then we intend to introduce the formula \eqref{downlinke} below, after defining $F_o$\index{fading!of the base station}. In the simplest scenario described by Assumption \ref{szamár}, we will assume that $\mathbb E [F_0]< \infty$ and we set $F_o$ to be a constant that is mapped to itself by all triadic discretizations that are defined in Section \ref{discretization}. These are analogues of the discretizations from Section \ref{explanation} extended by the extra fading dimension. About how the choice of the deterministic value $F_o$ influences the behaviour of the system, see Section \ref{Fernsehturm}. 
A more complex scenario is the one of Assumption \ref{csacsi}, where we extend the probability space by a random variable $F_o$, see also Section \ref{Fernsehturm} about this.

If we have defined $F_o$ (deterministic or random), then, in case of downlink communication, the QoS can be expressed as
\begin{equation} \label{downlinke} R((o,F_o),~X_i,~\mathbf L_{\lambda,t})=\max \lbrace D(o,~(X_i,F_{X_i})~\mathbf L_\lambda), \max\limits_{X_j \in X^\lambda} \Gamma((o,F_o),~(X_j,F_{X_j}),~X_i),~\mathbf L_{\lambda})  \rbrace. \end{equation}
Also, we can extend the definition of $R$ to an arbitrary $\nu \in \mathcal{M}(\mathbf{W})$ by setting
\[
R((\xi,s),~(\eta,u),~\nu)=\max \lbrace D((\xi,s),~(\eta,u),~\nu),~\nu\text{-}\esssup\limits_{(\sigma,v) \in \mathbf{W}} \Gamma((\xi,s),~(\sigma,v),~(\eta,u),~\nu) \rbrace,
\]
where $\nu\text{-}\esssup$ means essential supremum w.r.t. $\nu$. Note that $R$ is finite for all transmitters and receivers in $\mathbf W$ because $g \vert_{[\tilde{\varrho}_+,\infty)} \equiv \tilde{c}_+ < \infty$. 
Finally, we mention that $o$ need not be the origin of $\mathbb R^d$, any element of $W$ can serve as a location for the base station\index{base station}. We will still stick to $o$ being the origin, similarly to \cite{cikk}, because this simplifies the notation. The results of Section \ref{kijelentés} can easily be generalized to the case when $o$ is not the origin. However, in some special settings of Chapter \ref{effect}, we use radial symmetry of the intensity measure, where it is required that the base station is located in the origin.

\section{Frustration events and the discretized setting} \label{discretization}
In this section, we follow \cite[Section 1.2]{cikk}. As in that scenario, point processes of the users are frustrated\index{frustration events} if they experience a bad QoS. For a bounded and measurable function $\tau:\mathcal [0,\infty)	\to [0,\infty)$, we define the rescaled random measure associated to the point process of users as
\[
G(\mathbf L_{\lambda},~\tau,~{\mathrm{up}})=\frac{1}{\lambda} \sum\limits_{X_j \in X^\lambda} \delta_{X_j} \tau(R(X_j,~o,~\mathbf{L}_\lambda)).
\]
In particular, $G(\mathbf L_{\lambda},~\tau,~{\mathrm{up}}) \in \mathcal{M}(\mathbf W)$. Also, for a general $\nu \in \mathcal{M}(\mathbf W)$ we define $G(\nu,~\tau,~{\mathrm{up}})$ as an element of $\mathcal{M}(\mathbf W)$ via
\begin{equation} \label{tautológia}
\frac{\mathrm{d} G(\nu,~\tau,~{\mathrm{up}})}{\mathrm{d} \nu}(x,s)=\tau(R((x,s),~(o,F_o),~\nu)).
\end{equation}
In other words, this means
\[ G(\nu,~\tau,~{\mathrm{up}})(\mathbf A)=\int_{\mathbf A}\tau(R((x,s),~(o,F_o),~\nu)) ~\nu(\d x, \d s) \]
for all $\nu$ measurable sets $\mathbf A$.

The property that $\nu$ appears both in the integrand and in the integrator measure on the right hand side may lead to discontinuities of the function $\nu \mapsto G(\nu,~\tau,~{\mathrm{up}})$, see Section \ref{pajti}.

\cite[Section 1.2]{cikk} indicates that when many users are connected to the base station $o$ via the same relay, then communication on full bandwith cannot be guaranteed\label{bandwidth}. In other words, the system may suffer from a small throughput even if many users are connected. This observation is also valid if users have random fadings. Hence, it is important to consider the random measure with respect to the users who have bad QoS w.r.t. direct communication with the base station, which we define as
\[
G(\mathbf L_{\lambda},~\tau,~{\text{up-dir}})=\frac{1}{\lambda} \sum\limits_{X_j \in X^\lambda} \delta_{X_j} \tau(D(X_j,~o,~\mathbf{L}_\lambda)).
\]
Also for ${\nu} \in \mathbf W$, we define $G(\nu,~\tau,~{\text{up-dir}})$ via
\[ \frac{\mathrm{d} G(\nu,~\tau,~{\text{up-dir}})}{\mathrm{d} \nu}(x,s)=\tau(D((x,s),~(o,F_o),~\nu)). \]
For the downlink we define 
\begin{equation} \label{do}
\frac{\mathrm{d} G(\nu,~\tau,~{\text{do}})}{\mathrm{d} \nu}(x,s)=\tau(R((o,F_o),~(x,s),~\nu)),
\end{equation} 
and $G(\nu,~\tau,~\text{do--dir})$ analogously. We define the vector of the four above quantities as
\[
G(\nu,~\boldsymbol \tau) =(G(\nu,~\tau_1,~{\text{up}}),~G(\nu,~\tau_2,~{\text{up-dir}}), G(\nu,~\tau_3,~{\text{do}}), G(\nu,~\tau_4,~{\text{do-dir}}) ).
\]
where $\boldsymbol \tau=(\tau_1,\tau_2,\tau_3,\tau_4)$.

Analogously the fading-free setting, we are interested in random variables $F(G(\mathbf L_\lambda,~ \boldsymbol \tau))$, where $F: \mathbf W^4 \to [-\infty, \infty)$ and $\tau_i:[0,\infty) \to [0,\infty)$, $i=1,\ldots,4$, have to satisfy some appropriate monotonicity conditions. That is, $\tau_i$ is assumed to be decreasing, and $F$ is assumed to be increasing in the sense that for all $\nu,~\nu' \in \mathcal{M}(\mathbf W)$ with $\nu \leq \nu'$ we have $F(\nu) \leq F(\nu')$. Here, as in the fading-free case, $\nu \leq \nu'$ means that $\nu(A) \leq \nu'(A)$ for all measurable $A \subseteq \mathbf W$. We also write $\nu < \nu'$ if $\nu \leq \nu'$ and $\nu \neq \nu'$.

The standard examples for such $F$ and $\tau$ are also analogous to the fading-free setting. We define the measurable functions $F_{\mathbf b}: \mathcal{M}(\mathbf W) \to [-\infty,\infty)$,
\begin{equation} \label{Ffading} 
(\nu_i)_{i \in \lbrace 1,\ldots,4 \rbrace} \mapsto \begin{cases} 0 \quad \quad \text{if } \nu_i(\mathbf W) > b_i, \forall i \in \lbrace 1,\ldots,4\rbrace \\ -\infty \> \text{ otherwise} \end{cases}
\end{equation}
for some $\mathbf b \in \mathbb R^4$ and $\tau_{c}: [0,\infty) \to [0,\infty)$,
\begin{equation} \label{taufading}
\gamma \mapsto \begin{cases} 1 \quad \> \text{if } \gamma<c, \\ 0 \quad \> \text{otherwise}, \end{cases}
\end{equation}
i.e., $\tau_c(x)=\mathds 1 {\lbrace x<c\rbrace}$, and thus \eqref{Ffading} is the mobility-free analogue of \eqref{kisFfading} with random fadings, while \eqref{taufading} is the analogue of \eqref{kistaufading}.

Then, if $\mathbb E [F_0]<\infty$, then for $\boldsymbol \tau_{\mathbf{c}}=(\tau_{c_i})_{i=1,\ldots,4}$ we have
\begin{equation} \label{másodikvégesvárhatóérték} \mathbb E_2 \exp(\lambda F_\mathbf b (G(\mathbf L_\lambda,~ \boldsymbol{\tau_{c}}))) = \mathbb P_2 (G(\mathbf L_\lambda,~ \boldsymbol \tau_{\mathbf c})(\mathbf W) > \mathbf b), \end{equation}
where, as in Section \ref{explanation}, we write $\mathbf a<\mathbf b$ for vectors $\mathbf a=(a_1,\ldots,a_4),~\mathbf b=(b_1,\ldots,b_4) \in \mathbb R^4$ if $a_i<b_i,~\forall i=1,\ldots,4$. Thus \eqref{másodikvégesvárhatóérték} describes the probability that the average number of users who experience a QoS of at most $c_i$ is more than $\lambda b_i$, for all $i \in \lbrace 1,\ldots,4 \rbrace$. 

In the following, we set up the notation for discretizing $\mathbf W$. We use the notation $\mathbb B=\lbrace 3^{-m}: ~m \geq 1 \rbrace$ from Section \ref{explanation}. Let us assume that the fading value $F_o$ of the base station $o$ is fixed. Analogously to the fading-free setting, we choose the triadic discretization to ensure that the origin $(o,F_o)$ is a center of a sub-cube and $\mathbf W$ is an union of sub-cubes of the form \[ \Lambda_\delta(\upsilon,s)=(\upsilon,s) + \left([-\delta r, \delta r]^d \times \left[ -\delta (\Fmax-\Fmin), \delta(\Fmax-\Fmin) \right] \right), \] 
with $\upsilon \in \delta 2r \mathbb Z^d$ and $s \in F_o+\delta 2r \mathbb Z$. To make these quantities well-defined, we will assume the following throughout this chapter:
\begin{ass} \label{szamár}\index{Assumption \ref{szamár} (i.i.d. fadings)}
It holds that $\Finite$, and the origin is defined as $(o,F_o)$, where $F_o=\frac{\Fmin+\Fmax}{2}$. Further we write $\mathbf W=W \times [\Fmin, \Fmax]$ (instead of $\mathbf W=W \times (0,\infty)$).
\end{ass}
The latter sentence of this assumption will be used for making the support of the non-random measures on $\nu \in \mathcal M(\mathbf W)$ bounded, in order to be able to discretize these measures as well, so as to obtain a \emph{finite} discretized space-fading landscape. Also, this new notion of $\mathbf W$ simplifies the notation (it makes it possible to omit pathological considerations with $\nu$-nullsets).

Analogously to the fading-free case, the loudness and space and discretizations are given by\index{discretization} 
\[ 
[\Fmin,\Fmax]_\delta=\left(F_o + 2\delta(\Fmax-\Fmin) \mathbb Z\right) \cap [\Fmin,\Fmax], \quad W_\delta=\delta 2r \mathbb Z^d \cap W.  \]
Thus, we can define the discretized space-fading landscape as
\[ \mathbf W_\delta=W_\delta \times [\Fmin,\Fmax]_\delta.  \]
Now consider two operations that relate the discretized space $\mathbf W_\delta$ to the continuous space $\mathbf W$. Note that under Assumption \ref{szamár}, $\mathcal{M}(\mathbf W_\delta)$ can be identified with $[0,\infty)^{\mathbf W_\delta}$, the set of functions with domain $\mathbf W_\delta$ mapping to $[0,\infty)$. First, we discretize $x \in \mathbf W$ by spatially moving $x	$ to the centers of space-loudness sub-cubes in $\mathbf W_\delta$, i.e. let us write $\varrho'((x,u))=(\varrho_1(x),~\varrho_2(u)) \in \mathbf W_\delta$,\footnote{ Note that this definition does not specify where $\varrho'$ maps the boundary points that are contained in more than one closed sub-cube. However, since the set of all these boundary points has measure zero, and $\mu'$ is a finite measure that is absolutely continuous w.r.t. the Lebesgue measure on $\mathbf W$, we can let $\varrho'$ map these points arbitrarily, e.g. to the one out of the neighbouring sub-cube centers that is the smallest w.r.t. the lexicographic ordering of $W \times [\Fmin, \Fmax]$.} for the discretized value given by 
\[ \varrho': \mathbf W \rightarrow \mathbf W_\delta, \quad (x,u) \mapsto \varrho'((x,u)), \]
where $\varrho'(x)$ denotes the shift of $x \in \mathbf{W}$ to its nearest sub-cube center in $\mathbf W_\delta$. For $\nu \in \mathcal{M}(W)$, the mapping $\varrho'$ induces an image measure
\begin{equation} \label{rócska} \nu^{\varrho'}=\nu \circ \varrho'^{-1} \in \mathcal{M}(\mathbf W_\delta). \end{equation} 
Second, since $\mathbf W_\delta \subset \mathbf W$, for each $\nu \in \mathcal{M}(\mathbf W_\delta)$ the inclusion map 
\[ \imath': \mathbf W_\delta \to \mathbf W,~x \mapsto x \] 
induces the image measure 
\[ \nu^{\imath'}=\nu \times \imath'^{-1} \in \mathcal{M}(\mathbf W) \]
which is essentially identical to the original measure. Note that $\imath'$ is just the inclusion map because there is no mobility in the model, in Section \ref{explanation} we had a more complicated embedding map $\imath$.

The notion of $\delta$-discretized function simplifies when we omit mobility from the model. For $\delta \in \mathbb B$, we say that a function $F:~ \mathcal{M}(\mathbf W)^4 \to [-\infty,\infty)$ is $\delta$-discretized if $F((\nu^{\varrho'})^{\imath'})=F(\nu)$ holds for all $\nu \in \mathcal{M}(\mathbf{W})^4$. Since $\imath'$ is the inclusion map, it is easy to see that the analogue of Lemma \ref{deltadiscretization} holds for this setting, i.e. if $F$ is $\delta$-discretized for some $\delta \in \mathbb B$, then $F$ is also $\delta'$-discretized for any $\delta' \in (0,\delta) \cap \mathbb B$. Note that the functions $F$ defined in \eqref{Ffading} are $\delta$-discretized for all $\delta \in \mathbb B$. We will also use this result throughout this chapter.\index{discretization!$\delta$-discretized function}

In the case of random fadings without mobility, the empirical measure of the discretized marked Poisson process is given as $\lambda \mathbf L_{\lambda}^{\varrho'} \in \mathcal{M}(\mathbf W_\delta)$,
\begin{equation} \label{mértékkel} \quad \lambda \mathbf L_{\lambda}^{\varrho'} ((\xi,w))=\sum\limits_{X_i \in X^\lambda} \delta_{\varrho_1(X_{i})}(\xi) \mathds 1\lbrace \varrho_2(F_{X_i})=w \rbrace=\sharp \lbrace X_i \in X^\lambda \vert~ \varrho_1(X_i)=\xi,~\varrho_2(F_{X_i})=w \rbrace. \end{equation}

We note that this measure does not belong to a simple point process any more, i.e., two distinct points $X_i, X_j \in X^\lambda$ may have the same $\delta$-discretized values $\varrho'_{\delta}(X_i)=\varrho'_{\delta}(X_j)$, and thus $\lambda \mathbf L_{\lambda}^{\varrho'}((\xi,w)) \geq 2$ has positive probability for $(\xi,w) \in \mathbf W_\delta$ such that $\mu'^{\varrho'}((\xi,w))>0$.
We also note that \begin{equation} \label{művarró} {(\mu')}^{\varrho'}((x,u)) = \mu^{\varrho_1}(x) \times \mathbb P(\varrho_2(F_0)=u) \end{equation} is the intensity measure of the discretized marked Poisson process $\mathbf \lambda L^\lambda_\delta$ on $\mathbf W_\delta$ for $\lambda=1$. 

\section{Statement of results of the chapter} \label{kijelentés}\index{main results!of the thesis}
Generalizing the propositions in \cite[Section 1.2]{cikk}, we present a large deviation analysis of the quantities $F(G(\mathbf L_\lambda,~ \boldsymbol \tau))$. Recall that throughout Chapter \ref{eleje} we always assume that Assumption \ref{szamár} holds. Our result is a level-2 large deviation result, as described in Section \ref{ldpcske}. Therefore, the relative entropy defined as \eqref{relatíventrópia}\index{relative entropy} on the space $\mathcal{\mathbf W}$ (endowed with the weak topology\index{weak topology}) plays an important r\^{o}le. Our general theorem is the following.
\begin{thm} \label{1.1}
Let $\tau_i: [0,\infty) \to [0,\infty)$, for $i \in \four$, be bounded, measurable and decreasing functions that map $[\tilde{c}_+,\infty)$ to 0. Further, let $F: \mathcal{M}(\mathbf W)^4 \to [-\infty,\infty) $ be an increasing function that is $\delta$-discretized for some $\delta \in \mathbb B$, bounded from above, and maps the vector of zero measures to $-\infty$. If the $\tau_i \circ \imath'$ are u.s.c. as functions on $[0,\infty)$ and $\nu \mapsto F(\nu^{\imath'})$ is u.s.c. as a function on $\mathcal{M}(\mathbf W_\delta)^4$, then we have
\[ \limsup_{\lambda \to \infty} \frac{1}{\lambda} \log \mathbb E_2 \exp (\lambda F(G(\mathbf L_\lambda,~\boldsymbol \tau))) \leq -\inf\limits_{\nu \in \mathcal{M}(\mathbf W)} \left\lbrace h(\nu \vert \mu') - F(G(\nu,~ \boldsymbol \tau)) \right\rbrace, \] 
whereas if the $\tau_i \circ \imath'$ are l.s.c. as functions on $[0,\infty)$ and $\nu \mapsto F(\nu^{\imath'})$ is l.s.c. as a function on $\mathcal{M}(\mathbf W_\delta)^4$, then 
\[ \liminf_{\lambda \to \infty} \frac{1}{\lambda} \log \mathbb E_2 \exp (\lambda F(G(\mathbf L_\lambda,~\boldsymbol \tau))) \geq -\inf\limits_{\nu \in \mathcal{M}(\mathbf W)} \left\lbrace h(\nu \vert \mu') - F(G(\nu,~ \boldsymbol \tau)) \right\rbrace. \]
\end{thm}
Here $\mu'$ denotes the intensity of the marked Poisson point process $X^\lambda$ for $\lambda=1$, defined in \eqref{újmérték}. According to our assumption that $F$ is $\delta$-discretized, the semicontinuity properties $\nu \mapsto F(\nu^{\imath'})$ and $\tau_i \circ \imath'$ can be checked on finite-dimensional spaces. This is substantially simpler than for $F$ and $\tau_i$, because the domain of $F$ is infinite-dimensional.

As a special case of Theorem \ref{1.1}, we determine the rate of decay for the frustration probabilities $\mathbb P_2 (G(\mathbf L_\lambda,~ \boldsymbol \tau_{\mathbf c}) (\mathbf W) > \mathbf b)$ where $\boldsymbol \tau_{\mathbf c}$ is defined in \eqref{taufading}. Furthermore, we define $[0,\tilde{\mathbf c}_+ )=[0,\tilde{c}_+)^4$. In particular, analogously to the notation of \cite{cikk}, we do not use boldface notation for the zero vector $0 \in \mathbb R^4$.
\begin{cor} \label{1.2}\index{frustration probabilities!rate function}
Let $\mathbf b \in \mathbb R^4$ and $\mathbf c \in [0,\tilde{\mathbf c}_+)$. Then,
\begin{equation} \label{minike} \lim_{\lambda \to \infty} \frac{1}{\lambda} \log \mathbb P_2 (G(\mathbf L_\lambda, ~\boldsymbol \tau_{\mathbf c})(\mathbf W)>\mathbf b)= - \inf\limits_{\nu \in \mathcal{M}(\mathbf W):~G(\nu, ~\boldsymbol \tau_{\mathbf c}) (\mathbf W) > \mathbf b} h(\nu \vert \mu'). \end{equation}
\end{cor}
Finally, we formalize the observation that also in the extended model random fadings, the probability of frustration events that are unlikely w.r.t. the a priori measure $\mu'$ decays at an exponential speed.
\begin{cor} \label{1.3}\index{frustration probabilities!exponential decay}
Let $\mathbf b \in \mathbb R^4,~\mathbf c \in [0, \tilde{\mathbf c}_+)$ and assume that $G((1+\varepsilon)\mu',~ \boldsymbol \tau_{\mathbf c})(\mathbf W) \leq \mathbf b$ for some $\varepsilon>0$. Then we have
\begin{equation} \label{fontos} \limsup_{\lambda \to \infty} \frac{1}{\lambda} \log \mathbb P_2(G(\mathbf L_\lambda,~ \boldsymbol \tau_{\mathbf c})(\mathbf W) > \mathbf b ) < 0. \end{equation}
\end{cor}
For a full classification of the cases that are not covered by Corollary \ref{1.3}, see Section \ref{kalsszikus}.

In the remainder of this chapter, we proceed similarly to the structure of \cite[Sections 2--6]{cikk}. In Section \ref{kettő} we present an outline of the proof of Theorem \ref{1.1}. We continue by preliminary results about monotonicity and continuity properties of the QoS trajectories and about linear perturbations of relative entropies in Section \ref{pajti}. This is followed by the sprinkling argument in Section \ref{sprinkle}, which is one key ingredient of the proof of Theorem \ref{1.1}. After this, in Sections \ref{kettőegy} and \ref{kettőkettő}, we establish the proof of Propositions \ref{2.1} and \ref{2.2}, which are important auxiliary results for the proof of Theorem \ref{1.1}. The proof of Theorem \ref{1.1} takes place in Section \ref{egyegy}, while the proofs of Corollaries \ref{1.2} and \ref{1.3} are established in Section \ref{egykettőegyhárom}. Section \ref{kalsszikus} gives a classification of frustration probabilities, determining when they decay at an exponential speed. Once we have obtained the results of this chapter, we can detect the effects coming from the random fadings, relax the i.i.d. assumption on the loudnesses of users, and accompany a fading random variable to the origin. These are the goals of Chapter \ref{effect}. 
\section{Outline of the proof of Theorem \ref{1.1}} \label{kettő}
The mathematical analysis of relay-based communication becomes much less technical if we discretize the possible user locations and loudness values. Pairs of spatial positions and fading values of the users are no longer distributed according to $\mu'$ but according to $\mu'^{\varrho'}$, as defined in \eqref{művarró}. In other words, analogously to the approach of \cite[Section 2]{cikk}, but with random fadings and without mobility, we use the approximation that all users are located at the sites in $W_\delta$ and have fading values in $[\Fmin,\Fmax]_\delta$. By the assumptions on $\mu$ we have that 
\begin{equation} \label{kapa} \max_{(x,u) \in \mathbf W_\delta} \mu'^{\varrho'}((x,u))\end{equation}
tends to zero as $\delta$ tends to zero. Thus, so does
\begin{equation} \label{kifalva} \kappa_\delta=\min_{(x,u) \in \mathbf W_\delta:~\mu'^{\varrho'}((x,u)) >0} \mu'^{\varrho'}((x,u)), \end{equation} which is a quantity that appears in several auxiliary results of this chapter. 
We introduce the analogue of $G(\nu,~ \boldsymbol \tau)$, as given in \eqref{tautológia}, in the discretized setting. For a general $\nu \in \mathcal{M}(\mathbf W_\delta)$ and a general bounded $\tau: [0,\infty) \to [0,\infty)$, $G(\nu,~\boldsymbol \tau, ~\text{up})$ is given as an element of the space of finite measures $\mathcal{M}(\mathbf W_\delta)$ via
\[ \frac{\d G(\nu,~\tau,~\text{up})}{\d \nu}=\tau(R((x,u),~(o,F_o),~\nu)) \] 
and similarly for $G(\nu,~\tau,~\text{up-dir})$, $G(\nu,~\tau,~\text{do})$ and $G(\nu,~\tau,~\text{do-dir})$. Also we put
\[ G(\nu,~\boldsymbol \tau) =(G(\nu,~\tau_1,~\text{up}),~G(\nu,~\tau_2,~\text{up-dir}),~G(\nu,~\tau_3,~\text{do}),~G(\nu,~\tau_4,~\text{do-dir})), \] 
where $\boldsymbol \tau=(\tau_i)_{i \in \four}$. 
The following proposition will be used for passing from considering frustrated users of the discretized setting to the original one for small discretization parameter values $\delta$. 
\begin{prop} \label{2.1}\label{frustration probabilities!in the discrete setting}
For all $\varepsilon >0$ there exists $\delta'=\delta'(\varepsilon) \in \mathbb B$ such that for all $\delta \in \mathbb B \cap (0,\delta')$, $\nu \in \mathcal{M}(\mathbf W)$ and $\tau_i:\mathcal [0,\infty) \to [0,\infty)$ bounded and decreasing for all $i \in \lbrace 1,\ldots,4 \rbrace$, 
\[ G((1-\varepsilon)\nu^{\varrho'},~ \boldsymbol \tau \circ \imath') \leq G(\nu,~ \boldsymbol{\tau})^{\varrho'} \leq G((1+\varepsilon)\nu^{\varrho'}, ~\boldsymbol \tau \circ \imath'). \]
\end{prop}
Working in the discrete setting is substantially simpler. Instead of the marked Poisson process in $\mathbf W$, we have independent Poisson random variables attached to every element of the space-fading grid $\mathbf W_\delta$, with parameters $\mu'^{\varrho'}(\cdot,\centerdot)$. In particular, the corresponding relative entropy for the discretized setting is defined as\index{relative entropy!discrete}
\[ h(\nu \vert \mu'^{\varrho'}) =\sum\limits_{(x,u) \in \mathbf W_\delta} h \left( {\nu}((x,u)) \left| \right. {\mu}'^{\varrho'}(u) \right), \]
where for $a,~b \geq 0$ we write $h(a\vert b)=a \log \frac{a}{b}-a+b$, in accordance with \eqref{diszkrétrelatíventrópia}.

Then the analogue of Theorem \ref{1.1} in the discretized setting is stated as follows.
\begin{prop} \label{2.2}\index{main results!of the thesis!discrete version}
Let $0<\alpha<2$ and $\tau_i: [0,\infty) \rightarrow [0,\infty)$, for $i \in \four$, be bounded, measurable and decreasing functions which map $[\tilde{c}_+,\infty)$ to 0. Further, let $F: \mathcal{M}(\mathbf W_\delta)^4 \rightarrow [-\infty,\infty)$ be an increasing measurable function that is bounded from above. Moreover, assume that $F$ maps the vector of zero measures to $-\infty$. If $F$ and $\tau_i$ are u.s.c., then we have
\[ \limsup\limits_{\lambda \to \infty} \frac{1}{\lambda} \log \mathbb E_2 \exp (F(G(\alpha \mathbf L_\lambda^{\varrho'},~\boldsymbol \tau))) \leq - \inf\limits_{{\nu} \in \mathcal{M} (\mathbf W_\delta)} \left\lbrace h({\nu} \vert {\mu}'^{\varrho'}) - F(G(\alpha {\nu},~\boldsymbol \tau)) \right\rbrace, \]
whereas if $F$ and $\tau_i$ are l.s.c., then
\[ \liminf\limits_{\lambda \to \infty} \frac{1}{\lambda} \log \mathbb E_2 \exp (F(G(\alpha \mathbf L_\lambda^{\varrho'},~\boldsymbol \tau))) \geq - \inf\limits_{{\nu} \in \mathcal{M} (\mathbf W_\delta)} \left\lbrace h({\nu} \vert {\mu}'^{\varrho'}) - F(G(\alpha {\nu},~\boldsymbol \tau)) \right\rbrace. \]
\end{prop}
The difficulty in the proof of Proposition \ref{2.2} is caused by the discontinuity of the function ${\nu} \mapsto G(\nu,~ \boldsymbol \tau)$\index{frustration events!discontinuity}. Indeed, if the number of users on a certain site tends to zero, then in the limit users cannot relay via this site. This might cause a sudden drop in the QoS and therefore a sudden increase of the number of frustrated users. Therefore we have to handle continuity problems originating from configurations that exhibit a small but positive number of users. Such problems can only appear for relayed communication; we will prove in Section \ref{pajti} that ${\nu} \mapsto G(\nu,~ \tau,~\text{up-dir})$ and ${\nu} \mapsto G(\nu,~\tau,~\text{do-dir})$ are continuous. In order to overcome such continuity problems, we will use a sprinkling argument, following \cite[Section 3.2]{cikk}. That is, increasing the Poisson intensity slightly, we add a small number of additional users in such a way that after the sprinkling every occupied site contains a number of users of the same order as the Poisson intensity. We show that assuming that we observe the desired kind of sprinkling comes at negligible cost on the exponential scale, and that on the resulting configurations the map $\nu \mapsto G(\nu,~\boldsymbol \tau)$ exhibits the desired continuity properties. As long as we assume that $\Finite$, we can perform the sprinkling construction analogously to \cite[Section 3.2]{cikk}, incorporating the fading dimension. The effect of the random nature of the fadings appears rather in the form of the minimizers of relative entropy \eqref{fontos} than in the proof techniques of this chapter. 

\section{Auxiliary lemmas} \label{pajti}

The purpose of this section is to set up the preliminary results that we will need for proving the propositions of Section \ref{kijelentés}, following the approach of \cite[Section 3.1]{cikk}, incorporating the new loudness dimension. Some preliminary results in the fading-free setting are related to mobility of users, these are therefore not needed in the setting of this thesis. We also omit proofs that are, after trivial changes of notation, entirely analogous to the corresponding proofs in the fading-free setting of \cite{cikk}, while we present all the proofs which do not immediately follow from those fading-free analogues. We introduce the following notations for the empirical measures in the discretized setting: $X^\lambda_\delta=\lambda L_\lambda^{\varrho_1}$ and $\mathbf X^\lambda_\delta=\lambda \mathbf L_\lambda^{\varrho'}$. We write $V(\nu)=\lbrace (\upsilon,s) \in \mathbf W_\delta \vert~ \nu((\upsilon,s))=0 \rbrace$, for $\nu \in \mathcal{M}(\mathbf W_\delta)$ as the zero set of the measure $\nu$\index{zero sets of measures}.

\begin{lem} \label{3.1}
Let $\delta \in \mathbb B$, $(\xi,s),(\eta,u) \in \mathbf W_\delta$ be arbitrary.
\begin{enumerate}[(i)]
\item If $\nu, \nu' \in \mathcal{M}(\mathbf W_\delta)$ are such that $\nu \leq \nu'$, then $D((\xi,s),~(\eta,u),~\nu') \leq D((\xi,s),~(\eta,u),~\nu)$. 
\item If $\nu, \nu' \in \mathcal{M}(\mathbf W_\delta)$ are such that $\nu < \nu'$ and $D((\xi,s),~(\eta,u),~\nu')<\tilde{c}_+$, then \\ $D((\xi,s),~(\eta,u),~\nu') < D((\xi,s),~(\eta,u),~\nu)$.
\item If $\nu, \nu' \in \mathcal{M}(\mathbf W_\delta)$ are such that $\nu \leq \nu'$ and $V(\nu)=V(\nu')$, then \\ $R((\xi,s),~(\eta,u),~\nu') \leq R((\xi,s),~(\eta,u),~\nu)$.
\item If $\nu, \nu' \in \mathcal{M}(\mathbf W_\delta)$ are such that $\nu < \nu'$, $V(\nu)=V(\nu')$ and $R((\xi,s),~(\eta,u),~\nu')<\tilde{c}_+$, then $R((\xi,s),~(\eta,u),~\nu') < R((\xi,s),~(\eta,u),~\nu)$.
\item If $\lambda' \geq \lambda >0$ and $\sigma \in (0,1)$ are such that $\frac{X_{\delta}^{\lambda'} (W_{\delta})}{ X_{\delta}^{\lambda} (W_{\delta})} \leq 1 + \frac{\ell_{\min} \Fmin (1-\sigma)}{\ell_{\max} \Fmax \sigma}$, then we have that $\mathbb P_2$-almost surely, $D(\xi,~\eta,~\mathbf L^{\varrho'}_{\lambda}) \leq D(\xi,~\eta,~\sigma \mathbf L^{\varrho'}_{\lambda'})$.
\end{enumerate}
\end{lem}
\begin{proof}
Assume $\nu \leq \nu'$. Then, since \begin{equation} \label{kétsir} \frac{\SIR((\xi,s),~(\eta,u),~\nu')}{\SIR((\xi,s),~(\eta,u),~\nu)}=\frac{\nu(\ell(\vert \cdot-\eta \vert\centerdot)}{\nu'(\ell(\vert \cdot-\eta \vert\centerdot)} \end{equation} and both the numerator and the denominator is finite under Assumption \ref{szamár}, the monotonicity properties of $g$ imply (i) and (ii). This monotonicity extends to expressions of the form $\Gamma((\xi,s),~(\eta,u),~(\upsilon,v),~\nu)$ as well, where $(\xi,s),~(\eta,u),~(\upsilon,v) \in \mathbf W_\delta$. Under the additional condition $V(\nu)=V(\nu')$, we have that $\nu$ and $\nu'$ have the same zero-set, and therefore (iii) and (iv) also holds.

For (v), by \eqref{kétsir} and the monotonicity of $g$, it suffices to show that $\mathbb{P}_2$-a.s.
\[ \frac{\mathbf L_{\lambda',t}(L(\cdot,\eta))-\mathbf L_{\lambda}(L(\cdot,\eta))}{\mathbf L_{\lambda}(L(\cdot,\eta))} \leq \frac{\lambda' \mathbf L_{\lambda'}(L(\cdot,\eta)) - \lambda \mathbf L_{\lambda}(L(\cdot,\eta))}{\lambda \mathbf L_{\lambda}(L(\cdot,\eta))} \leq \frac{1-\sigma}{\sigma}. \]
Indeed, under the event $\lbrace X^\lambda<\infty,~ F_\xi \in [\Fmin, \Fmax], \forall \xi \in X^\lambda \rbrace$, i.e., $\mathbb P_2$-a.s., we have
\[ \frac{\mathbf L_{\lambda'}(L(\cdot,\eta))-\mathbf L_{\lambda}(L(\cdot,\eta))}{\mathbf L_{\lambda}(L(\cdot,\eta))} = \frac{\frac{\lambda}{\lambda'} \mathbf X^{\lambda'}_{\delta}(L(\cdot,\eta))-\mathbf X^{\lambda}_{\delta}(L(\cdot,\eta))}{\mathbf X^{\lambda}_{\delta}(L(\cdot,\eta))} \leq \frac{\ell_{\max} \Fmax}{\ell_{\min} \Fmin} \frac{X^{\lambda'}_{\delta}(W_\delta)-X^{\lambda}_{\delta}(W_\delta)}{X^{\lambda}_{\delta}(W_\delta)} \leq \frac{1-\sigma}{\sigma},  \]
as asserted.
\end{proof}
The next lemma establishes Lipschitz continuity and in particular Borel measurability for the QoS quantities as functions on $\mathbf W$. Note that since on a finite-dimensional vector space all norms are equivalent, it is sufficient to use the following $\ell^1$ product norm\index{$\ell^1$ product norm} $\Vert \cdot \Vert$ on $\mathbb R^d \times \mathbb R=\mathbb R^{d+1} \supset \mathbf W$ \[ \mathbb R^{d+1}=\mathbb R^d \times \mathbb R \ni (x,u) \mapsto \vert x \vert + \vert u \vert = \Vert x \Vert_2 + \vert u \vert. \] This norm simplifies the following computations, and shows the effect of the deviations of path-losses and fadings separately. In the same time, it preserves the geometrical structure given by the square norm in the space coordinate. 
\begin{lem} \label{3.5} Let $(y,v) \in \mathbf W$ and $\nu \in \mathcal{M}(\mathbf W)$. Then, as mappings from $\mathbf W$ to $[0,\infty)$
\begin{enumerate}[(i)]
\item $(x,u) \mapsto D((x,u),~(y,v),~{\nu})$ is Lipschitz continuous,
\item $(x,u) \mapsto R((x,u),~(o,F_o),~{\nu})$, $(x,u) \mapsto R((o,F_o),~(x,u),~{\nu})$ are Lipschitz continuous.
\end{enumerate}
\end{lem}
\begin{proof}
First we note that if ${\nu}(\mathbf W)=0$, then by the definition of $g$, we have that ${D}$ and ${R}$ are constant and hence Lipschitz continuous. Let ${\nu}(\mathbf W) >0$. We show that $x \mapsto \SIR((x,u),~(y,v),~{\nu})$ is Lipschitz continuous. Indeed, for any $(x,u),~(x',u') \in \mathbf W$ the following holds
\begin{align*} \Vert \SIR((x,u),~(y,v),~{\nu})-\SIR((x',u'),~(y,v),~{\nu}) \Vert &\leq \frac{1}{\ell_{\min} \Fmin {\nu}(\mathbf  W)} \vert \ell(\vert x-y \vert) u-\ell(\vert x'-y \vert) u'\vert \\ &\leq   \frac{\left( \vert \ell(\vert x-y\vert) (u - u') \vert  + \vert \ell ( \vert x-y\vert) u' - \ell(\vert x'-y \vert) u' \vert \right) }{\ell_{\min} \Fmin {\nu}(\mathbf W)}  \\ &\leq  \frac{1}{\ell_{\min} \Fmin {\nu}(\mathbf W)}  \left( \ellmax \vert u-u' \vert + \Fmax J_2 \vert x-x' \vert \right) \\ &\leq \frac{1}{\ell_{\min} \Fmin {\nu}(\mathbf W)} (\ellmax+J_2 \Fmax) \Vert (x,u)-(x',u') \Vert. \end{align*} Note that the Lipschitz parameter provided by this inequality is independent of $(y,v)$. Now, since $g$ is Lipschitz continuous, (i) follows from the definition of ${D}$. Now we show that the claim of (ii) holds for ${R}((x,u),~(y,v)~{\nu})$. 
Let $\varepsilon>0$ be arbitrary and $(x,u),~(x',u') \in \mathbf W$, then we have
\begin{multline} \label{Epsz Ilona} \nu\text{-}\esssup\limits_{(y,v) \in \mathbf W} \Gamma((x,u),~(y,v),~(o,F_o),~\nu)-{\nu}\text{-}\esssup\limits_{(y,v) \in \mathbf W} \Gamma((x',u'),~(y,v),~(o,F_o),~{\nu}) \\ \leq  \Gamma((x,u),~(y,v),~(o,F_o),~{\nu})- \Gamma((x',u'),~(y,v),~(o,F_o),~{\nu})+2\varepsilon \end{multline}
for some $(y,v)=(y,v)( (x,u)) \in N^c$, where $N=N((x',u'))$ is a $\nu$-nullset. Since $\Gamma$ is a maximum of Lipschitz continous functions, with parameter independent of $(y,v)$, in the r.h.s. of \eqref{Epsz Ilona} we can further estimate \begin{small}
\[\Gamma((x,u),~(y,v),~(o,F_o),~{\nu})- \Gamma((x',u'),~(y,v),~(o,F_o),~{\nu}) \leq \alpha (\vert x - x' \vert + \vert u - u' \vert) = \alpha \Vert (x,u)-(x',u') \Vert \] \end{small}
\hspace{-8pt} for some constant $\alpha >0$. Sending $\varepsilon$ to zero and using the symmetry in $(x,u)$ and $(x',u')$ gives the Lipschitz continuity of ${R}((x,u),~(o,F_o),~{\nu})$. For ${R}((o,F_o),~(x,u),~{\nu})$, the proof is analogous.
\end{proof}
The following result corresponds to the discretized setting in the case $\Finite$. Note that in the case of relayed communication, the QoS of a given user is at least as sensitive to the distribution of the surrounding users than in the fading-free setting. This is true because the disappearance of possible relays\index{means of communication!uplink!relayed}\index{means of communication!downlink!relayed}, caused by the user being situated in an area that is empty, 
may lead to a sudden decrease in QoS. This is the reason why ${\nu} \mapsto {R}((x,u),~(o,F_o),~{\nu})$ and ${\nu} \mapsto {R}((o,F_o),~(x,u),~{\nu})$ are only l.s.c, analogously to the fading-free setting.
\begin{lem} \label{3.6}
For all $(x,u),~(y,v)~\in \mathbf W_\delta$, the maps ${\nu} \mapsto {D}((x,u),~(y,v),~{\nu}),~ {\nu} \mapsto {R}((x,u),~(o,F_o),~{\nu}),~\\ {\nu} \mapsto {R}((o,F_o),~(x,u),~{\nu})$, from $\mathcal{M}(\mathbf W_\delta)$ to $[0,\infty)$ are continuous, l.s.c. and l.s.c. respectively.
\end{lem}
The proof of this lemma is entirely analogous to the one of \cite[Lemma 3.6]{cikk}, therefore we omit it. 
As in \cite[Section 3.1]{cikk}, we call a function $f: [0,\infty)^m \to [-\infty,\infty)$ decreasing if $f$ is decreasing w.r.t. the partial order on $[0,\infty)^m$ given by \begin{equation} \label{kisebbegyenlő} x \leq y \quad \Leftrightarrow \quad (x_i \leq y_i,~\forall i \in \lbrace 1, \ldots, m \rbrace).\end{equation} $f$ is called increasing if $-f$ is decreasing. Further we call a function $g: [0,\infty)^n \to \mathbb R^m$ u.s.c. if $g$ is u.s.c. as a mapping in every coordinate $1 \leq i \leq m$ in the image space. $g$ is called l.s.c. if $-g$ is u.s.c.\index{semicontinuous function!taking $[0,\infty)^m$ to $[-\infty,\infty)$}
\begin{rem} \label{3.7,8}
Using \cite[Lemma 3.7,~Remark 3.8]{cikk}, one can easily see that under Assumption \ref{szamár} and the assumption that $\tau$ is u.s.c. and decreasing, the map ${\nu} \mapsto \tau(R((x,u)~(o,F_o),~{\nu})$ is u.s.c., and the map $\nu \mapsto F(G(\nu,~\boldsymbol \tau))$ appearing in Proposition \ref{2.2} is u.s.c. for any increasing and u.s.c. function $F$. 
\end{rem}
Finally, we state the lemmas of \cite[Section 3.3]{cikk} on relative entropies under linear perturbation. Since we work with a model without mobility, we write them with measures on $\mathbf W$ (instead of $\mathcal{L}$ from Section \ref{explanation}), but this does not cause any substantial change in their proofs, which therefore we omit.
\begin{lem} \label{3.11}
Let $a>0$ and ${\nu} \in \mathcal{M}(\mathbf W)$ be arbitrary. Then,
\[ h(a {\nu} \vert {\mu}') = a h(\nu \vert {\mu}')+a \log(a) {\nu} (\mathbf W)+(1-a) {\mu}'(\mathbf W). \]
\end{lem}
\begin{cor} \label{3.12}\index{relative entropy!under linear perturbation}
Let ${\nu} \in \mathcal{M}(\mathbf W)$ and $\varepsilon \in (0,~1/2)$ be arbitrary. Then,
\[ h((1+\varepsilon) {\nu} \vert {\mu}') \leq (1+3 \varepsilon) h({\nu} \vert {\mu}')+3 \varepsilon {\mu}'(\mathbf W) \]
and
\[ h((1-\varepsilon) {\nu} \vert {\mu}') \geq (1-3 \varepsilon) h({\nu} \vert {\mu}')-3 \varepsilon {\mu}'(\mathbf W). \]
\end{cor}
Note that since $\mu'=\mu \tensor \zeta$, where $\zeta=\mathbb P \circ F_0^{-1}$ is a probability measure, we have $\mu'(\mathbf W)=\mu(W)$.
\begin{rem} \label{3.13}
Lemma \ref{3.11} and Corollary \ref{3.12} remain true if $\mathbf W$ and $\mu'$ are replaced by $\mathbf W_\delta$ and ${\mu}'^{\varrho'}$, respectively.
\end{rem}
\section{Sprinkling construction with random fadings} \label{sprinkle}\index{sprinkling construction}
As we described in the paragraph after stating Proposition \ref{2.2}, the main difficulty in analyzing the empirical measures $G(\mathbf L^{\varrho'}_{\lambda},~\boldsymbol \tau)$ comes from the discontinuity of the indicators $\mathds 1\lbrace \mathbf L^{\varrho'}_{\lambda} ((x,u)) >0 \rbrace,~(x,u) \in \mathbf W_\delta$. In other words, similarly to the fading-free case, the configurations that obstruct us in applying the contraction principle (Theorem \ref{contractionprinciple}) are those that exhibit $\delta$-discretized trajectories with a small but non-zero number of users. Now we show that, knowing the values of the marked Poisson point process $\mathbf X^\lambda$, a small increase in the intensity of $\mathbf X^\lambda$ provides us with a sufficient amount of additional randomness to be able to neglect such pathological configurations on the exponential scale. In other words, we follow the approach of \cite[Section 3.2]{cikk} in order to give a sprinkling argument and obtain analogues of the lemmas of that section. The difference from the fading-free case is that our estimations involve the quantity $\kappa_\delta$ defined in \eqref{kifalva}, which depends on the lowest density value of the fading variable $F_0$, and the quantity $\beta'_o$ defined in \eqref{kisbecslés}, which depends on the extremal fading values $\Fmin$ and $\Fmax$.


To perform the sprinkling operator with parameter $\varepsilon_0 \in (0,1)$, we define $X^\lambda_\delta$, $X^{\lambda'}_\delta$, $\mathbf X_\delta^\lambda$ and $\mathbf X_\delta^{\lambda'}$ as before, with $\lambda'=(1+\varepsilon_1)\lambda$, where we put $\varepsilon_1=2\varepsilon_0 {\kappa_\delta}^{-1}$, defining
$\kappa_\delta$ via \eqref{kifalva}.
In the following, we always assume that $\delta \in \mathbb B$ is sufficiently small to ensure that $\kappa_\delta$ is not greater than 1. Similarly to \cite[Section 3.2]{cikk}, we define
\[ \mathbf Q = \lbrace (x,u) \in \mathbf W_\delta \vert~ \mathbf L_{\lambda'}^{\varrho'} ((x,u)) \leq \varepsilon_0 \rbrace \quad \text{ and } \mathbf V=\lbrace (x,u) \in \mathbf W_\delta \vert~\mathbf L^{\varrho'}_\lambda ((x,u)) =0 \rbrace \]
to be the \emph{quasi-empty}\index{sites!quasi-empty} and the \emph{virtual}\index{sites!virtual} sites in $\mathbf W_\delta$, respectively. Furthermore,
we introduce the event
\[ \mathbf E_{\varepsilon_0}=\lbrace \mathbf Q \subseteq \mathbf V \rbrace \]
that all quasi-empty sites are virtual. Finally, we define
\[ \mathbf{E}'_{\varepsilon_0} = \lbrace  \mathbf X^\lambda_\delta (\mathbf W_\delta) \geq (1-{\varepsilon}_2 ) \mathbf X^{\lambda'}_{\delta} (\mathbf W_\delta) \rbrace, \] 
where $\varepsilon_2=4 \varepsilon_1 {\beta'_o}^{-1} \sharp \mathbf W_\delta (1+{\mu'}^{\varrho'} (\mathbf W_\delta))$ and $\sharp \mathbf W_\delta$ denotes the number of sub-cubes in the discretization $\mathbf W_\delta$ of $\mathbf W$. Note that $\kappa_\delta$ depends on the distribution of the discretized fading variable $\varrho_2(F_0)$ and $\beta'_o$ depends on $\Fmin$ and $\Fmax$ via \eqref{kisbecslés}. By the following two lemmas, we formalize the sprinkling heuristic explained above.
\begin{lem} \label{3.9}
For all sufficiently small $\varepsilon_0 \in (0,1)$, there exists $\lambda_0=\lambda_0(\varepsilon_0)$ such that for all $\lambda \geq \lambda_0$ we have
\[ \mathbb P_2 (\mathbf E_{\varepsilon_0} \cap \mathbf E'_{\varepsilon_0} \vert~ \mathbf X_\delta^{\lambda'}) \geq \exp(-\sqrt{\varepsilon_0}\lambda) \mathds 1\lbrace L_{\lambda'}^{\varrho_1}(W_\delta) \geq \beta'_0/2 \rbrace \]
holds $\mathbb P_2$-a.s.
\end{lem}
\begin{proof}
We proceed analogously to the proof of \cite[Lemma 3.9]{cikk}. First of all, the key observation is that an independent thinning\index{thinning} of $\mathbf X_\delta^{\lambda'}$ with survival probability $\frac{1}{1+\varepsilon_1}$ results $\mathbf X_\delta^\lambda$. (About thinning of Poisson processes, see Theorem \ref{colouring}.) This means that for each $(x,u) \in \mathbf W_\delta$ there exist $N'_{(x,u)}=\mathbf X_\delta^{\lambda'}(x)$ independent Bernoulli($1/(1+\varepsilon_1)$)-distributed random variables $\lbrace U_k((x,u)) \rbrace_{1 \leq k \leq N'_{(x,u)}}$ such that 
\[ \mathbf X_\delta^\lambda((x,u))=\sum\limits_{k=1}^{N'_{(x,u)}} U_k((x,u)). \]
In particular, $\mathbf E_{\varepsilon_0}$ is the event that $U_k((x,u))=0$ for all $(x,u) \in \mathbf Q$, for all $1 \leq k \leq N'_{(x,u)}$. 
Now, let $\mathbf E_{\varepsilon_0}''$ denote the event that $\mathbf X_\delta^\lambda((x,u)) \geq (1-\varepsilon_1) \mathbf X_\delta^{\lambda'}((x,u))$ holds for all $(x,u) \in \mathbf W_\delta \setminus \mathbf Q$. Since $\sharp \mathbf Q \leq \sharp \mathbf W_\delta$, this event implies that $\mathbf X^\lambda_\delta(\mathbf W_\delta)$ is bounded below by
\begin{multline*} (1-\varepsilon_1) \mathbf X_\delta^{\lambda'} (\mathbf W_\delta \setminus \mathbf Q) =(1-\varepsilon_1) \left( \mathbf X_\delta^{\lambda'} (\mathbf W_\delta)-\mathbf X_\delta^{\lambda'} (\mathbf Q) \right) \geq (1-\varepsilon_1) \left( \mathbf X_\delta^{\lambda'} (\mathbf W_\delta) - \sharp \mathbf W_\delta \varepsilon_0 \lambda' \right). \end{multline*} 
If additionally $\lbrace L^{\varrho_1}_{\lambda'} (W_\delta) \geq \beta'_o/2 \rbrace$ occurs, then we can estimate 
\[ \mathbf X_\delta^{\lambda} (\mathbf W_\delta) = X_\delta^\lambda (W_\delta) \geq (1-\varepsilon_1) (1-2 \sharp \mathbf W_\delta \varepsilon_0 {\beta'_o}^{-1}) X_\delta^{\lambda'} (W_\delta) \geq (1-\varepsilon_2) X^{\lambda'}_\delta (W_\delta)=(1-\varepsilon_2) \mathbf X^{\lambda'}_\delta (\mathbf W_\delta). \]
Therefore $\mathbf{E}''_{\varepsilon_0} \cap \lbrace L^{\varrho_1}_{\lambda'} (W_\delta) \geq \beta'_o/2 \rbrace \subseteq \mathbf{E}'_{\varepsilon_0}$, and it remains to bound $\mathbb P_2(\mathbf E_{\varepsilon_0} \cap \mathbf E''_{\varepsilon_0} \vert \mathbf X^{\lambda'}_\delta)$ from below. First, we have
\[ \mathbb P_2(\mathbf E''_{\varepsilon_0}) \geq \mathbb P_2(U_k((x,u))=1~ \forall (x,u) \in \mathbf W_\delta \setminus \mathbf Q,~\forall k= 1,\ldots,N'_{(x,u)} \vert~\mathbf X^{\lambda'}_\delta) \geq (1+\varepsilon_1)^{- \sum_{(x,u) \in \mathbf W_\delta} N'_{(x,u)}}. \]
By the law of large numbers, if $N'_{(x,u)}$ in $\mathbf X_\delta^{\lambda'}((x,u))$ is large, $\mathbb P_2(\mathbf X^\lambda_\delta((x,u)) \geq (1-\varepsilon_1) \mathbf X_\delta^{\lambda'} (x) \vert~\mathbf X_\delta^{\lambda'})$ is close to one. 
Thus, there exists $c_1 >0$ such that $\mathbb P_2(\mathbf E''_{\varepsilon_0} \vert~\mathbf X^{\lambda'}_\delta) \geq c_1$ holds $\mathbb P_2$-almost surely for every $\lambda \geq 1$. Hence, since conditioned on $\mathbf X^\lambda_\delta$, the events $ \mathbf E_{\varepsilon_0}$ and $\mathbf E''_{\varepsilon_0}$ are independent, we conclude that
\[ \mathbb P_2(\mathbf E_{\varepsilon_0} \cap \mathbf E_{\varepsilon_0}'' \vert ~ \mathbf X^{\lambda'}_\delta)=\mathbb P_2 (\mathbf E_{\varepsilon_0} \vert ~ \mathbf X^{\lambda'}_\delta) \mathbb P_2(\mathbf E''_{\varepsilon_0} \vert ~ \mathbf X^{\lambda'}_\delta) \geq (\varepsilon_1/2)^{\sum_{(x ,u) \in \mathbf Q} N'_{(x,u)}} c_1 \geq (\varepsilon_1/2)^{\varepsilon_0 \lambda' \sharp \mathbf W_\delta} c_1. \]
Observing that $-\varepsilon_0 \log(\varepsilon_1 /2) = -\varepsilon_0 \log(\varepsilon_0 {\kappa}_\delta^{-1}) \leq -\varepsilon_0 \log (\varepsilon_0) \in o(\sqrt{\varepsilon_0})$ finishes the proof.
\end{proof}
Now, with the notation $V(\nu)= \lbrace (x,u) \in \mathbf W_\delta \vert~\nu((x,u))=0 \rbrace$\index{zero sets of measures} from Section \ref{pajti}, we let
\[ \mathbf{E}^\ast_{\varepsilon_0} = \mathbf E_{\varepsilon_0} \cap \lbrace V(\mathbf L_\lambda^{\varrho'})=V(\mathbf L_{\lambda'}^{\varrho'} ) \rbrace.\]
Note that by $\Finite$, it holds (apart from a $\mathbb P_2$ nullset) that this event occurs if and only if all the quasi-empty sites are not only virtual, but also empty w.r.t. $\mathbf L_{\lambda}^{\varrho'}$. Our second sprinkling lemma is the following.
\begin{lem} \label{3.10}
For sufficiently small $\varepsilon_0 \in (0,1)$, there exists $\lambda_0 =\lambda_0(\varepsilon_0)$ such that for all $\lambda \geq \lambda_0$ we have
\[ \mathbb P_2(\mathbf{E}_{\varepsilon_0}^\ast \vert ~ \mathbf X^\lambda_\delta) \geq \exp(-\sqrt{\varepsilon_0}\lambda). \]
\end{lem}
\begin{proof}
We proceed similarly to the proof of \cite[Lemma 3.10]{cikk}, and we note that $ \mathbf{H}_{\varepsilon_0} \subseteq \mathbf{E}_{\varepsilon_0}$, where $ \mathbf{ H}_{\varepsilon_0}$ denotes the event that $N''_{(x,u)} \geq \varepsilon_0 \lambda'$ holds for all $(x,u) \in \mathbf W_\delta \setminus V(\mathbf L_\lambda^{\varrho'})$. Here if $\mu'^{\varrho'}((x,u))>0$, then $N''_{(x,u)}=\mathbf X^{\lambda'}_\delta((x,u))-\mathbf X^\lambda_\delta ((x,u))$ is independent of $\mathbf X^\lambda_\delta(x)$ and Poisson distributed with parameter $(\lambda'-\lambda){\mu'}^{\varrho'}((x,u))=2 {\mu'}^{\varrho'}((x,u)) \varepsilon_0 {\kappa_\delta}^{-1} \lambda > \varepsilon_0 \lambda' $. On the other hand, if $\mu'^{\varrho'}((x,u))=0$, then $(x,u) \in V(\mathbf L_\lambda^{\varrho'}) \cap V(\mathbf L_{\lambda'}^{\varrho'})$, $\mathbb P_2$-a.s. 
Thus, similarly to the proof of Lemma \ref{3.9}, there exists $c_1>0$ such that $\mathbb P_2(\mathbf H_{\varepsilon_0} \vert~\mathbf X^\lambda_\delta) \geq c_1$ holds $\mathbb P_2$-almost surely for every $\lambda \geq 1$. Hence, since conditioned on $\mathbf X^\lambda_\delta$, $\mathbf H_{\varepsilon_0}$ and $\lbrace V(\mathbf L_\lambda^{\varrho'})=V(\mathbf L_{\lambda'}^{\varrho'} ) \rbrace $ are independent, this implies that
\[ \mathbb P_2(\lbrace V(\mathbf L_\lambda^{\varrho'})=V(\mathbf L_{\lambda'}^{\varrho'} ) \rbrace  \cap \mathbf H_{\varepsilon_0} \vert~\mathbf X^\lambda_\delta)=\mathbb P_2(\mathbf H_{\varepsilon_0} \vert~\mathbf X^\lambda_\delta)\mathbb P_2(V(\mathbf L_\lambda^{\varrho'})=V(\mathbf L_{\lambda'}^{\varrho'} ) \vert~\mathbf X^\lambda_\delta).\]
Since the r.h.s. is bounded below by $\exp(-(\lambda'-\lambda) {\mu'}^{\varrho'}(\mathbf W_\delta)) c_1=\exp(-(\lambda'-\lambda) {\mu}^{\varrho_1}(W_\delta)) c_1$ and \[ \frac{1}{\lambda} \left( -(\lambda'-\lambda) {\mu}^{\varrho_1}(W_\delta) \right) =-2 \varepsilon_0 {\kappa}_\delta^{-1} {\mu}^{\varrho_1} (W_\delta) \in o(\sqrt{\varepsilon_0}), \] we conclude the proof. 
\end{proof}
\section{Proof of Proposition \ref{2.1}} \label{kettőegy}
This proof is based on the proof of \cite[Proposition 2.1]{cikk}. 
The random fadings give rise to additional terms and factors in the computations. These can easily be explained by considering the $\ell^1$ product norm as in the proof of Lemma \ref{3.5}.

First, for all $(x,u) \in \mathbf W_\delta$ with ${\nu}^{\varrho'}((x,u))=0$ we have $G((1\pm\varepsilon){\nu}^{\varrho'},~\boldsymbol \tau \circ \imath') ((x,u))=0$ and \\ $G(\nu,~ \boldsymbol \tau)^{\varrho'}((x,u))=0$, and hence the inequalities are trivially satisfied. Now, fix $\varepsilon >0$ and assume that there exists $(x,u) \in \mathbf W_\delta$ such that ${\nu}^{\varrho'}((x,u)) >0$. Let us write \begin{equation} \label{cé} C((\xi,u),~\nu) \in \lbrace D((\xi,u),~(o,F_o),~\nu),~D((o,F_o),~(\xi,u),~\nu),~R((\xi,u),~(o,F_o),~\nu),~R((o,F_o),~(\xi,u),~\nu) \rbrace \end{equation} for the different forms of communication, for $(\xi,u) \in \mathbf W$ and $\nu \in \mathcal{M}(\mathbf W)$. 

We first prove the upper bound. It suffices to find $\delta'=\delta'(\varepsilon) \in \mathbb B$ such that for all $\delta \in \mathbb B \cap (0,\delta')$ and all $C$ we have
\[ \sup\limits_{(\xi,s) \in \varrho'^{-1} ((x,u))} \tau_i({C}((\xi,s),~{\nu})) \leq \tau_i(\imath'(C((x,u),~(1+\varepsilon){\nu}^{\varrho'} ))) \]
for all $i \in \four$. Since $\tau_i$ are decreasing, it is enough to find $\delta'=\delta'(\varepsilon) \in \mathbb B$ such that $\delta \in \mathbb B \cap (0,\delta')$, $i \in \four$ and $(\xi,s) \in \varrho'^{-1}((x,u))$
\begin{equation} \label{Tizesbérc} {C}((\xi,s),~{\nu}) \geq \imath'(C((x,u),~(1+\varepsilon){\nu}^{\varrho'} )). \end{equation}
We first show that for sufficiently small $\delta$, for all $(\xi,s) \in \varrho'^{-1}((x,u))$ and $(\chi,b) \in \varrho'^{-1}((y,v))$ with $(x,u),~(y,v) \in \mathbf W_\delta$ we have
\begin{equation} \label{Tízenegyesbérc} {\SIR}((\xi,s),~(\chi,b),~{\nu}) \geq \imath'(\SIR((x,u),~(y,v),~(1+\varepsilon) {\nu}^{\varrho'} )). \end{equation}
Using the definition of $\SIR$, this is equivalent to showing that
\[ \frac{\ell(\vert x-y \vert) u ~ {\nu}(\ell(\vert \cdot- \chi\vert)\centerdot)}{\ell(\vert \xi-\chi \vert) s ~ {\nu}^{\varrho'}(\ell(\vert \cdot - y \vert)\centerdot)}-1 \leq \varepsilon. \]
The absolute value of the left hand side can be estimated as follows
\begin{align*}
\left| \frac{\ell(\vert x-y \vert) u ~ \nu (\ell(\vert \cdot- \chi \vert)\centerdot)}{\ell(\vert \xi -\chi \vert) s ~ \nu^{\varrho'}(\ell(\vert \cdot - y \vert)\centerdot)}-1  \right| &= \left| \frac{\ell(\vert x-y \vert) u ~ \nu(\ell(\vert \cdot- \chi \vert)\centerdot)-\ell(\vert \xi -\chi \vert) s ~ \nu^{\varrho'}(\ell(\vert \cdot - y \vert)\centerdot)}{\ell(\vert \xi-\chi \vert) s ~ \nu^{\varrho'}(\ell(\vert \cdot - y \vert)\centerdot)} \right| \\ &\leq \frac{\ell(\vert x-y \vert) u}{\ell(\vert \xi-\chi \vert) s} \left| \frac{\nu(\ell(\vert \cdot- \chi\vert)\centerdot)-\nu^{\varrho'}(\ell(\vert \cdot - y \vert)\centerdot)}{\nu^{\varrho'}(\ell(\vert \cdot - y \vert)\centerdot)}\right| \\&+ \left| \frac{\ell(\vert x-y \vert) u-\ell(\vert \xi-\chi \vert) s}{\ell(\vert \xi-\chi \vert) s} \right| \\ &\leq \frac{\ellmax \Fmax}{\ellmin^2 \Fmin^2} \frac{\sum_{(w,p) \in \mathbf W_\delta} \int\limits_{\varrho'^{-1}((w,p))} \left| \ell(\vert \upsilon-\chi \vert)q-\ell(\vert w-y \vert)p \right|~ \nu (\mathrm d \upsilon,\mathrm d q) }{\nu(\mathbf W)} \\ &+ \left(\frac{J_2 \Fmax}{\ellmin  \Fmin} \left| x-y-\xi+\chi \right|+\frac{\ellmax}{\ellmin \Fmin} \vert u-s \vert \right) \\ &\leq \frac{J_2 \Fmax \ellmax}{\ellmin^2 \Fmin^2} \sup_{\begin{smallmatrix} (w,p) \in \mathbf W_\delta \\ (\upsilon,q) \in \varrho'^{-1}((w,p)) \end{smallmatrix}} \left( J_2 \Fmax \vert \upsilon-\chi-w+y \vert + \ellmax \vert p-q \vert \right) \\ &+ \left(\frac{J_2 \Fmax}{\ellmin} \left| x-y-\xi+\chi \right|+\ellmax \vert u-s \vert \right) \\ &\leq \alpha_1' \sup_{(w,p) \in \mathbf W_\delta} \sup_{(\upsilon,q) \in \varrho'^{-1}((w,p))} \left( \left| \upsilon-w \right| + \left| p-q \right| \right) \leq \alpha_2' \delta, \numberthis  \label{Tízenkettő}
\end{align*}
where $\alpha_1', \alpha_2'$ are some constants depending also on the edge length $r$ of $W$. Since $g$ is increasing by assumption, \eqref{Tízenegyesbérc} implies
\begin{equation} \label{Tizenháárom} {D}((\xi,s),~(\chi,b),~{\nu}) \geq \imath'(D((x,u),~(y,v),~(1+\varepsilon) {\nu}^{\varrho'} ))
\end{equation}
for all $(\xi,s) \in \varrho'^{-1}((x,u))$ and $(\chi,b) \in \varrho'^{-1}((y,v))$. Now we show that for every $C$, the inequality \eqref{Tizesbérc} can be derived from the inequality \eqref{Tizenháárom}. Clearly, the direct up- and downlink cases \eqref{Tizesbérc} follow from \eqref{Tizenháárom} setting $(\chi,b)=(y,v)=(o,F_o)$ and $(\xi,s)=(x,u)=(o,F_o)$ respectively. For the relayed uplink case $C((\xi,u),~\nu)=R((\xi,u),~(o,F_o),~\nu)$, it is sufficient to prove \eqref{Tizesbérc} for the relaying component in $R((\xi,u),~(o,F_o),~\nu)$, since the direct communication part has already been verified. For this, we show that for all $(\xi,s) \in \varrho'^{-1}((x,u))$ we have
\[ {\nu}\text{-}\esssup\limits_{(\chi,b) \in \mathbf W} \Gamma((\xi,s),~(\chi,b),~(o,F_o),~{\nu}) \geq {\nu}^{\varrho'}\text{-}\esssup\limits_{(y,v) \in \mathbf W_\delta} \Gamma((x,u),~(y,v),~(o,F_o),~(1+ \varepsilon) {\nu}^{\varrho'}). \] 
Let us assume that the essential supremum on the right hand side is attained in $(y,v) \in \mathbf W_\delta$, where it is necessary that ${\nu}^{\varrho'}((y,v))>0$. Then it is sufficient to find $\delta'=\delta'(\varepsilon) \in \mathbb B$ such that for all $\delta \in \mathbb B \cap (0,\delta')$ and $(\xi,s) \in \varrho'^{-1}((x,u)),~(\chi,b) \in \varrho'^{-1}((y,v))$ we have
\begin{small} \[ \min \left\lbrace D((\xi,s),~(\chi,b),~{\nu}),~D((\chi,b),~(o,F_o),~{\nu}) \right\rbrace \geq \min \left\lbrace D((x,u),~(y,v),~(1+\varepsilon){\nu}^{\varrho'}),~ D((y,v),~(o,F_o),~(1+\varepsilon){\nu}^{\varrho'})   \right\rbrace, \] \end{small} 
\hspace{-8pt} and this can be done using \eqref{Tizenháárom}.

Similarly, in the case of relayed downlink communication $C((\xi,s),~\nu)=R((o,F_o),~(\xi,s),~\nu)$, using the same argument as in the previous case, we have to show 
\begin{small} \[ \min \left\lbrace D((o,F_o),~(\chi,b),~{\nu}),~D((\chi,b),~(\xi,s),~{\nu}) \right\rbrace \geq \min \left\lbrace D((o,F_o),~(y,v),~(1+\varepsilon){\nu}^{\varrho'}),~ D((y,v),~(x,u),~(1+\varepsilon){\nu}^{\varrho'})   \right\rbrace, \] \end{small} \hspace{-8pt}
for all $(\xi,s) \in \varrho'^{-1}((x,u)),~(\chi,b) \in \varrho'^{-1}((y,v))$ if $\mathbb B \ni \delta <\delta'$, with some threshold value $\delta'=\delta'(\varepsilon) \in \mathbb B$. But this can also be verified using \eqref{Tizenháárom}.

For the lower bound, by the above it suffices to find $\delta'=\delta'(\varepsilon) \in \mathbb B$ such that for all $\delta \in \mathbb B \cap (0,\delta')$, $(\xi,s) \in \varrho'^{-1}((x,u))$ and all $C$ we have
\begin{equation} \label{Tizennégyesbérc} {C}((\xi,s),~{\nu}) \leq \imath'(C((x,u),~(1-\varepsilon){\nu}^{\varrho'} )). \end{equation}
Again, we start with showing that for sufficiently small $\delta$, for all $(\xi,s) \in \varrho'^{-1}((x,u))$ and $(\chi,b) \in \varrho'^{-1}((y,v))$ with $(x,u) \in \mathbf W_\delta$ we have
\[ {\SIR}((\xi,s),~(\chi,b),~{\nu}) \leq \imath'(\SIR((x,u),~(y,v),~(1-\varepsilon) {\nu}^{\varrho'})), \]
which is equivalent to proving that \[ 1-\frac{\ell(\vert x-y \vert) u ~ {\nu}(\ell(\vert \cdot- \chi\vert)\centerdot)}{\ell(\vert \xi-\chi \vert) s ~ {\nu}^{\varrho'}(\ell(\vert \cdot - y \vert)\centerdot)} \leq \varepsilon. \]
But this follows from the estimate \eqref{Tízenkettő}. This implies
\begin{equation} \label{Tízenötösbérc} {D}((\xi,s),~(\chi,b),~{\nu}) \leq \imath'(D((x,u),~(y,v),~(1-\varepsilon) {\nu}^{\varrho'})).
\end{equation}
For the direct uplink and downlink cases in \eqref{Tizennégyesbérc} are implied by \eqref{Tízenötösbérc} by setting $(y,v)=(\chi,b)=(o,F_o)$ respectively $(x,u)=(\xi,s)=(o,F_o)$. Thus, for the relayed communication cases we have to prove \eqref{Tizennégyesbérc} only for the relaying component. For this purpose, in the direct uplink case $C((x,u),~(o,F_o))=R((x,u),~(o,F_o),~\nu)$ we show that for all $(\xi,s) \in \varrho'^{-1}((x,u))$ we have
\[ \nu\text{-} \esssup\limits_{(\chi,b) \in \mathbf W} \Gamma((\xi,s),~(\chi,b),~(o,F_o),~\nu) \leq \nu^{\varrho'}\text{-}\esssup\limits_{(y,v) \in \mathbf W_\delta} \Gamma((x,u),~(y,v),~(o,F_o),~(1-\varepsilon)\nu^{\varrho'}). \]
We note that for $(\xi,s) \in \varrho'^{-1}((x,u))$ we can write
\[ \nu\text{-} \esssup\limits_{(\chi,b) \in \mathbf W} \Gamma((\xi,s),~(\chi,b),~(o,F_o),~\nu)=\nu^{\varrho'}\text{-}\esssup\limits_{(y,v) \in \mathbf W_\delta} \left( \nu\text{-}\esssup\limits_{(\chi,b) \in \varrho'^{-1}((y,v))} \Gamma((\xi,s),~(\chi,b),~(o,F_o),~\nu) \right). \] 
Hence this essential supremum is attained, say in $(y,v) \in \mathbf W_\delta$, where it is necessary that $\nu((y,v))>0$. Then it suffices to find $\delta' \in \mathbb B$ such that for all $\delta \in \mathbb B \cap (0,\delta')$, $(\xi,s) \in \varrho'^{-1}((x,u))$ and $(\chi,b) \in \varrho'^{-1}((y,v))$ we have
\begin{small} \[ \min \left\lbrace D((\xi,s),~(\chi,b),~{\nu}),~D((\chi,b),~(o,F_o),~{\nu}) \right\rbrace \leq \min \left\lbrace D((x,u),~(y,v),~(1-\varepsilon){\nu}^{\varrho'}),~ D((y,v),~(o,F_o),~(1-\varepsilon){\nu}^{\varrho'})   \right\rbrace, \]\end{small} \hspace{-7pt} which can be done using \eqref{Tízenötösbérc}.

As for the relayed downlink communication $C((x,u),~\nu)=R((o,F_o),~(x,u),~\nu)$, we need to show
\begin{small} \[ \min \left\lbrace D((o,F_o),~(\chi,b),~{\nu}),~D((\chi,b),~(\xi,s),~{\nu}) \right\rbrace \leq \min \left\lbrace D((o,F_o),~(y,v),~(1-\varepsilon){\nu}^{\varrho'}),~ D((y,v),~(x,u),~(1-\varepsilon){\nu}^{\varrho'})   \right\rbrace \]\end{small} \hspace{-7pt} for all $(\xi,s) \in \varrho'^{-1}((x,u))$, $(\chi,b) \in \varrho'^{-1}((y,v))$ and sufficiently small $\delta$. But this again follows from \eqref{Tízenötösbérc}. This finishes the proof.
\section{Proof of Proposition \ref{2.2}} \label{kettőkettő}
We first establish a large deviation principle for the random variables $\left\lbrace \mathbf L_\lambda^{\varrho'} ((x,u)) \right\rbrace_{(x,u) \in \mathbf W_\delta}$ indexed by $\lambda > 0$, and then we can proceed analogously as in \cite[Section 4.2]{cikk}.
By Cramér's theorem\index{empirical measure}\index{Cramér's theorem} (Theorem \ref{nagycramér}), for each $(x,u) \in \mathbf W_\delta$, the empirical measures
\[ \mathbf L_{\lambda}^{\varrho'} ((x,u))=\frac{1}{\lambda} \sum_{X_i \in X^\lambda} \mathds 1 {\lbrace \varrho_1(X_i)=x \rbrace} \mathds 1 {\lbrace \varrho_2(F_{X_i})=u \rbrace} \]
satisfy a large deviation principle\index{large deviation principle} with good rate function\index{rate function}\index{relative entropy}\index{relative entropy!discrete}
\[ a \mapsto h(a \vert \mu^{\varrho'}((x,u))), \]
where $\mu^{\varrho'}((x,u))$ has the product form $\mu^{\varrho_1}(x)\mathbb P(\varrho_2(F_0)=u)$ in our case when the fading variable of a user does not depend on the spatial position of the user.

To prove this, we note that $\lbrace \mathbf X^\lambda_\delta((x,u)) \rbrace_{\lambda>0}$ is a homogeneous Poisson process on $(0,\infty)$ with intensity $\mu'^{\varrho'}((x,u))$. This implies that for $n \in \mathbb N$, the random variables $\mathbf X^1_\delta((x,u))$, $\mathbf X^2_\delta((x,u))-\mathbf X^1_\delta((x,u))$, $\ldots$, $\mathbf X^n_\delta((x,u))-\mathbf X^{n-1}_\delta((x,u))$ are i.i.d. Poisson random variables with common parameter $\mu'^{\varrho'}((x,u))$. First we show that the LDP holds for $\mathbf L_{n}^{\varrho'} ((x,u))$ w.r.t. the limit $n \to \infty$ with $n \in \mathbb N$, instead of the continuous limit $\lambda \to \infty$. Using the notions defined in Section \ref{ldpcske}, for $n \in \mathbb N$, the moment generating function of $\mathbf X^1_\delta((x,u))$ is given by
\begin{equation} \label{lözsandrka} \mathbb E_2[\exp(\alpha \mathbf X^1_\delta((x,u)))] =\sum_{k=0}^\infty e^{- \mu'^{\varrho'}((x,u))} \frac{(\mu'^{\varrho'}((x,u)))^k}{k!} e^{\alpha k}=\exp( \mu'^{\varrho'}((x,u)) e^{\alpha}-1). \end{equation}
The rate function for the LDP of $\lbrace \mathbf X^n_\delta((x,u)) \rbrace_{n \in \mathbb N}$ is the Fenchel--Legendre transform $\Lambda^\ast_{\mathbf X^1_\delta((x,u))}(\cdot)$ of the logarithm of the quantity \eqref{lözsandrka}. For fixed $y \in \mathbb R$, $\Lambda^\ast_{\mathbf X^1_\delta}(y)$ equals the supremum of the function
\[ f(\alpha)=\alpha y+\mu'^{\varrho'}((x,u)) (1-e^\alpha)\]
over $\alpha \geq 0$. The unique $\alpha$ with $f'(\alpha)=0$ is $\alpha_0=\log \frac{y}{\mu'^{\varrho'}((x,u))}$, and also $f''(\alpha_0) \leq 0$ holds, therefore the fact that $\lim_{\alpha \to \infty} f(\alpha)=-\infty$ implies that $f(\alpha_0)$ is the global maximum of $f$ on $[0,\infty)$. Thus, using the definition \eqref{diszkrétrelatíventrópia} of relative entropy, we have  \[ \Lambda_{\mathbf X^1_\delta((x,u))}^\ast(y)= y \log \frac{y}{ \mu'^{\varrho'}((x,u))}+\mu'^{\varrho'}((x,u)) \left(1-\frac{y}{ \mu'^{\varrho'}((x,u))} \right)=h(y \vert \mu'^{\varrho'}((x,u))), \]
as wanted. Now, for fixed $k \in \mathbb N$, the LDP for $\lambda=n/k \to \infty$ follows by analogous arguments for $\mathbf X^{1/k}_\delta((x,u))$, $\mathbf X^{2/k}_\delta((x,u))-\mathbf X^{1/k}_\delta((x,u))$, $\ldots$, $\mathbf X^{n/k}_\delta((x,u))-\mathbf X^{(n-1)/k}_\delta((x,u))$ instead of $\mathbf X^1_\delta((x,u))$, $\mathbf X^2_\delta((x,u))-\mathbf X^1_\delta((x,u))$, $\ldots$, $\mathbf X^n_\delta((x,u))-\mathbf X^{n-1}_\delta((x,u))$. This implies the LDP for $\lambda \to \infty$, $\lambda \in \mathbb Q$. 

Now, we can show that the desired LDP follows. This is equivalent to proving that for all sequence $\lbrace \lambda_n \rbrace$ in $\mathbb R$ with $\lim_{n \to \infty} \lambda_n=\infty$, we have the correct large deviation rate. For all $n \in \mathbb N$, let $\lbrace q_{n,k} \rbrace$ be a sequence of rational numbers converging to $\lambda_n$ from above. E.g. for the upper bound, for $\alpha \geq 0$ we estimate
\begin{align*} \limsup_{n \to \infty} \frac{1}{\lambda_n} \log \mathbb P_2(\mathbf L_{\lambda_n}^{\varrho'}((x,u)) \geq \alpha)  &= \limsup_{n \to \infty} \frac{1}{\lambda_n} \log \mathbb P_2(\lim_{k \to \infty} \mathbf L_{q_{n,k}}^{\varrho'}((x,u)) \geq \alpha, ~\forall k \in \mathbb N) \numberthis \label{Harmincnégyesbérc} \\ &= \limsup_{n \to \infty} \frac{1}{\lambda_n} \log \lim\limits_{\lambda \to \infty} \mathbb P_2(\mathbf L_{q_{n,k}}^{\varrho'}((x,u)) \geq \alpha) \numberthis \label{Harmincötösbérc} \\ &= \limsup_{n \to \infty} \limsup\limits_{k \to \infty} \frac{1}{q_{n,k}} \log \mathbb P_2(\mathbf L_{q_{n,k}}^{\varrho'}((x,u)) \geq \alpha) \\ &= \limsup\limits_{k \to \infty} \limsup_{n \to \infty}  \frac{1}{q_{n,k}} \log \mathbb P_2(\mathbf L_{q_{n,k}}^{\varrho'}((x,u)) \geq \alpha) \numberthis \label{Harminchatosbérc}  \\ &\leq \limsup\limits_{k \to \infty} -\Lambda_{\mathbf X^1_\delta((x,u))}^\ast(\alpha) \numberthis \label{Harminchetesbérc} = -\Lambda_{\mathbf X^1_\delta((x,u))}^\ast(\alpha), \end{align*} where in \eqref{Harmincnégyesbérc} we used that the homogeneous Poisson process has right-continuous paths (cf. \cite[Section 1.1]{lévyprocesses}), in \eqref{Harmincötösbérc} that these paths are also increasing, in \eqref{Harminchatosbérc} we used the continuity of measures. Finally, in \eqref{Harminchetesbérc} we were able to interchange the limits since our argumentation for rational intensities implies that the following two-dimensional limsup: $\limsup_{(n,k) \to \infty} \frac{1}{q_{n,k}} \log \mathbb P_2(\mathbf L_{q_{n,k}}^{\varrho'}((x,u)) \geq \alpha)$ is finite. This limit is equal to both sides of \eqref{Harminchetesbérc}, since for fixed $n$, $\limsup_{k \to \infty} \frac{1}{q_{n,k}} \log \mathbb P_2(\mathbf L_{q_{n,k}}^{\varrho'}((x,u)) \geq \alpha)=\frac{1}{\lambda_n} \log \mathbb P_2(\mathbf L_{\lambda_n}^{\varrho'}((x,u)) \geq \alpha)$ is finite (this holds also with $\lim$ instead of $\limsup$), and for fixed $k$, $\limsup_{n \to \infty} \frac{1}{q_{n,k}} \log \mathbb P_2(\mathbf L_{q_{n,k}}^{\varrho'}((x,u)) \geq \alpha) \leq -\Lambda_{\mathbf X^1_\delta((x,u))^\ast}(\alpha)$ is finite. Analogous arguments apply for the lower bound of the same probabilities, and thus also for upper respectively lower bounds for probabilities of $\mathbf L_{q_{n,k}}^{\varrho'}((x,u))$ taking place in more complex open respectively closed sets. Thus, we have obtained the LDP for $\lambda \to \infty$.
Therefore, Proposition \ref{4.2.7} implies that the independent random variables $\left\lbrace \mathbf L_\lambda^{\varrho'} ((x,u)) \right\rbrace_{(x,u) \in \mathbf W_\delta}$ satisfy an LDP with good rate function
\[ \lbrace a_{x,u} \rbrace_{(x,u) \in \mathbf W_\delta}  \mapsto \sum_{(x,u) \in \mathbf W_\delta} h(a_{x,u} \vert \mu^{\varrho'}((x,u))). \]
Now we turn to the proof of Proposition \ref{2.2}.
\begin{flushleft} \emph{Upper bound} \end{flushleft}
According to Remark \ref{3.7,8}, for every $(x,u) \in \mathbf W_\delta$ the map \begin{small}
\[ {\nu} \mapsto \left( \tau_1({R}((x,u),~(o,F_o),~\alpha {\nu}),~\tau_2({D}((x,u),~(o,F_o),~\alpha {\nu}),~\tau_3({R}(o,F_o),~((x,u),~\alpha {\nu}),~\tau_4({D}(o,F_o),~((x,u),~\alpha {\nu})\right) \] \end{small}
\hspace{-8pt} is u.s.c. Thus, also the map ${\nu} \mapsto G(\alpha {\nu},~ \boldsymbol \tau)$ is u.s.c. Another application of \cite[Lemma 3.7]{cikk} yields that $F(G(\alpha {\nu},~\boldsymbol \tau))$ is u.s.c., and therefore the upper bound in Proposition \ref{2.2} immediately follows from Varadhan's lemma for the upper bound (Lemma \ref{varadhanupper}\index{Varadhan's lemmas}). 

\begin{flushleft} \emph{Lower bound} \end{flushleft}
In contrast, it follows from our argumentation before Lemma \ref{3.6} that the map $\nu \mapsto G(\nu,~ \boldsymbol \tau)$ is not l.s.c., similarly to the fading-free case. Hence, proving the lower bound requires a substantial amount of work. Therefore, first we approximate the map $\nu \mapsto G(\nu,~ \boldsymbol \tau)$ by l.s.c. functions. We will see that the cost of these approximations is negligible on the exponential scale. More precisely, for the uplink, we define the following approximating measures
\[ G_\varepsilon(\nu,~\tau,~\text{up})(A)=\sum\limits_{(x,u) \in A} \tau( R_\varepsilon ((x,u),~(o,F_o),~{\nu}) ), \quad  A \subseteq \mathbf W_\delta, \] 
where \[  R_\varepsilon ((x,u),~(o,F_o),~{\nu})=\max \left\lbrace D((x,u),~(o,F_o),~{\nu}),~\max_{(y,v) \in \mathbf{W}_\delta} \Gamma_\varepsilon ((x,u),~(y,v),~(o,F_o),~{\nu}) \right\rbrace \]
is defined using \[ \Gamma_\varepsilon ((x,u),~(y,v),~(o,F_o),~{\nu}) = \min \lbrace 1, \varepsilon^{-1} {\nu}((y,v)) \rbrace~ \Gamma ((x,u),~(y,v),~(o,F_o),~{\nu}). \] Since by Lemma \ref{3.6} the lower semicontinuity of $\nu \mapsto G(\nu,~ \boldsymbol \tau)$ can only be obstructed by the relaying component, we easily see that for any $\varepsilon>0$, all $\Gamma_\varepsilon$ are l.s.c., and \[ \Gamma_\varepsilon ((x,u),~(y,v),~(o,F_o),~{\nu}) \leq \mathds 1\lbrace  {\nu}((y,v)))>0 \rbrace \Gamma ((x,u),~(y,v),~(o,F_o),~{\nu}), \] where equality holds if and only if ${\nu}((y,v)) \in \lbrace 0 \rbrace \cup \left[ \varepsilon, \infty \right) $. 

Similarly, for the direct uplink, the downlink and the direct downlink we introduce the approximating empirical measures $G_\varepsilon(\nu,~\tau,~\text{up-dir}),~G_\varepsilon(\nu,~\tau,~\text{do}),~G_\varepsilon(\nu,~\tau,~\text{do-dir})$ respectively, and we put
\[ G_\varepsilon(\nu,~\boldsymbol \tau)=\left(G_\varepsilon(\nu,~\tau_1,~\text{up}),~G_\varepsilon(\nu,~\tau_2,~\text{up-dir}),~G_\varepsilon(\nu,~\tau_3,~\text{do}),~G_\varepsilon(\nu,~\tau_4,~\text{do-dir})\right). \]
Now we formalize the approximation property under the event $\mathbf E_{\varepsilon_0} \cap \mathbf E_{\varepsilon_0}'$ from Section \ref{sprinkle}. 
\begin{lem} \label{4.1}
Let $0<\alpha_{-}<\alpha<2$ and $\tau_i: [0,\infty) \to [0,\infty),~i \in \four$ be decreasing measurable functions such that $\tau_i(\gamma)=0$ if $\gamma\geq \tilde{c}_+$. Then, for every sufficiently small $\varepsilon_0>0$ there exists $\lambda_0=\lambda_0(\varepsilon_0)$ such that for $\lambda \geq \lambda_0$, the following hold $\mathbb P_2$-a.s., for every $(x,u) \in \mathbf W_\delta$,
\[ \mathds 1\lbrace \mathbf E_{\varepsilon_0} \cap \mathbf E_{\varepsilon_0}' \rbrace R((x,u),~(o,F_o),~\alpha \mathbf L^{\varrho'}_{\lambda}) \leq  R_{\varepsilon_0 \alphaminus} ((x,u),~(o,F_o),~\alphaminus \mathbf L^{\varrho'}_{\lambda'}), \] 
and
\[ \mathds 1\lbrace \mathbf E_{\varepsilon_0} \cap \mathbf E_{\varepsilon_0}' \rbrace R((o,F_o),~(x,u),~\alpha \mathbf L^{\varrho'}_{\lambda}) \leq  R_{\varepsilon_0 \alphaminus} ((o,F_o),~(x,u),~\alphaminus \mathbf L^{\varrho'}_{\lambda'}) \] 
where $\lambda'=(1+\varepsilon_1) \lambda = (1+2 \varepsilon_0 {\kappa}_\delta^{-1}) \lambda$. In particular,
\[ \mathds 1\lbrace \mathbf E_{\varepsilon_0} \cap \mathbf E_{\varepsilon_0}' \rbrace G_{\varepsilon_0 \alphaminus}(\alphaminus \mathbf L_{\lambda'}^{\varrho'},~\boldsymbol \tau) \leq \alpha G(\mathbf L_{\lambda}^{\varrho'},~\boldsymbol \tau). \] 
\end{lem}
\begin{proof}
First, with the notation of Section \ref{sprinkle}\index{sprinkling construction}, under the event $\mathbf E'_{\varepsilon_0}$ we have $\frac{\mathbf X^{\lambda'}_\delta(\mathbf W_\delta)}{\mathbf X^{\lambda}_\delta (\mathbf W_\delta)}=\frac{ X^{\lambda'}_\delta( W_\delta)}{X^{\lambda}_\delta (W_\delta)} \leq \frac{1}{1-\varepsilon_2},$ which converges to 1 as $\varepsilon_0 \to 0$. Thus, for sufficiently small $\varepsilon_0$, we have $\left( \frac{1}{1-\varepsilon_2}-1 \right) \leq \frac{ \left( 1-\frac{\alphaminus}{\alpha} \right) \ellmin \Fmin }{\frac{\alphaminus}{\alpha} \ellmax \Fmax}$, and hence the part (v) of Lemma \ref{3.1} implies that $D((x,u),~(y,v),~\alpha \mathbf L^{\varrho'}_{\lambda} ) \leq D((x,u),~(y,v),~\alphaminus \mathbf L^{\varrho'}_{\lambda'} )$ holds for all $(x,u),~(y,v) \in \mathbf W_\delta$.

Now, it suffices to show that under the event $\mathbf E_{\varepsilon_0} \cap \mathbf E'_{\varepsilon_0}$
\[ \Gamma((x,u),~(y,v),~(z,w),~\alpha \mathbf L^{\varrho'}_{\lambda} ) \leq \min \lbrace 1, \varepsilon_0^{-1} \mathbf L^{\varrho'}_{\lambda'}((y,v)) \rbrace \Gamma((x,u),~(y,v),~(z,w), ~\alphaminus \mathbf L^{\varrho'}_{\lambda'}) \]
holds for all $(x,u),~(y,v),~(z,w) \in \mathbf W_\delta$ with $\mathbf L^{\varrho'}_{\lambda}((y,v))>0$. We claim that for such $(y,v)$, $\mathbf L^{\varrho'}_{\lambda'} ((y,v)) \geq \varepsilon_0$ holds under the event $\mathbf E_{\varepsilon_0}$. Indeed, otherwise
$ \varepsilon_0 >\mathbf L^{\varrho'}_{\lambda'} ((y,v))$,
and thus by the definition of $\mathbf E_{\varepsilon_0}$, we have $ 0=\mathbf L_{\lambda}^{\varrho'}((y,v)).$ Now we conclude by applying the inequality for $D$.
\end{proof}
Our next lemma shows that Lemma \ref{4.1} implies closeness in the exponential scale.
\begin{lem} \label{4.2}
Let $0<\alphaminus<\alpha<2$. Let $\tau_i:~[0,\infty) \to [0,\infty),~i \in \four$ and $F:~\mathcal{M}(\mathbf W_\delta)^4 \to [-\infty,\infty)$ be measurable functions such that $\tau_i$ are increasing and $F$ is decreasing. Furthermore, assume that $\tau_i(\gamma)=0$ for every $i \in \four$ if $\gamma \geq \tilde{c}_+$, and that $F$ maps the vector of zero measures to $-\infty$. Then, for sufficiently small $\varepsilon_0 \in (0,1)$, there exists $\lambda_0=\lambda_0(\varepsilon_0)$ such that for all $\lambda \geq \lambda_0$ we have
\[ \mathbb E_2 \exp \left( \lambda F(G(\alpha\mathbf L^{\varrho'}_\lambda,~\boldsymbol \tau)) \right) \geq \exp(-\sqrt{\varepsilon_0} \lambda) \mathbb E_2 \exp \left( \lambda F(G_{\varepsilon_0 \alphaminus}(\alphaminus \mathbf L^{\varrho'}_{\lambda'} ,~\boldsymbol \tau)) \right).\]
\end{lem}
\begin{proof}
First, since $F$ is increasing, Lemma \ref{4.1} implies that
\[ \mathbb E_2 \exp \left( \lambda F(G(\alpha\mathbf L^{\varrho'}_\lambda,~\boldsymbol \tau)) \right) \geq \mathbb E_2 \left( \mathds 1 {\lbrace \mathbf E_{\varepsilon_0} \cap \mathbf E'_{\varepsilon_0} \rbrace} \exp ( \lambda F(G_{\varepsilon_0 \alphaminus}(\alphaminus \mathbf L^{\varrho'}_{\lambda'} ,~\boldsymbol \tau))) \right).\] 
Note that if $ L^{\varrho_1}_{\lambda'} (W_\delta) < \beta'_o/2$, then by the definition of the QoS quantities in Section \ref{Anfang}, we have $G_{\varepsilon_0 \alphaminus}(\alphaminus \mathbf L^{\varrho'}_{\lambda'} ,~\boldsymbol \tau)=0$. Hence, using the assumption that $F$ maps the vector of zero measures to $-\infty$, using Lemma \ref{3.9} we deduce that
\begin{align*} \mathbb E_2 \left( F(G(\alpha\mathbf L^{\varrho'}_\lambda,~\boldsymbol \tau))  \right) &\geq \mathbb E_2 \left( \mathbb P_2 (\mathbf E_{\varepsilon_0} \cap \mathbf E'_{\varepsilon_0} \vert \mathbf X^{\lambda'}_{\delta} ) \exp \left( \lambda F(G_{\varepsilon_0 \alphaminus}(\alphaminus \mathbf L^{\varrho'}_{\lambda'} ,~\boldsymbol \tau)) \right) \right) \\ &\geq \exp(-\sqrt{\varepsilon_0} \lambda) \mathbb E_2 \left( \mathds 1\lbrace L^{\varrho_1}_{\lambda'} (W_\delta) \geq \beta'_0/2 \rbrace \exp \left( \lambda F(G_{\varepsilon_0 \alphaminus}(\alphaminus \mathbf L^{\varrho'}_{\lambda'} ,~\boldsymbol \tau)) \right) \right) \\& \geq \exp(-\sqrt{\varepsilon_0} \lambda) \mathbb E_2 \left(\exp \left( \lambda F(G_{\varepsilon_0 \alphaminus}(\alphaminus \mathbf L^{\varrho'}_{\lambda'} ,~\boldsymbol \tau)) \right) \right) ,
\end{align*}
as required.
\end{proof}
Now we can prove the lower bound of Proposition \ref{2.2} similarly to the lower bound of \cite[Proposition 2.2]{cikk}. First, note the map ${\nu} \mapsto G_{\varepsilon_0}(\nu,~ \boldsymbol \tau)$ is continuous. This is clear by the essential bounds of $\varrho_2(F_0)$, the upper semicontinuity of ${\nu} \mapsto G(\nu,~ \boldsymbol \tau)$, and the lower semicontinuity of ${\nu} \mapsto G_\varepsilon(\nu,~ \boldsymbol \tau)$ that follows from the definition of $\Gamma_\varepsilon$. Thus, for $0<\alphaminus<\alpha$, combining Lemma \ref{2.2} with Varadhan's lemma\index{Varadhan's lemmas} for the lower bound (Lemma \ref{varadhanlower}) shows that
\begin{align*}
\liminf_{\lambda \to \infty} \frac{1}{\lambda} \log \mathbb E_2 \left( F(G(\alpha\mathbf L^{\varrho'}_\lambda,~\boldsymbol \tau))  \right) \geq -\sqrt{\varepsilon_0} + \liminf_{\lambda \to \infty} \frac{1}{\lambda} \log \mathbb E_2 \exp \left( \lambda F(G_{\varepsilon_0 \alphaminus}(\alphaminus \mathbf L^{\varrho'}_{\lambda'} ),~\boldsymbol \tau) \right) \\ \geq -\sqrt{\varepsilon_0} - (1+2  \varepsilon_0 {\kappa}_{\delta}^{-1}) \inf\limits_{{\nu} \in \mathcal{M} (\mathbf W_\delta)} \left\lbrace h({\nu} \vert {\mu}'^{\varrho'}) - F(G_{ \varepsilon_0 \alphaminus}(\alphaminus {\nu},~\boldsymbol \tau)) \right\rbrace. 
\end{align*}
Furthermore, using that $\tau_1$ is decreasing, we have
\[ \tau_1 \left(  R((x,u),~(o,F_o),~\alphaminus \nu) \right) \leq \tau_1 \left(  R_{\varepsilon_0 \alphaminus} ((x,u),~(o,F_o),~\alphaminus {\nu}) \right) \]
for all $(x,u) \in \mathbf W_\delta$ and ${\nu} \in \mathcal{M}(\mathbf W_\delta)$. Similarly for the remaining communication cases. Hence,
\[ G_{\varepsilon_0 \alphaminus}(\alphaminus {\nu},~ \boldsymbol \tau) \geq G(\alphaminus \nu,~ \boldsymbol \tau), \]
and sending $\varepsilon_0$ to zero yields
\[ \liminf_{\lambda \to \infty} \frac{1}{\lambda} \log \mathbb E_2 \left(\exp(\lambda F(G(\alpha\mathbf L^{\varrho'}_\lambda,~\boldsymbol \tau)) ) \right) \geq -\inf\limits_{{\nu} \in \mathcal{M}(\mathbf W_\delta)} \left\lbrace h({\nu} \vert {\mu}'^{\varrho'}) - F(G(\alphaminus {\nu},~\boldsymbol \tau)) \right\rbrace.  \]
Finally, it remains to prove that 
\[ \inf\limits_{{\nu} \in \mathcal{M}(\mathbf W_\delta)} \left\lbrace h(\alpha^{-1} {\nu} \vert {\mu}'^{\varrho'}) - F(G(\nu,~\boldsymbol \tau)) \right\rbrace \geq \limsup_{\alphaminus \uparrow \alpha} \inf\limits_{{\nu} \in \mathcal{M}(\mathbf W_\delta)} \left\lbrace h(\alphaminus^{-1} {\nu} \vert {\mu}'^{\varrho'}) - F(G(\nu,~\boldsymbol \tau)) \right\rbrace. \] In order to prove this claim, consider ${\nu} \in \mathcal{M}(\mathbf W_\delta)$. If ${\nu}$ is not absolutely continuous with respect to ${\mu}^{\varrho'}$, then the left-hand side is infinite and there is nothing to prove. Otherwise, Lemma \ref{3.11} implies that $\lim\limits_{\alphaminus \uparrow \alpha} h(\alphaminus^{-1} {\nu} \vert {\mu}'^{\varrho'})=h(\alpha^{-1} {\nu} \vert {\mu}'^{\varrho'})$, as required. The proposition follows by putting $\alpha=1$.
\section{Proof of Theorem \ref{1.1}}  \label{egyegy}
Analogously to the proof of \cite[Theorem 1.1]{cikk}, after having established Proposition \ref{2.1} and Proposition \ref{2.2}, the proof of Theorem \ref{1.1} is reduced to a result on the behaviour of the rate functions in Proposition \ref{2.2} as $\delta \downarrow 0$.
\begin{lem} \label{5.1}
Let $F: \mathcal{M}(\mathbf W) \to [-\infty,\infty)$ and $\tau:\mathcal{B}' \to [0,\infty)$, $i \in \four$ be measurable functions that are respectively increasing and decreasing, and assume that $F$ is $\delta_0$-discretized for some $\delta_0 \in \mathbb B$ and bounded from above. Then,
\[ \lim_{\varepsilon \to 0} \limsup_{\begin{smallmatrix} \delta \to 0 \\ \delta \in \mathbb B \end{smallmatrix}} \inf_{\nu \in \mathbf W_\delta} \left\lbrace h({\nu} \vert {\mu}'^{\varrho'}) - F (G(((1-\varepsilon) {\nu})^{\imath'},~\boldsymbol \tau)) \right\rbrace \leq \inf_{{\nu} \in \mathcal{M}(\mathbf W)} \left\lbrace h({\nu} \vert {\mu}') - F(G(\nu,~\boldsymbol \tau) ) \right\rbrace \]
and
\[ \lim_{\varepsilon \to 0} \liminf_{\begin{smallmatrix} \delta \to 0 \\ \delta \in \mathbb B \end{smallmatrix}} \inf_{\nu \in \mathbf W_\delta} \left\lbrace h({\nu} \vert {\mu}'^{\varrho'}) - F (G(((1+\varepsilon) {\nu})^{\imath'},~\boldsymbol \tau)) \right\rbrace \geq \inf_{{\nu} \in \mathcal{M}(\mathbf W)} \left\lbrace h({\nu} \vert {\mu}') - F(G(\nu,~\boldsymbol \tau) ) \right\rbrace. \]
\end{lem}
Now, first we show how Lemma \ref{5.1} can be used to complete the proof of Theorem \ref{1.1}, and then we provide a proof for the lemma.

\emph{Proof of Theorem \ref{1.1}.} We only prove the lower bound, the proof of the upper bound is analogous. Let $\varepsilon \in (0,1)$ be arbitrary. Then, Propositions \ref{2.1} and \ref{2.2} show that for all sufficiently small $\delta \in \mathbb B$
\begin{align*}
\liminf_{\lambda \to \infty} \frac{1}{\lambda} \log \mathbb E_2 \exp \left(\lambda F(G(\mathbf L_\lambda,~\boldsymbol \tau)) \right) 
&= \liminf_{\lambda \to \infty} \frac{1}{\lambda} \log \mathbb E_2 \exp \left( \lambda F(((G(\mathbf L_\lambda,~\boldsymbol \tau))^{\varrho'})^{\imath'}) \right) \\ 
&\leq \liminf_{\lambda \to \infty} \frac{1}{\lambda} \log \mathbb E_2 \exp \left( \lambda F (G((1-\varepsilon) \mathbf L^{\varrho'}_\lambda,~\boldsymbol \tau \circ \imath' )^{\imath'})\right) \\ &
\geq - \inf_{{\nu} \in \mathcal{M}(\mathbf W_\delta)} \left\lbrace h({\nu} \vert {\mu}'^{\varrho'})-F (G((1-\varepsilon) \nu,~\boldsymbol \tau \circ \imath' )^{\imath'}) \right\rbrace \\ &
= -\inf_{{\nu} \in \mathcal{M}(\mathbf W_\delta)} \left\lbrace h({\nu} \vert {\mu}'^{\varrho'} ) - F(G(((1-\varepsilon){\nu})^{\imath'},~\boldsymbol \tau )) \right\rbrace. 
\end{align*}
Hence, Lemma \ref{5.1} implies that
\[ \liminf_{\lambda \to \infty} \frac{1}{\lambda} \log \mathbb E_2 \exp \left( \lambda F(G(\mathbf L_\lambda,~\boldsymbol \tau)) \right) \geq - \inf_{{\nu} \in \mathcal{M}(\mathbf W)} \left\lbrace h({\nu} \vert {\mu}') - F(G(\nu,~\boldsymbol \tau)) \right\rbrace, \]
as required. \hspace{378pt} $\square$
\begin{flushleft} Now we prove Lemma \ref{5.1}. \end{flushleft}
\emph{Proof of Lemma \ref{5.1}}. First, we consider the upper bound. We need to show that for arbitrary $\varepsilon_0 \in (0,1)$ and ${\nu}_0 \in \mathcal{M}(\mathbf W)$ we have that
\[ \limsup_{\delta \to 0} \inf_{{\nu} \in \mathcal{M}(\mathbf W_\delta)} \left\lbrace h({\nu} \vert {\mu}'^{\varrho'})-F(G(((1-\varepsilon){\nu})^{\imath'},~\boldsymbol \tau )) \right\rbrace \leq \varepsilon_0 + h({\nu}_0 \vert {\mu}') - F(G(\nu_0,~\boldsymbol \tau)) \]
holds provided that $\varepsilon_0 \in (0,1)$ is sufficiently small. Since $F$ is bounded from above, without loss of generality we can assume that ${\nu}_0$ is absolutely continuous with respect to ${\mu}'$.

First, Proposition \ref{2.1} shows that if $\delta$ is sufficiently small, then $(G({\nu}_0,~ \boldsymbol \tau))^{\varrho'} \leq G((1+\varepsilon){\nu}_0^{\varrho'},~\boldsymbol \tau \circ \imath')$. In particular, since $F$ is $\delta$-discretized, using the argumentation of Lemma \ref{deltadiscretization} we have
\[ F(G({\nu}_0,~\boldsymbol \tau))=F \left(((G( {\nu}_0,~\boldsymbol \tau)^{\varrho'})^{\imath'} \right) \leq F \left(G (((1+\varepsilon) {\nu}^{\varrho'}_0)^{\imath'},~\boldsymbol \tau) \right). \]
Thus, putting $1+\varepsilon'=(1+\varepsilon)/(1-\varepsilon)$, it suffices to show that
\[ \limsup_{\delta \to 0} h((1+\varepsilon'){\nu}_0^{\varrho'} \vert {\mu}'^{\varrho'}) \leq \varepsilon_0 + h({\nu}_0 \vert {\mu}') \]
holds for all sufficiently small $\varepsilon$. By Jensen's inequality, we have $h({\nu}_0 \vert {\mu}') \geq h({\nu}_0^{\varrho'} \vert {\mu}'^{\varrho'})$. Hence, by Corollary \ref{3.12}, we can estimate
\[ h((1+\varepsilon') {\nu}_0^{\varrho'} \vert {\mu}'^{\varrho'}) - h({\nu}_0^{\varrho'} \vert {\mu}'^{\varrho'}) \leq 3 \varepsilon' {\mu}'^{\varrho'} (\mathbf W_\delta)+3\varepsilon' h({\nu}_0 \vert {\mu}'^{\varrho}) \leq  3 \varepsilon' {\mu}'(\mathbf W)+3 \varepsilon' h({\nu}_0 \vert {\mu}'). \]
Since this upper bound tends to zero as $\varepsilon$ tends to zero, we conclude the proof.

Next, we consider the lower bound. We need to show that for arbitrary $\varepsilon_0>0$, we have that
\[ \liminf_{\delta \to 0} \inf_{{\nu} \in \mathcal{M}(\mathbf W_\delta)} \left\lbrace h({\nu} \vert {\mu}'^{\varrho'})-F(G(((1+\varepsilon){\nu})^{\imath'},~\boldsymbol \tau )) \right\rbrace \geq - \varepsilon_0 + \inf_{{\nu} \in \mathcal{M}(\mathbf W)}  \left\lbrace h({\nu} \vert {\mu}') - F(G({\nu},~ \boldsymbol \tau)) \right\rbrace \]
holds if $\varepsilon \in (0,1)$ is sufficiently small. First, for all $\varepsilon \in (0,1)$ we can choose a suitable sequence $\lbrace \delta_k \rbrace_{k \geq 1}$ in $\mathbb B$ such that $\lim\limits_{k \to \infty} \delta_k=0$ and such that the $\liminf\limits_{\delta \to 0}$ above is replaced by $\lim\limits_{\delta_k \to 0}$. Furthermore, for $\varepsilon \in (0,1)$ and $k \geq 1$ we choose ${\nu}_{k,\varepsilon} \in \mathcal{M} (\mathbf W_{\delta_k})$ such that
\[ h({\nu}_{k, \varepsilon} \vert {\mu}'^{\varrho'_{\delta_k}}) - F(G(((1+\varepsilon){\nu}_{k,\varepsilon})^{\imath'},\boldsymbol \tau) \leq \varepsilon_0/2 + \inf_{{\nu} \in \mathcal{M}(\mathbf W_\delta)} \left\lbrace h({\nu} \vert {\mu}'^{\varrho'_{\delta_k}} )-F(G((1+\varepsilon){\nu})^{\imath'} ,~\boldsymbol \tau)) \right\rbrace. \] 
Hence, it remains to show that
\[ \liminf_{\varepsilon \to 0} \liminf_{k \to \infty} h({\nu}_{k, \varepsilon} \vert {\mu}'^{\varrho'_{\delta_k}}) - F(G(((1+\varepsilon){\nu}_{k,\varepsilon})^{\imath'},~\boldsymbol \tau)  \geq \inf_{{\nu} \in \mathcal{M}(\mathbf W)}  \left\lbrace h({\nu} \vert {\mu}') - F(G(\nu,~\boldsymbol \tau)) \right\rbrace.   \]
In particular, we may assume that ${\nu}_{k,\varepsilon}$ is absolutely continuous with respect to ${\nu}'^{\varrho'_{\delta_k}}$. Then, we define ${\nu}'_{k,\varepsilon} \in \mathcal{M}(\mathbf W)$ by
\[ {\nu}'_{k,\varepsilon} (\cdot) = (1+2 \varepsilon) \sum_{(x,u) \in \mathbf W_{\delta_k}} \frac{{\nu}_{k,\varepsilon}((x,u))}{{\mu}'^{\varrho'_{\delta_k}}((x,u))} {\mu}'({\varrho'_{\delta_k}}^{-1} ((x,u)) \cap \cdot),  \]
so that $h({\nu}'_{k, \varepsilon} \vert {\mu}')=h({{\nu}'}_{k,\varepsilon}^{\varrho'_{\delta_k}} \vert {{\mu}'}^{\varrho'_{\delta_k}})$. Then, Proposition \ref{2.1} implies that
\[ (G(\nu'_{k,\varepsilon},~\boldsymbol \tau))^{\varrho'_{\delta_k}} \geq G(((1-\varepsilon'')(1+2\varepsilon) {\nu}_{k,\varepsilon}),~\boldsymbol \tau \circ \imath' )=G(((1+\varepsilon){\nu}_{k,\varepsilon}),~\boldsymbol \tau \circ \imath') \]
for all sufficiently small $\delta_k$ where $1-\varepsilon''=(1+\varepsilon)(1+2\varepsilon)^{-1}$. Also, by Corollary \ref{3.12} we have
\begin{equation} \label{utcsóbecslés1} h(({\nu}'_{k,\varepsilon})^{\varrho'_{\delta_k}} \vert {{\mu}'}^{\varrho'_{\delta_k}})-h({\nu}_{k,\varepsilon} \vert {{\mu}'}^{\varrho'_{\delta_k}} ) \leq 6 \varepsilon h({\nu}_{k,\varepsilon} \vert {{\mu}'}^{\varrho'_{\delta_k}}) +6 \varepsilon {\mu}'(\mathbf W). \end{equation}
Since $F$ is bounded from above, we have that if
\[ \liminf_{\varepsilon \to 0} \liminf_{k \to \infty} h({\nu}_{k,\varepsilon} \vert {{\mu}'}^{\varrho'_{\delta_k}} ) = \infty, \] 
then there is nothing to show. Otherwise, it follows from \eqref{utcsóbecslés1} that \[ \liminf_{\varepsilon \to 0} \liminf_{k \to \infty} h(({\nu}'_{k,\varepsilon})^{\varrho'_{\delta_k}} \vert {{\mu}'}^{\varrho'_{\delta_k}})-h({\nu}_{k,\varepsilon} \vert {{\mu}'}^{\varrho'_{\delta_k}} ) \leq 0, \] and hence, using that $F$ is increasing, we have
\[ \liminf_{\varepsilon \to 0} \liminf_{k \to \infty} h({\nu}_{k, \varepsilon} \vert {\mu}'^{\varrho'_{\delta_k}}) - F(G(((1+\varepsilon){\nu}_{k,\varepsilon})^{\imath'},~\boldsymbol \tau)) \geq  \liminf_{\varepsilon \to 0 } \liminf_{k \to \infty} h({\nu}'_{k,\varepsilon} \vert {\mu}')-F(G({\nu}'_{k,\varepsilon},~\boldsymbol \tau)), \] 
as required. \hspace{375pt} $\square$
\section{Proof of Corollaries \ref{1.2} and \ref{1.3}} \label{egykettőegyhárom}
First, defining the maps $F_{\mathbf b}: \mathcal{M} (\mathbf W) \to [-\infty,\infty)$ and $\boldsymbol \tau_{\mathbf{c}}:~[0,\infty)^4 \to [0,\infty)^4$ as in \eqref{Ffading} and \eqref{taufading}, we see that the maps ${\nu} \mapsto F_{\mathbf{b}} ({\nu}^{\imath'})$ and $\boldsymbol{\tau}_{\mathbf{c}} \circ \imath'$ are l.s.c. on $\mathcal{M}(\mathbf W_\delta)$ and $[0,\infty)^4$, respectively. We also recall that $F_{\mathbf b}$ is $\delta$-discretized for all $\delta \in \mathbb B$. Hence, by Theorem \ref{1.1}, only the upper bound needs a proof. In the following, we assume that $\mathbf b \geq 0$, which we can do without loss of generality, since negative coordinates of $\mathbf b$ mean no constraints on the corresponding component of $G(\mathbf L_\lambda,~\boldsymbol \tau_{\mathbf{c}})$.

Along this section we follow the proceed as in \cite[Section 6]{cikk}. We first derive the upper bound in Corollary \ref{1.2} in the discretized model, for fixed $\delta \in \mathbb B$ such that $\kappa_\delta \leq 1$. 
\begin{prop} \label{6.1}
Let $0<\alpha<2$, $\mathbf b \geq 0$ and $\mathbf c \in [0,\tilde{\mathbf c}_+)$. Then,
\[ \limsup_{\lambda \to \infty} \frac{1}{\lambda} \log \mathbb P(G((\alpha \mathbf L_\lambda^{\varrho'}),~ \boldsymbol \tau_{\mathbf{c}} \circ \imath') (\mathbf W_\delta) > \mathbf b) \leq - \inf\limits_{\begin{smallmatrix} {\nu} \in \mathcal{M}(\mathbf W_\delta) \\ G(\alpha {\nu},~ \boldsymbol \tau_{\mathbf{c}} \circ \imath') (\mathbf W_\delta) >\mathbf b \end{smallmatrix}} h({\nu} \vert {\mu}'^{\varrho}). \]
\end{prop} 
Since the maps ${\nu} \mapsto F_{\mathbf b}({\nu}^{\imath'})$ and $\boldsymbol \tau_{\mathbf{c}} \circ \imath'$ are not upper semicontinuous\footnote{ See the argumentation about the sensitivity of the relayed QoS of a given user w.r.t. its neighbours, before Lemma \ref{3.6}.}, we cannot use Proposition \ref{2.2} directly. However, if we define $\tau_{c}^{\text{usc}}: [0,\infty) \to [0,\infty)$ by
\begin{equation} \label{tausc} \gamma \mapsto \begin{cases} 1 \quad \text{ if } \imath'(\gamma) \leq c, \\ 0 \quad \text{ otherwise}, \end{cases} \end{equation}
then $\tau^{\text{usc}}_{a,c}$ does not have downwards jumps, and therefore it is u.s.c. For $\mathbf c \in [0,\tilde{\mathbf c}_+)$, we put $\boldsymbol{\tau}^{\text{usc}}_{\mathbf{c}}=(\tau^{\text{usc}}_{c_i})_{i \in \four}$. 

If $\mathbf b$ were positive, then $\lbrace G(\mathbf L^{\varrho'}_{\lambda},~\boldsymbol \tau_{\mathbf{c}}^{\text{usc}})(\mathbf W_\delta) \geq \mathbf b \rbrace$ would be a useful u.s.c. approximation of the considered event. However, we aim to deal with the general case where certain entries of $\mathbf b$ may be zero, 
and therefore it will be convenient to introduce quantities describing the worst QoS that is experienced by any user in the system. For this purpose, for $(\xi,u) \in \mathbf W_\delta$ and $\nu \in \mathcal{M}(\mathbf W_\delta)$ we put
\[ \mathbf \Phi((\xi,u),~\nu)=\left( R((\xi,u),~(o,F_o),~\nu),~D((\xi,u),~(o,F_o),~\nu),~R((o,F_o),~(\xi,u),~\nu),~D((o,F_o),~(\xi,u),~\nu) \right) \]
and note that for fixed $(\xi,u) \in \mathbf W_\delta$, ${\nu} \mapsto \mathbf \Phi ((\xi,u),~{\nu})$ is l.s.c., see Lemma \ref{3.6}. Here discontinuities may come from the effect that sites can become unavailable as relay locations if the number of users at certain sites tends to zero, as ${\nu}_n \to {\nu}$. 
Further, for $\mathbf c \in [0,\tilde{\mathbf c}_+),~(x,u) \in \mathbf W_\delta$ and ${\nu} \in \mathcal{M}(\mathbf W_\delta)$, we define
\[ \mathbf{\Phi}' (\mathbf c,~(\xi,u),~{\nu})=\left( \mathds 1\lbrace \pi_i (\Phi'((\xi,~u),~{\nu})  \leq c_i) \rbrace \right)_{i \in \four} , \]
as the indicator of the event that a user $(\xi,u)$ experiences QoS of at most $c_i$ for all $i \in \four$. Here $\pi_i$ denotes the projection $[0,\infty)^4 \to [0,\infty)$, $(a_j)_{j \in \four} \mapsto a_i$, $i \in \four$. 
Finally, we define
\[ \mathbf{\Phi}' (\mathbf c, ~{\nu})=\max\limits_{\begin{smallmatrix} (\xi,u) \in \mathbf W_{\delta} \\ {\nu}((\xi,u))>0 \end{smallmatrix}} \mathbf{\Phi}' (\mathbf c, ~(\xi,u),~{\nu}). \]
Similarly to the fading-free case in \cite[Section 6]{cikk}, $\mathbf \Phi'(\mathbf c,~(\xi,u),~{\nu})$ does not satisfy any semicontinuity properties. For example, lower semicontinuity can be obstructed by the effect that users along trajectories with bad QoS become irrelevant if the number of these users tends to zero. Upper semicontinuity can be obstructed e.g. by obstruction of the lower continuity of $\Phi ((\xi,u),~{\nu})$, cf. \cite[Lemma 3.7]{cikk} about compositions of monotonous semicontinuous functions and our Remark \ref{3.7,8}. Therefore we introduce the approximations
\[ \mathbf \Phi'_{\varepsilon} (\mathbf c, {\nu})=\max\limits_{(\xi,u) \in \mathbf W_\delta} \left\lbrace \mathbf \Phi'(\mathbf c,~(\xi,u),~{\nu}) \min \lbrace 1, \varepsilon^{-1}\nu((\xi,u)) \rbrace \right\rbrace. \]
In particular, $\mathbf \Phi' (\mathbf c, {\nu}) \geq \mathbf \Phi'_{\varepsilon} (\mathbf c, {\nu})$. In the following, for $\varepsilon>0$, $\mathbf b \geq 0,~\mathbf c \in [0, \tilde{\mathbf c}_+),~i \in \four$ we define
\[ C'_i(\mathbf{b},~\mathbf{c},~\varepsilon)= \begin{cases} \left\lbrace {\nu} \in \mathcal{M}(\mathbf W_\delta): \pi_i(\mathbf \Phi'_{\varepsilon} (\mathbf c, {\nu})) =1 \right\rbrace \quad \mathrm{ ~~if~} b_i=0, \\ \left\lbrace {\nu} \in \mathcal{M}(\mathbf W_\delta): \pi_i(G({\nu},~ \boldsymbol \tau^{\text{usc}}_{\mathbf c} ))(\mathbf W_\delta) \geq b_i \right\rbrace \text{~~if } b_i>0. \\   \end{cases} \]
Moreover, we put 
\[\mathbf{C}'(\mathbf b,~ \mathbf c,~\varepsilon) = \bigcap\limits_{i=1}^4 C'_i(\mathbf b,~ \mathbf c,~ \varepsilon). \]
Note that $\mathbf{C}'(\mathbf b,~ \mathbf c,~\varepsilon)$ is a closed set, since the maps ${\nu} \mapsto G(\nu,~ \boldsymbol \tau^{\text{usc}}_{\mathbf c} )$ and ${\nu} \mapsto \mathbf \Phi'_{\varepsilon} (\mathbf c, {\nu})$ are u.s.c. Note that by Lemma \ref{3.1} parts (ii) and (iv), for each $\varepsilon >0$ and $\alpha_+ > \alpha >0$ we have an inclusion
\begin{equation} \label{observationbefore6.2} \lbrace \alpha \mathbf L^{\varrho'}_\lambda \in \mathbf C'((1+\varepsilon) \mathbf b, ~ \mathbf c, ~ \mathbf \varepsilon) \rbrace \subset \lbrace G((\alpha_+ \mathbf L^{\varrho'}_\lambda ),~\boldsymbol \tau_{\mathbf c} \circ \imath') (\mathbf W_\delta) > \mathbf b \rbrace. \end{equation}
Now we show that under the event $\mathbf E_{\varepsilon_0}^\ast$ introduced in Section \ref{sprinkle}\index{sprinkling construction}, for $\alpha_+ > \alpha > 0$ the inclusion
\[ \left\lbrace \alpha \mathbf L^{\varrho'}_\lambda \in \mathbf C'(\sqrt{\alpha_+ \alpha^{-1}} \mathbf b, ~ \mathbf c, ~ \mathbf \alpha_+ \varepsilon_0) \right\rbrace \subset \left\lbrace G( \alpha_+ \mathbf L^{\varrho'}_\lambda , ~\boldsymbol \tau_{\mathbf c} \circ \imath') (\mathbf W_\delta) > \mathbf b \right\rbrace \] 
is close to being an equality. 
\begin{lem} \label{6.2}
Let $\alpha_+ > \alpha >0$, $\mathbf b \geq 0$ and $\mathbf c \in [0, \tilde{\mathbf c}_+)$ be arbitrary. Then, for every sufficiently small $\varepsilon_0 \in (0,1)$ there exists $\lambda_0 = \lambda_0(\varepsilon_0)$ such that for all $\lambda \geq \lambda_0$ we have
\begin{equation} \label{csirkefej} \mathbf E_{\varepsilon_0}^\ast \cap \left\lbrace G(\alpha \mathbf L^{\varrho'}_\lambda,~ \boldsymbol \tau_{\mathbf c} \circ \imath') (\mathbf W_\delta) > \mathbf b \right\rbrace \subset \left\lbrace \alpha_+ \mathbf L^{\varrho'}_{\lambda'} \in \mathbf C'(\sqrt{\alpha_+ \alpha^{-1}} \mathbf b, ~ \mathbf c, ~ \mathbf \alpha_+ \varepsilon_0) \right\rbrace, \end{equation}
with $\lambda'=(1+2\varepsilon_0 {\kappa}_\delta^{-1}) \lambda$. 
\end{lem}
\begin{proof} 
First, recall that under the event $\mathbf E^{\ast}_{\varepsilon_0}$, by passing from $\lambda$ to $\lambda'$ users can only be added along space trajectories which are occupied by at least one user. Hence, under the event $\mathbf E^{\ast}_{\varepsilon_0}$, parts (i) and (iii) of Lemma \ref{3.1} give that
\[ {D} \left( (x,u),~(y,v),~\alpha_+ \mathbf L^{\varrho'}_{\lambda'} \right) \leq {D}  \left( (x,u),~(y,v),~\alpha \mathbf L^{\varrho'}_{\lambda} \right),~~{R} \left( (x,u),~(y,v),~\alpha_+ \mathbf L^{\varrho'}_{\lambda'} \right)  \leq {R} \left( (x,u),~(y,v),~\alpha \mathbf L^{\varrho'}_{\lambda} \right) \]
for all $(x,u),~(y,v) \in \mathbf{W}_\delta$ if $\varepsilon_0$ is sufficiently small. 
Hence, under the event $\mathbf E_{\varepsilon_0}^\ast$, we conclude that $\pi_i\left(G (\alpha \mathbf L^{\varrho'}_\lambda,~\boldsymbol \tau_{\mathbf c} \circ \imath') (\mathbf W_\delta) \right) > b_i$ implies
\[ \pi_i\left( G(\alpha_+ \mathbf L^{\varrho'}_{\lambda'},~\boldsymbol \tau_{\mathbf c}^{\text{usc}}) (\mathbf W_\delta) \right) \geq \frac{\alpha_+ \lambda}{\alpha \lambda'} b_i \geq \sqrt{\alpha_+ \alpha^{-1}} b_i. \] 
In particular, $\alpha_+ \mathbf L^{\varrho'}_{\lambda'} \in C'_i(\sqrt{\alpha_+ \alpha^{-1}} \mathbf b, \mathbf c, \alpha_+ \varepsilon_0)$ if $b_i>0$.

For the case $b_i=0$, let $(x,u) \in \mathbf W_\delta$ with $\mathbf L^{\varrho'}_{\lambda}((x,u)) >0$ be arbitrary. Since the event $\mathbf E^\ast_{\varepsilon_0}$ occurs, we have $\mathbf L^{\varrho'}_{\lambda'}((x,u)) \geq \varepsilon_0$ and therefore
\[ \min \left\lbrace 1, (\alpha_+ \varepsilon_0)^{-1} (\alpha_+ \mathbf L^{\varrho'}_{\lambda'} ((x,u)))\right\rbrace =1. \]
Moreover, by parts (i) and (iii) of Lemma \ref{3.1}, we have that $\pi_i(\Phi'(\mathbf c,~ \alpha_+ \mathbf L^{\varrho'}_{\lambda'}))) \geq \pi_i(\Phi'(\mathbf c,~ \alpha \mathbf L^{\varrho'}_{\lambda}))$. Therefore,
\[ \pi_i(\Phi'_{\alpha_+ \varepsilon_0} (\mathbf c,~\alpha_+ \mathbf L^{\varrho'}_{\lambda'})) = \pi_i(\Phi' (\mathbf c,~\alpha_+ \mathbf L^{\varrho'}_{\lambda'})) \geq \pi_i (\Phi'(\mathbf c,~ \alpha \mathbf L^{\varrho'}_{\lambda})) =1, \]
which means that $\alpha_+ \mathbf L^{\varrho'}_{\lambda'} \in C'_i(\sqrt{\alpha_+ \alpha^{-1}} \mathbf b, \mathbf c, \alpha_+ \varepsilon_0)$, as wanted.
\end{proof}

\color{white} a \color{black} \\
Now, we can continue with the proof of Proposition \ref{6.1}. \\
\emph{Proof of Proposition \ref{6.1}.} First, by Lemmas \ref{3.10} and \ref{6.2}, we have
\begin{align*} \mathbb P(G(\alpha \mathbf L^{\varrho'}_{\lambda},~\boldsymbol \tau_{\mathbf c})(\mathbf W_\delta)>\mathbf b) &\leq \exp(\sqrt{\varepsilon_0} \lambda) \mathbb P_2 \left( \mathbf E^\ast_{\varepsilon_0} \cap \left\lbrace \alpha_+ \mathbf L^{\varrho'}_{\lambda'} \in \mathbf C'( \sqrt{\alpha_+ \alpha^{-1}} \mathbf b, ~\mathbf c, ~\alpha_+ \varepsilon_0) \right\rbrace \right) \\ &\leq \exp(\sqrt{\varepsilon_0} \lambda) \mathbb P_2 \left(\alpha_+ \mathbf L^{\varrho'}_{\lambda'} \in \mathbf C'( \sqrt{\alpha_+ \alpha^{-1}} \mathbf b, ~\mathbf c, ~\alpha_+ \varepsilon_0) \right). \end{align*}
In particular, the large deviation principle for $\mathbf L^{\varrho'}_\lambda$ and the observation \eqref{observationbefore6.2} together yield
\begin{align*} \limsup_{\lambda \to \infty} \frac{1}{\lambda} \log \mathbb P_2 \left( G(\alpha \mathbf L^{\varrho'}_{\lambda},~\boldsymbol \tau_{\mathbf c}) (\mathbf W_\delta)>\mathbf b \right) &\leq \sqrt{\varepsilon_0} - (1+2\varepsilon_0){\kappa}_\delta^{-1})  \inf\limits_{\tiny{\begin{smallmatrix} {\nu} \in \mathcal{M}(\mathbf W_\delta) \\ (\alpha_+ \nu) \in \mathbf C'( \sqrt{\alpha_+ \alpha^{-1}} \mathbf b, ~\mathbf c, ~\alpha_+ \varepsilon_0)  \end{smallmatrix}}} h({\nu} \vert \mu'^{\varrho'}) \\ &\leq \sqrt{\varepsilon_0}- (1+2\varepsilon_0 {\kappa}_\delta^{-1}) \inf\limits_{\tiny{\begin{smallmatrix} {\nu} \in \mathcal{M}(\mathbf W_\delta) \\ G(\alpha^2_+ \alpha^{-1} {\nu},~ \boldsymbol \tau_{\mathbf{c}} \circ \imath') (\mathbf W_\delta)>\mathbf b  \end{smallmatrix}}} h({\nu} \vert {\mu}'^{\varrho'}). \end{align*}
Hence, letting $\varepsilon_0$ to 0, we conclude
\[  \limsup_{\lambda \to \infty} \frac{1}{\lambda} \log \mathbb P_2 \left(G (\alpha \mathbf L^{\varrho'}_{\lambda},~\boldsymbol \tau_{\mathbf c}) (\mathbf W_\delta)>\mathbf b \right) \leq - \inf\limits_{\begin{smallmatrix} {\nu} \in \mathcal{M}(\mathbf W_\delta) \\ G(\alpha^2_+ \alpha^{-1} {\nu},~ \boldsymbol \tau_{\mathbf{c}} \circ \imath') (\mathbf W_\delta)>\mathbf b  \end{smallmatrix}} h({\nu} \vert {\mu}'^{\varrho'}). \]
Moreover, by Corollary \ref{3.12},
\[ - \inf\limits_{\tiny{\begin{smallmatrix} {\nu} \in \mathcal{M}(\mathbf W_\delta) \\ G(\alpha {\nu},~ \boldsymbol \tau_{\mathbf{c}} \circ \imath') (\mathbf W_\delta)>\mathbf b  \end{smallmatrix}}} h(\alpha^2 \alpha_+^{-2} {\nu} \vert {\mu}'^{\varrho'}) \leq -(1-3 \varepsilon') \inf\limits_{\begin{smallmatrix} {\nu} \in \mathcal{M}(\mathbf W_\delta) \\ G(\alpha {\nu},~\boldsymbol \tau_{\mathbf{c}} \circ \imath') (\mathbf W_\delta)>\mathbf b  \end{smallmatrix}} h({\nu} \vert \mu'^{\varrho'}) + 3\varepsilon'{\mu}'^{\varrho'} (\mathbf W_\delta), \]
where $\varepsilon'>0$ is such that $1-\varepsilon' = \alpha_+^{-2} \alpha^2$. Sending $\alpha_+$ to $\alpha$ completes the proof. \hspace{72pt} $\square$

\color{white} Hülye aki elolvassa \color{black} \\
Next, we can conclude the proof of Corollary \ref{1.2}. \\ 
\emph{Proof of Corollary \ref{1.2}.} Let $\varepsilon \in (0,1)$ be arbitrary. Then, using \eqref{másodikvégesvárhatóérték}, Proposition \ref{2.1} yields
\[ \mathbb P_2 (G(\mathbf L_\lambda,~\boldsymbol \tau_{\mathbf{c}})(\mathbf W) > \mathbf b) \leq \mathbb P_2 (G((1+\varepsilon)\mathbf L^{\varrho'}_\lambda,~\boldsymbol{\tau}_{\mathbf c} \circ \imath') (\mathbf W_\delta) > \mathbf b) \]
for sufficiently small $\delta \in \mathbb B$. Therefore, Proposition \ref{6.1} implies that
\[ \limsup_{\lambda \to \infty} \frac{1}{\lambda} \log \mathbb P(G(\mathbf L_{\lambda},~\boldsymbol \tau_{\mathbf c}) (\mathbf W)>\mathbf b) \leq - \inf\limits_{\begin{smallmatrix} {\nu} \in \mathcal{M}(\mathbf W_\delta) \\ G((1+\varepsilon) {\nu},~\boldsymbol \tau_{\mathbf{c}} \circ \imath') (\mathbf W_\delta)>\mathbf b  \end{smallmatrix}} h({\nu} \vert {\mu}'^{\varrho'}). \]
Using Lemma \ref{5.1}, we conclude that
\[ \lim_{\varepsilon \to 0} \liminf_{\delta \to 0}  \inf\limits_{\begin{smallmatrix} {\nu} \in \mathcal{M}(\mathbf W_\delta) \\ G((1+\varepsilon) {\nu},~\boldsymbol \tau_{\mathbf{c}} \circ \imath') (\mathbf W_\delta)>\mathbf b  \end{smallmatrix}} h({\nu} \vert {\mu}'^{\varrho'}) \geq \inf\limits_{\begin{smallmatrix} {\nu} \in \mathcal{M}(\mathbf W) \\ G({\nu},~\boldsymbol \tau_{\mathbf c}) (\mathbf W) >\mathbf b \end{smallmatrix}} h({\nu} \vert {\mu}'), \] 
as required. \hspace{370pt} $\square$

\color{white} Mindjárt kész! \color{black} \\
Finally, we prove Corollary \ref{1.3}.\\
\emph{Proof of Corollary \ref{1.3}.} 
Let $\varepsilon>0$ such that $G((1+\varepsilon)\mu',~\tau_{\mathbf c})(\mathbf W) \leq \mathbf b$. First, Lemma \ref{3.1} parts (ii) and (iv) together with Proposition \ref{2.1} imply that for sufficiently small $\delta \in \mathbb B$ we have
\[ \pi_i\left(G\left(\left(1+\frac{\varepsilon}{2}\right){\mu}'^{\varrho'},~\boldsymbol \tau^{\text{usc}}_{\mathbf c} \right) (\mathbf W_\delta)\right) \leq \pi_i \left(G\left(\left(1+\frac{3\varepsilon}{4}\right){\mu}'^{\varrho'},~\boldsymbol \tau_{\mathbf c} \circ \imath' \right) (\mathbf W_\delta)\right)\leq b_i, \] 
for all $i\in\four$. Moreover, another application of Proposition \ref{2.1} yields that it suffices to show that 
\begin{equation} \label{egyszerűsödnifog} \limsup_{\lambda \to \infty} \frac{1}{\lambda} \log \mathbb P_2 \left( G \left( \left( 1+\frac{\varepsilon}{3}\right) \mathbf L^{\varrho'}_\lambda, ~\boldsymbol \tau^{\text{usc}}_{\mathbf c} \right) (\mathbf W_\delta) > b \right) <0 \end{equation}
holds for sufficiently small $\delta \in \mathbb B$. Using Proposition \ref{6.1}, in order to verify \eqref{egyszerűsödnifog} we merely have to show
\[ \inf\limits_{\begin{smallmatrix} {\nu} \in \mathcal{M}(\mathbf W_\delta) \\ G((1+\varepsilon/3){\nu},~ \boldsymbol \tau_{\mathbf c}^{\text{usc}})(\mathbf W_\delta)>\mathbf b \end{smallmatrix}} h({\nu} \vert {\mu}'^{\varrho'})>0. \] 
We assume for contradiction that there exist ${\nu}_k \in \mathcal{M}(\mathbf W_\delta)$ such that $G((1+\varepsilon/3) {\nu}_k,~\boldsymbol \tau_{\mathbf c}^{\text{usc}})(\mathbf W_\delta) > \mathbf b$ and $\lim_{k \to \infty} h(\nu_k \vert {\mu}'^{\varrho'})=0$. In particular, the goodness of the rate function $h(\cdot \vert {\mu}'^{\varrho'})$ implies that the measures $ \lbrace {\nu}_k \rbrace$ have a subsequence $\lbrace {\nu}_{k_l} \rbrace$ that converges weakly\index{weak convergence of measures} to some ${\nu}^\ast$. The lower semicontinuity of the relative entropy w.r.t. $\mu'$ implies that $h({\nu}^\ast \vert {\mu}'^{\varrho'}) \leq \liminf_{l \to \infty} h({\nu}_{k_l} \vert {\mu}'^{\varrho'})=0$, thus we conclude that ${\mu}'^{\varrho'}={\nu}^\ast$. Hence, since the function ${\nu} \mapsto G(\nu,~ \boldsymbol \tau_{\mathbf c}^{\text{usc}})$ is also u.s.c., we obtain that $\pi_i(G((1+\frac{\varepsilon}{3}){\mu}'^{\varrho'},~ \boldsymbol \tau^{\text{usc}}_{\mathbf c} ) (\mathbf W_\delta)) \geq \mathbf b$. If $b_i>0$, then this together with parts (i) and (iii) of Lemma \ref{3.1} implies that $\pi_i(G((1+\frac{\varepsilon}{2}){\mu}'^{\varrho'},~\boldsymbol \tau^{\text{usc}}_{\mathbf c} )(\mathbf W_\delta)) \geq b_i$. Thus, we have obtained a contradiction to our assumption $\pi_i(G((1+\frac{\varepsilon}{2}){\mu}'^{\varrho'},~\boldsymbol \tau^{\text{usc}}_{\mathbf c} ) (\mathbf W_\delta)) \leq b_i$. On the other hand, if $b_i=0$, then we can apply the above argument with the u.s.c. function ${\nu} \mapsto \Phi'_{\kappa_\delta}(\mathbf c, ~(1+\frac{\varepsilon}{3}) {\nu})$ instead of ${\nu} \mapsto G(\nu,~\boldsymbol \tau_{\mathbf c}^{\text{usc}})$. More precisely, since ${\nu} \mapsto \Phi'_{\kappa_\delta}(\mathbf c, ~\nu)$ takes values in a discrete set, and since $\Phi'_{\kappa_\delta}(\mathbf c, ~(1+\frac{\varepsilon}{3}) {\nu}_k)=1$ holds for $k$ sufficiently large, we conclude that also ${\nu} \mapsto \Phi'_{\kappa_\delta}(\mathbf c,~ (1+\frac{\varepsilon}{3}) {\mu}'^{\varrho'})=1.$ Therefore,
\[ \pi_i(\Phi'(\mathbf c, ~(1+\frac{\varepsilon}{2}) {\mu}'^{\varrho'} )) \geq \pi_i(\Phi'(\mathbf c,~ (1+\frac{\varepsilon}{3}) {\mu}'^{\varrho'} )) = \pi_i (\Phi'_{\kappa_{\delta}} (\mathbf c,~ (1+\frac{\varepsilon}{3})  {\mu}'^{\varrho'} ))= 1, \] 
where we used parts (i) and (iii) of Lemma \ref{3.1} in the first inequality. Thus, we have obtained a contradiction to the assumption $\pi_i(\Phi'(c,~(1+\frac{\varepsilon}{2}) {\mu}'^{\varrho'})) = 0$. \hspace{200pt} $\square$
\section{Classification of the asymptotic behaviour of frustration probabilities} \label{kalsszikus}\index{frustration probabilities!classification}
In order to obtain a full understanding of the question when the frustration events $\mathbb P_2(G(\mathbf L_\lambda,~\boldsymbol\tau_{\mathbf c})(\mathbf W)>\mathbf b)$ for $\mathbf b \geq 0$, $\mathbf c \in (0, \tilde{\mathbf c}_+)$ decay exponentially, we give a classification of these events w.r.t. properties of the a priori measure $\mu'$, assuming that $\mu(W)>0$. (If $\mu(W)=0$, then $\mathbb P_2$-a.s. there are no users in the system, and all QoS quantities equal $\tilde{c}_+$.). We proceed using Corollaries \ref{1.2} and \ref{1.3}. (The case $c_i=0$ is pathological since all QoS quantities are nonnegative by our assumptions.) We summarize these results in Corollary \ref{új1.4} in the end of this section. 

Throughout the rest of the thesis, we say that for $A \subseteq \mathcal{M}(\mathbf W)$, $\inf \lbrace h(\nu \vert \mu') \vert~\nu \in A \rbrace$ is \emph{attained in} $\nu^\ast$ if $\nu^\ast \in A$ and $h(\nu^\ast \vert \mu')=\inf \lbrace h(\nu \vert \mu') \vert~\nu \in A \rbrace$. Using this definition, under Assumption \ref{szamár} with i.i.d. fadings, we have 
\begin{enumerate}
    \item If there exists $\varepsilon>0$ such that $G((1+\varepsilon)\mu',~\boldsymbol \tau_{\mathbf c})(\mathbf W) \leq \mathbf b$, then, according to Corollary \ref{1.3}, $\mathbb P_2 \left( G(\mathbf L_\lambda,~\boldsymbol \tau_{\mathbf c} )(\mathbf W)>\mathbf b \right)$ decays exponentially in $\lambda$.
    \item If $G(\mu',~\boldsymbol \tau_{\mathbf c})(\mathbf W)>\mathbf b$, then the infimum in \eqref{fontos} is attained by $\nu=\mu'$. Hence by Corollary \ref{1.2}, $\lim_{\lambda \to \infty} \mathbb P_2(G(\mathbf L_\lambda,~\boldsymbol\tau_{\mathbf c})(\mathbf W)>\mathbf b)=0$. I.e., probability of frustration events that are not unlikely w.r.t. the a priori product measure $\mu'$ does not decay exponentially.
    
\item If although $G(\mu',~\boldsymbol \tau_{\mathbf c})(\mathbf W)<\mathbf b$, we have $G((1+\varepsilon)\mu',~\boldsymbol \tau_{\mathbf c})(\mathbf W) > \mathbf b$ for all $\varepsilon>0$, then an easy computation shows that for all $\varepsilon>0$ \[ h((1+\varepsilon)\mu' \vert \mu')=((1+\varepsilon)\log(1+\varepsilon) -\varepsilon)\mu(W), \] and thus Corollary \ref{1.2} implies
\[ \lim_{\lambda \to \infty} \frac{1}{\lambda} \log \mathbb P_2(G(\mathbf L_\lambda,~\boldsymbol \tau_{\mathbf c})(\mathbf W)>\mathbf b) = - \inf\limits_{\begin{smallmatrix} \nu \in \mathcal{M}(\mathbf W) \\ G(\nu, ~\boldsymbol \tau_{\mathbf c}) (\mathbf W) > \mathbf b \end{smallmatrix}} h(\nu \vert \mu') \geq -((1+\varepsilon)\log(1+\varepsilon) -\varepsilon)\mu(W),  \] and the r.h.s. tends to 0 as $\varepsilon$ tends to 0. Hence, the frustration probabilities considered in this case do not decay exponentially. 
\item If $G(\mu',~\boldsymbol \tau_{\mathbf c})(\mathbf W)=\mathbf b$ and $\mathbf b >0$ (i.e. $b_i>0$ for some $i \in \four$), we make the following observations. For any $\nu \in \mathcal{M}( \mathbf W)$, $c >0$ and $\varepsilon>0$ we have
\[ \lbrace \SIR((x,u),~(y,v),~\nu)<(1+\varepsilon)c \rbrace = \lbrace \SIR((x,u),~(y,v),~(1+\varepsilon)\nu)<c \rbrace , \]
and this together with the monotonicity properties of $g$ implies for $\mathbf c \in (0,\tilde{\mathbf c_+})$
\[ \lbrace G(\nu,~\boldsymbol \tau_{(1+\varepsilon)\mathbf c})(\mathbf W) \leq \mathbf b \rbrace = \lbrace G((1+\varepsilon)\nu,~\boldsymbol \tau_{\mathbf c})(\mathbf W) \leq (1+\varepsilon)\mathbf b \rbrace. \] Thus, if $G(\mu',~ \boldsymbol \tau_{\mathbf c})(\mathbf W)=\mathbf b$, then for all $\varepsilon>0$ we have $G((1+\varepsilon)\mu',~ \boldsymbol \tau_{\mathbf c/(1+\varepsilon)})(\mathbf W)=(1+\varepsilon)\mathbf b$. In particular, since $b_i>0$ for some $i \in \four$, $G((1+\varepsilon)\mu', ~\boldsymbol \tau_{\mathbf c})(\mathbf W)>\mathbf b$ . Thus, analogously to the previous case, $\mathbb P_2(G(\mathbf L_\lambda,~\boldsymbol \tau_{\mathbf c})(\mathbf W)>\mathbf b)$ does not decay exponentially as $\lambda \to \infty$. 
\item If $G(\mu',~\boldsymbol \tau_{\mathbf c})(\mathbf W)=\mathbf b$ and $\mathbf b=0$, then $G(\mu',~\boldsymbol \tau_{\mathbf c})(\mathbf W)=\mathbf b$ implies that the set of locations where the a user observes QoS less than $\mathbf c$ is a nullset of the intensity measure $\mu'$. By Campbell's theorem (Theorem \ref{campbell})\index{Campbell's theorem} and the measurability properties observed according to \eqref{másodikvégesvárhatóérték}, we obtain 
\[ \lim_{\lambda \to \infty} \mathbb P_2(G(\mathbf L_\lambda,~\boldsymbol \tau_{\mathbf c})(\mathbf W)=0)=\int_{\mathbf W} \mathds 1 \lbrace G(\mu',~ \boldsymbol \tau_{\mathbf c})=0 \rbrace\d \mu =1. \]
We want to determine whether $\mathbb P_2(G(\mathbf L_\lambda,~\boldsymbol \tau_{\mathbf c})(\mathbf W)>0)$ decays exponentially or not. If $\mu(W)=0$, the answer is clearly yes, since then all QoS values in the system equal $\tilde{c}_+$. In the rest of this case, let $\mu(W)>0$. \\ We first note that \[ S_{\text{up-dir}}=\frac{\ellmin \Fmin}{\int_{\mathbf W} \ell(\vert x \vert) \mu'(\d x, \d u)}=\frac{\ellmin \Fmin}{\int_{\mathbf W} \ell(\vert x \vert) \mu(\d x) \mathbb E[F_0]}~\text{and} \] \[ S_{\text{do-dir}}=\mu\text{-}\essinf_{y \in W} \frac{\ellmin F_o}{\int_{\mathbf W} \ell(\vert x-y \vert) \mu(\d x) \mathbb E[F_0]}\] are the $\mu'$-essential infima of the SIR values in the system w.r.t. $\mu'$, coming from uplink communication and from downlink communication, respectively. This is true because $\Fmin$ is the $\mathbb P$-essential infimum (and $\Fmax$ is the $\mathbb P$-essential supremum) of $F_0$, moreover $\ell$ is Lipschitz continuous. We also define \[ K_{\text{up}}=\mu\text{-}\essinf_{(x,u) \in \mathbf W} \mu\text{-}\esssup_{(y,v) \in \mathbf W} \Gamma((x,u)~(y,v),~(o,F_o),~\mu'), \text{ and} \] \[  K_{\text{do}}=\mu\text{-}\essinf_{(x,u) \in \mathbf W} \mu\text{-}\esssup_{(y,v) \in \mathbf W} \Gamma((o,F_o)~(y,v),~(x,u),~\mu') \] which are the $\mu'$-essential infima of QoS levels for relayed uplink and relayed downlink communication, respectively. Thus, since $G(\mu',~\boldsymbol \tau_{\mathbf c})(\mathbf W)=0$ occurs, we conclude that for at least one $i \in \four$, $c_i \leq K_i$, where we define \begin{equation} \label{minsir} \mathbf K=(K_1,~K_2,~K_3,~K_4)=(K_{\text{up}},~K_{\text{up-dir}},~K_{\text{do}},~K_{\text{do-dir}})=(K_{\text{up}},~g(S_{\text{up-dir}}),~K_{\text{do}},~g(S_{\text{do-dir}})). \end{equation} We call $\mathbf K$ \emph{minimal $\SIR$ vector}\index{SIR!minimal SIR vector}.\footnote{ Formally, it would be more precise to call $\mathbf K$ minimal QoS vector, but since in all the concrete cases considered in the next chapter, $g$ will be given as the identity truncated at a large positive number, most often we will have that the $D$ type QoS quantities are indeed the same as the corresponding SIR quantities.}\\ First we consider the sub-case that this inequality is strict in at least one coordinate, say in the $i$th one, i.e. $c_i<K_i$. We show that then, $\mathbb P_2(G(\mathbf L_\lambda,~\boldsymbol \tau_{\mathbf c})(\mathbf W)> 0)$ decays exponentially. Indeed, for $\varepsilon>0$, let us define $\mathbf c_{\varepsilon} \in \mathbb R^4$ via $\mathbf c_{\varepsilon}(i)=(1+\varepsilon) c_i$ and $c_{\varepsilon}(j)=c_j$ for $j \neq i$. Let us write $m_1=\text{up}$, $m_2=\text{up-dir}$, $m_3=\text{do}$, $m_4=\text{do-dir}$. Then for sufficiently small $\varepsilon>0$ we have
\[ G((1+\varepsilon) \mu', ~\tau_{c_i},~m_i)(\mathbf W)=G(\mu', ~\boldsymbol \tau_{\mathbf c_\varepsilon (i)},~m_i)(\mathbf W)=0. \]
Hence, by Corollary \ref{1.3}, we conclude that 
\[ \limsup_{\lambda \to \infty} \frac{1}{\lambda} \log \mathbb P_2(G(\mathbf L_\lambda,~\tau_{c_i},~m_i)(\mathbf W)>0)<0. \]
Noting that for all $\lambda>0$ we have $\lbrace G(\mathbf L_\lambda,~\tau_{c_i},~m_i)(\mathbf W)>0 \rbrace \supseteq \lbrace G(\mathbf L_\lambda,~\boldsymbol \tau_{\mathbf c})(\mathbf W)>0 \rbrace$ implies the exponential decay of the probability of the latter event as $\lambda \to \infty$.

In the sub-case when there is no coordinate such that $c_j < K_j$, but $c_i=K_i$ for some $i \in \four$, we observe the following. Using the arguments of case 4, we conclude that $G((1+\varepsilon)\mu',~\boldsymbol \tau_{\mathbf c})(\mathbf W)>0$ for all $\varepsilon>0$. Hence, we can estimate
\[ \inf \lbrace h(\nu \vert \mu') |~G(\nu,~\boldsymbol \tau_{\mathbf c},~\mu')(\mathbf W)>0 \rbrace \leq \inf_{\varepsilon>0}  h((1+\varepsilon)\mu' \vert \mu')=\inf_{\varepsilon>0} ((1+\varepsilon)\log(1+\varepsilon) -\varepsilon)\mu(W)=0, \]
which, according to Corollary \ref{1.2}, implies that $\mathbb P_2(G(\mathbf L_\lambda,~\boldsymbol \tau_{\mathbf c})(\mathbf W)>0)$ does not decay exponentially.

\item The last remaining case is that $G(\mu',~\boldsymbol \tau_{\mathbf c})(\mathbf W)$ and $\mathbf b$ are incomparable w.r.t. our partial order \eqref{kisebbegyenlő} on $\mathbb R^4$. Then $\mathbf b \neq 0$, since $0$ is the minimal element in $[0,\infty)^4$ w.r.t. this partial order, and hence it is comparable with all non-negative vectors. Let us write $m_1=\text{up}$, $m_2=\text{up-dir}$, $m_3=\text{do}$, $m_4=\text{do-dir}$. This means that for some coordinate $i \in \four$, $G(\mu,~ \tau_{c_i}, m_i)(\mathbf W) \leq b_i$, and for another $j \in \four$, $G(\mu,~ \tau_{c_j}, m_i)(\mathbf W) > b_j$. Now, note that for all $j \in \four$ we have
\begin{small} \[ \mathbb P_2(G(\mathbf L_\lambda,~\boldsymbol \tau_{\mathbf c})(\mathbf W)>\mathbf b)=\mathbb P_2(G(\mathbf L_\lambda,~ \tau_{c_i},m_i)(\mathbf W)>b_i,~\forall i=1,\cdots,4) \leq \mathbb P_2(G(\mathbf L_\lambda,~ \tau_{c_j},m_j)(\mathbf W)>b_j).\] \end{small} Thus, if there exists at least one $i \in \four$ such that $\mathbb P_2(G(\mathbf L_\lambda,~ \tau_{c_i},m_i)(\mathbf W)>b_i)$ decays exponentially as $\lambda \to \infty$, then also $\mathbb P_2(G(\mathbf L_\lambda,~\boldsymbol \tau_{\mathbf c})(\mathbf W)>\mathbf b)$ decays exponentially. Using the previous cases, this happens if and only if there exists at least one $i \in \four$ and $\varepsilon>0$ such that $G((1+\varepsilon)\mu',~ \tau_{c_i})(\mathbf W) \leq b_i$. \\
Finally, since $\mathbf b \neq 0$, it remains to consider the sub-case when $G(\mu',~\boldsymbol \tau_{\mathbf c})$ and $\mathbf b$ are incomparable but $G((1+\varepsilon)\mu',~ \boldsymbol{\tau}_{\mathbf c})(\mathbf W) > \mathbf b$ for all $\varepsilon>0$. But then, again, we can use the argumentation of case 3 to conclude that $\mathbb P_2(G(\mathbf L_\lambda,~\boldsymbol \tau_{\mathbf c})(\mathbf W)>\mathbf b)$ does not decay exponentially. 
\end{enumerate}
Thus, we have proved the following corollary of Theorem \ref{1.1}, which determines whether the frustration probabilities decay exponentially, knowing the a priori measure $\mu'$. 
\begin{cor} \label{új1.4}\index{frustration probabilities!exponential decay}
Let $\mu(W)>0$. Under the assumptions of this chapter, for $\mathbf b \geq 0$ and $\mathbf c \in (0,\tilde{\mathbf c}_+)$, 
\[ \limsup_{\lambda \to \infty} \frac{1}{\lambda} \log \mathbb P_2(G(\mathbf L_\lambda,~\boldsymbol \tau_{\mathbf c})(\mathbf W)>\mathbf b)<0 \] 
holds if and only if one of the following is true:
\begin{enumerate}[(i)] \item there exists $i \in \four$ and $\varepsilon>0$ such that $G((1+\varepsilon)\mu',~\tau_{c_i},~m_i)(\mathbf W) \leq b_i$, \item $\mathbf b=0$, $G(\mu',~\boldsymbol \tau_{\mathbf c})(\mathbf W)=\mathbf b$, and there exists $i \in \four$ such that $c_i < K_i$, where $K_i$ is defined as the $i$th coordinate of the minimal $\SIR$ vector $\mathbf K$ in \eqref{minsir}. \end{enumerate}
\end{cor}
Such classification for the fading-free case does not appear in \cite{cikk}. We have nevertheless used Corollaries \ref{1.2} and \ref{1.3}, which are the analogues of \cite[Corollaries 1.2, 1.3]{cikk}, for deducing when frustration probabilities decay exponentially. We see from the argumentation of this section that this classification is formally analogous to the fading-free setting $F_0 \equiv 1$. However, the fading distribution contributes to $\mu'$, and therefore for fixed spatial intensity $\mu$ and for fixed frustration parameters $\mathbf b, \mathbf c$, it influences the asymptotic behaviour of the frustration probabilities with these parameters. A more delicate question is what the minimizers \eqref{minike} of relative entropy look like in specific cases, which we answer for a setting with two-dimensional communication area $W$ in Section \ref{ráadás}.
\chapter{Detecting the effects of random fadings} \label{effect}
We have developed an analogue of the results of \cite[Sections 1--6]{cikk} for the setting when users are static, fadings are random, i.i.d., bounded and bounded away from zero. In this chapter, we analyze the minimizers of relative entropy according to Corollary \ref{1.2}, and investigate how the randomness of the fadings influences the system. The chapter is organized as follows. In Section \ref{ráadás}, we obtain exact formulas for the minimizers of relative entropies w.r.t. direct uplink communication, in a special case of our setting with two-dimensional communication area, using rotational symmetry and variational calculus. In Section \ref{ocsú}, we describe an even more special setting where direct downlink communication can also be handled analytically. 
Where analytical approach is not applicable, we use numerical computations and simulations. This is done in Section \ref{I'm a simulant}, and it supports the main conjecture of this part of the chapter. In the high-density limit, in the most likely configurations that exhibit unexpectedly many frustrated users, the average loudness of the users is not unlikely large, but that the \emph{number} of users is unexpectedly high. Section \ref{bélus} investigates the existence of frustrated users in the system, using the results of Section \ref{kalsszikus}, in particular in more detail for the special setting of Section \ref{ráadás}. Finally, in Sections \ref{kernel} and \ref{Fernsehturm}, we relax the following assumptions of Chapter \ref{eleje}. In particular, Section \ref{kernel} describes space-dependent fading, and in Section \ref{Fernsehturm} we consider the case when also the fading value $F_o$ of the origin is random.
\section{Description of minimizers of the rate function in the direct uplink case}  \label{ráadás}
Following \cite[Section 7]{cikk}, in this section we work with a special case of the model of Section \ref{Anfang}.\footnote{ Note that the setting of \cite[Section 7]{cikk} is also mobility-free. It seems to be hard to describe the system in more detail than what Corollary \ref{1.2} gives if one considers the model with mobility, even if fadings are non-random.} Further, for geometrical reasons, we consider the disk $B_r(o) \subset \mathbb R^2$ centered at the origin with radius $r >0$. We observe that the approach of Chapter \ref{eleje} can be easily generalized to a setting where $r>0$ is not an integer, even the discretization in Section \ref{discretization} works the same way as before. We simplify the notation and write $W=B_r(o)$, $\mathbf W=W \times [\Fmin, \Fmax]$. We want to find the minimizers of the rate function presented in Corollary \ref{1.2} in the setting of Assumption \ref{szamár}, in particular to show how these minimizers depend on the distribution of the fading variable $F_0$. We restrict our observations to the direct uplink communication case. The reason for this is that in the direct uplink case minimizers exhibits radial symmetry, and this radial symmetry is transferred to our setting (cf. Proposition \ref{körbekörbekarikába} for a proof of this), while at the three other means of communication, even in the fading-free case, symmetry breaking is likely to occur at the minimizers, according to \cite[Sections 7.2, 7.3]{cikk}\index{radial symmetry!symmetry breaking} (see the exact definition of symmetry breaking also there). In those cases, there is no known explicit formula for describing the minimizers, and some interesting properties of the minimizers in have been conjectured by simulations \cite{cikk}.

We are interested in the specific frustration event
\[ \lbrace G(\mathbf L_\lambda,~\tau_{c},\text{up-dir})(\mathbf W)>b \rbrace, \]
also considered in the Corollaries \ref{1.2} and \ref{1.3}, which describes that w.r.t. direct uplink communication, at least $\lambda b$ users in $B_r(o)$ experience a QoS that is less than $c$.\index{radial symmetry}
We assume that the a priori measure $\mu$ for the locations of users is rotationally invariant on $B_r(o)$, with a continuos radial density. That is, \[ \mu(\d s)=q(s) \d s= q(\sqrt{x^2+y^2}) \d \sqrt{x^2+y^2}, \]
where $q$ is continuous. E.g. if $\mu$ is the restriction of the Lebesgue measure to $B_r(o)$, then the radial density is $q(s) = 2 \pi s$.
Under Assumption \ref{szamár}, if $g$ is given as $g(x)=\min \lbrace x, K \rbrace$ for $K$ sufficiently large to ensure that it is the identity on an interval that contains all possible direct uplink SIR values, i.e. $K \geq \frac{\ellmax \Fmax}{\int_W \ell(\vert x \vert) q(x) \d x  \mathbb E[F_0]}$, then the empirical measure of frustrated users w.r.t. direct uplink communication is given by
\[G( \mathbf L_\lambda,~\tau_c,~{\text{up-dir}}) = \frac{1}{\lambda} \sum_{X_j \in X^\lambda} \mathds 1 \lbrace \SIR (X_j,~o,~\mathbf L_{\lambda}) < c \rbrace \delta_{X_j}. \]
According to Corollary \ref{1.2}, for $b >0$, the probability for the event $\lbrace G(\mathbf L_\lambda,~\tau_{c},\text{up-dir})(\mathbf W) \geq b \rbrace$ decays exponentially at rate $\lambda \mathbf J^{\text{up-dir}}(c,b)$ where\footnote{ According to Section \ref{kalsszikus}, this rate may be equal to 0.}
\[ \mathbf J^{\text{up-dir}}(c,b)=\inf_{\nu \in \mathcal{M}(\mathbf W):~ G(\nu,~\tau_c,~{\text{up-dir}} )(B_r(o)) \geq b} h(\nu \vert \mu). \]
Here we used the definition \[ G(\nu,~\tau_c,~\text{up-dir}) (B_r(o) \times [\Fmin, \Fmax])=\int_{B_r(o)} \mathds 1 \lbrace \SIR ((x,u),~(o,F_o),~\nu) < c \rbrace \nu(\d x, \d u). \] Note that in these conditions we have written $\geq b$ instead of $>b$, in order to ensure that we always find a minimizer. Such a relaxation of the conditions is not possible if $b=0$, because then $G(\nu,~\tau_c,~{\text{up-dir}} )(B_r(o)) \geq b$ holds for all $\nu \in \mathcal{M}(\mathbf W)$.

Now we show that similarly to the fading-free case, also in our case all minimizers preserve the rotational symmetry in the direct uplink scenario.
\begin{prop} \label{körbekörbekarikába}\index{radial symmetry!of minimizers for the direct uplink}\index{means of communication!uplink!direct}
Suppose that Assumption \ref{szamár} holds, and $\mu$ is rotationally symmetric on $B_r(o)$. Assume further that $\mathbb P(F_0 \in \lbrace \Fmin, \Fmax \rbrace)=0$ and that $F_0$ has a strictly positive, continuous density $f$ with respect to the Lebesgue measure on $(\Fmin, \Fmax)$ (so that $\zeta(\d u)=\mathbb P(F_0 \in \d u)=f(u) \d u$ on $(\Fmin,\Fmax)$). Then the following holds. \\ In the direct uplink communication case, all minimizers in $\mathbf J^{\text{up-dir}}(c,b)$ are rotationally invariant (w.r.t. spatial configuration). I.e., they are of the form $\nu(\d x, \d y, \d u)=2 \pi s k(s,u) \d s \d u$ on $B_r(o) \times (\Fmin, \Fmax)$, where $s=\sqrt{x^2+y^2}$ denotes Euclidean norm in $\mathbb R^2$, $u$ denotes the fading coordinate, and $k: [0,r] \times (\Fmin, \Fmax) \to \mathbb R$ is a density function.
\end{prop}
\begin{proof}
If instead of our assumptions on $F_0$ we assume that $F_0 \equiv 1$, then \cite[Proposition 7.1]{cikk} ensures that the claim is true. 

Under the assumptions of our proposition, let $\mu$ be a rotation-invariant intensity measure on $B_r(o)$ that has a strictly positive density w.r.t. the Lebesgue measure, let $h_0$ denote its radial density. Now, assume that $\nu \in \mathcal{M}(B_r(o) \times [\Fmin, \Fmax])$ has a continuous density on $B_r(o) \times (\Fmin,\Fmax)$ and zero total mass on $B_r(o) \times \lbrace \Fmin, \Fmax \rbrace$. Furthermore, for $F \in (\Fmin, \Fmax)$, we have a projected measure $\nu_F$ on $B_r(o)$ given by
\[ \nu_F(\d x, \d y)=\lim_{\varepsilon \to 0} \frac{1}{2\varepsilon} \int_{F-\varepsilon}^{F+\varepsilon} \nu(\d x, \d y, \d u).\]
Then $\nu_F$ has a density $g_F$ w.r.t. $h_0$. By the continuity of $f$, we may assume that $F \mapsto g_F$ is continuous. We show that if $\nu$ is a minimizer, than $\nu$ is rotationally invariant.

Assume that $\nu$ is not rotationally invariant, then by \cite[Proposition 7.1]{cikk}, there exists a measure $\nu' \in \mathcal{M}(B_r(o))$ with density $g'_F$ w.r.t. $h_0$ such that 
\[ \int_{B_r(o)} g_F(x) h_0(\vert x \vert) \log g_F(x) dx > \int_{B_r(o)} g'_F(x) h_0(\vert x \vert) \log g'_F(x) \mathrm d x. \]
Now we define a new probability measure $\nu'$ on $B_r(o) \times [\Fmin, \Fmax]$ by defining its density $f'_F$ as follows
\[ f'_F(x,y,u)=\frac{f(x,y,u) \frac{g'_F(x,y)}{g_F(x,y)} \mathds 1 \lbrace g_F(x,y) >0 \rbrace}{\int_{B_r(o) \times (\Fmin,\Fmax)} f(x',y',u') \frac{g'_F(x',y')}{g_F(x',y')} \mathds 1 \lbrace g_F(x',y') >0 \rbrace \d x' \d y' \d u'} . \]
By the proof of \cite[Proposition 7.1]{cikk} $g'$ can be constructed in the following way
\[ g'_F(\vert x \vert)=\frac{1}{2 \pi \vert x \vert} \int_{\partial B_{\vert x \vert} (o)} g_F(y) \mathcal{H}_1(\d y), \]
where $\mathcal{H}_1(\d y)$ denotes one-dimensional Hausdorff measure (equivalently, curve length in $\mathbb R^2$), thus we see that $f'_F$ is a well-defined density function. By the continuity of $F \mapsto g_F$, there exists some $\varepsilon>0$ such that for all $F' \in (F-\varepsilon, F+\varepsilon) \subset (\Fmin, \Fmax)$ we have that 
\[ \int_{B_r(o)} g_F(x) h_0(\vert x \vert) \log g_F(x) dx > \int_{B_r(o)} g'_{F'}(x) h_0(\vert x \vert) \log g'_{F'}(x), \]
where $g'_{F'}=f'_F \circ \pi_{F'}^{-1}$. Let now $\nu'' \in \mathcal{M}(\mathbf W)$ be defined as $\nu''=\nu'$ on $W \times [F-\varepsilon, F+\varepsilon]$ and $\nu''=\nu$ on $\mathbf W \setminus W \times [F-\varepsilon, F+\varepsilon]$. Since by construction for all $u \in (\Fmin, \Fmax)$ we have that $\frac{f'_F(x,y,u)}{f(x,y,u)}=\alpha(x,y)$ (i.e., $\frac{f'_F}{f}$ does not depend on $u$), this implies that $h(\nu \vert \mu)>h(\nu'' \vert \mu)$, and hence $\nu$ is not a minimizer of the corresponding relative entropy.
\end{proof}
Now we give an approximate description of the minimizers in the direct uplink case, in the spirit of \cite[Section 7.1]{cikk}. We assume that the Lipschitz continuous path-loss function $\ell$ is monotone decreasing on $[0,r)$\index{path-loss!function!decreasing}, and that there exists $0 \leq r_0<r$ such that $\ell$ is strictly monotone decreasing on $[r_0,r)$. This condition ensures that $\ell_{\min} < \ell_{\max}$. We note that e.g. the path-loss function $\ell(\vert x \vert)=\min \lbrace R, \vert x \vert^{-\beta} \rbrace$, with constants $\beta>0$ and $R>r^{-\beta}$, satisfies this condition. This path-loss function it is widely used in the wireless communication literature, since it corresponds to the case of isotropic antennas with ideal Hertzian propagation\index{Hertzian propagation}\index{path-loss!function!ideal Hertzian propagation} (see e.g. \cite[Section II]{pathlossLDP}). Further, we assume again in addition to Assumption \ref{szamár} that $\zeta=\mathbb P \circ F_0^{-1}$ has zero mass on $\lbrace \Fmin, \Fmax \rbrace$ and a continuous density $f$ w.r.t. the Lebesgue measure on $(\Fmin, \Fmax)$. This implies $\Fmin < \Fmax$. The intensity measure $\mu'$ for the marked Poisson point process $\mathbf X^\lambda=\lbrace (X_i, F_{X_i}) \vert~ X_i \in X^\lambda \rbrace$ on $\mathbf W$ is given by $\mu'(\d s, \d u)=q(s) f(u) \d s \d u$.

First, we note that the decay of the path-loss function and the i.i.d. nature of the fading variables implies that it is entropically more efficient to increase the interference by placing more users close to the base station. As for the fadings of these users, in Section \ref{I'm a simulant} we investigate the question if the average fading in the minimizers is more than $\mathbb E[F_0]$. Second, if the interference at the origin is held fixed, then the SIR decays with the random path-loss of the user. Hence, users with bad QoS will be located at the boundary rather than at the centre of the cell $W$, and rather with low than with high fadings. Note that unlike in the fading-free case, the set where users with bad QoS are located may be a topologically complicated subset of $\mathbf W$. This is true because knowing the interference in the origin, the QoS of a user depends on the product of its path-loss value and fading value. Hence, in an arbitrarily small spatial neighbourhood of a connected user, there can be users who have lower fadings and hence they are disconnected.

The idea for the approximation is the following. Let $0<c<\tilde{c}_+$ and $b>0$ be fixed. Here $\tilde{c}_+=\frac{\ellmax \Fmax}{\int_{\Fmin}^{\Fmax} \int_{W} \ell(\vert x \vert) q(x) u f(u) \d x \d u}$ is the essential supremum of $\SIR(\cdot,((0,0),F_o),~\mu')$ w.r.t. $\mu'$. We fix a suitable $\alpha>0$, and compute the minimizer of the relative entropy under two simultaneous constraints. The first one is that a given SIR threshold $c$ is met precisely for $(x,u) \in \mathbf W$ such that $\ell(\vert x \vert)u=\alpha$. In particular, this implies that in the region \[ \mathbf D_\alpha=\lbrace (x,u) \in B_r(o) \vert~\ell(\vert x \vert)u <\alpha \rbrace \] all users are disconnected w.r.t. direct uplink communication, while in $\mathbf W \setminus \mathbf D_\alpha$ all users are connected. The second constraint is that the number of users on $\mathbf D_\alpha$ equals $b$. In order to achieve the desired quantity $b$ of disconnected users, in the outer region we add a positive or negative number of additional users (leaving the total number of users on $\mathbf D_\alpha$ positive). For these additional users, we use a profile that is flat in the spatial coordinate for each loudness value $u$, since this is entropically favourable. The optimization has to be performed now over the parameter $\alpha$ to balance the entropic costs of creating the desired amount of interference in the origin and changing the number of disconnected users in the outer regime. In the case of frustration events which are \emph{unlikely} w.r.t. $\mu'$, it is clear that both the interference of the origin and the number of users under SIR level $c$ will be larger w.r.t. the minimizer measure than w.r.t. $\mu'$. 

By the definition of $\mathbf D_\alpha$, $\alpha$ has to be an element of $(\Fmin \ell_{\min}, \Fmax \ell_{\max}]$ to ensure that $\mathbf D_\alpha$ is a set of positive Lebesgue measure in $\mathbb R^3$. This condition is necessary in order to be able to have a positive number of disconnected users on $\mathbf D_\alpha$ w.r.t. the absolutely continuous minimizing measure. If $\alpha=\Fmax \ell_{\max}$, then we have $\mathbf D_\alpha=\mathbf W$.

Using variational calculus, as presented in \cite{gelfand}, we derive expressions for minimizers of
\[ \inf_{\nu \in \mathcal{M}(\mathbf W):
~ G(\nu,~\tau_c,~\text{up-dir})(\mathbf W)=b} h(\nu \vert \mu'), \]
where the minimization is performed over $\alpha$. 
By continuity, we can obtain the minimizer of relative entropy among measures for which the number of $c$-frustrated users is greater than $b$ as the minimizer among the ones for which it is \emph{equal to} $b$, similarly to the fading-free setting. By Proposition \ref{körbekörbekarikába}, the minimizer $\nu$ has a density $h$ w.r.t. the Lebesgue measure on $\mathbb R^3$ such that $h$ is radial symmetric in the space coordinate. 
Let us use polar coordinates on $W=B_r(o)$, this way we can formulate our optimization problem for $\mathbf W'=[0,r] \times [\Fmin, \Fmax]$ and $\mathbf D'_\alpha =\lbrace (s,u) \in [0,r] \times [\Fmin, \Fmax] \vert~ \ell(s) u <\alpha \rbrace$ instead of $\mathbf W$ and $\mathbf D_\alpha$, respectively. Since $\nu$ is an extremal point of the relative entropy function
\[ h \mapsto \int_{\Fmin}^{\Fmax} \int_0^r  \frac{h(s,u)}{q(s) f(u)} \left[ \log \frac{h(s,u)}{ q(s) f(u)}-1 \right] q(s) f(u) \d s \d u = \iint\limits_{\mathbf W} h(s,u) \left[ \log \frac{h(s,u)}{ q(s) f(u)}-1 \right] \d s \d u \] under the two constraints
\begin{align} \label{egyikkutya} \iint\limits_{\mathbf W'} u~  \ell(s)~h(s,u) \d s \d u=\frac{\alpha}{c}, \\ \label{másikeb} \iint\limits_{\mathbf D'_\alpha} h(s,u) \d s \d u=b,  \end{align}
using \cite[Section 12, Theorem 1]{gelfand}\index{variational calculus!Lagrange multiplier} there exist constants $\beta_\alpha,~\delta_\alpha$ such that the minimizing density has the form
\[ h_{\alpha}(s,u)=q(s) f(u) \exp \left( \beta_{\alpha} u \ell(s) + \delta_\alpha \mathds 1 {\lbrace (s,u) \in \mathbf D'_\alpha \rbrace} \right). \]
Using this density, 
the entropic cost is given by
\begin{multline} \label{entrópikusköltség} \gamma_{\text{int}} (\alpha)= \iint_{\mathbf W'}\exp \left( \beta_{\alpha} u \ell(s) +\delta_{\alpha} \mathds 1 {\lbrace (s,u) \in \mathbf D'_\alpha \rbrace}  \right) \\ \times \left[ \beta_{\alpha} u \ell(s) +\delta_{\alpha} \mathds 1 {\lbrace (s,u) \in \mathbf D'_\alpha \rbrace} -1 \right]  q(s) f(u)\d s \d u + \mu(W).  
\end{multline}
For the minimizing value $\alpha_{\min}$, this leads to a minimizing density of the form\index{frustration probabilities!rate function!minimizers}
\begin{equation}  \label{ojlerlagrandzs} h(s,u)=f(u)~q(s)~\exp\left({\beta_{\alpha_{\min}} \ell(s) u  +  {\delta_{\alpha_{\min}}} \mathds 1 {\lbrace \ell(s)u <\alpha_{\min} \rbrace}}\right). \end{equation}\index{relative entropy!minimizers}
which indicates the dependence of the rate function on the distribution of $F_0$ in the direct uplink case in Corollary \ref{1.3}, in this simple setting with two-dimensional space, without mobility of the users. 

We note that by Sanov's theorem\index{Sanov's theorem} (Theorem \ref{nagysanov}), the rate function $h(\cdot \vert \mu')$ is convex, therefore the minimizer of relative entropy has to be unique under the constraint that the number of users under SIR level $c$ is \emph{equal to $b$}\index{frustration probabilities!rate function!convexity}. Note that also our optimization has a unique solution $\alpha_{\min}$ corresponding to \eqref{ojlerlagrandzs}.

After deducing this formula for the minimizer under the constraint that the number of frustrated users is equal to $b$, we provide a necessary condition about when the relative entropy of this minimizer equals the infimum of relative entropies of finite measures on $\mathbf W$ with strictly more than $b$ frustrated users. 
\begin{claim} \label{legyenmínusz}
Assume that the density $q$ is strictly positive on $W$. Assume that for $b>0$ and $c \in (0,\tilde{c}_+)$, the infimum \begin{equation} \label{hosszúinf} \inf \lbrace h(\nu \vert \mu') \vert~G(\nu,~\tau_c,~\mathrm{up-dir})(\mathbf W) \geq b \rbrace \end{equation} equals $h(\nu_{b,c} \vert \mu')$, where $\nu_{b,c}$ has radial density \eqref{ojlerlagrandzs}, then for the minimizing value $\alpha_{\min}$ for \eqref{entrópikusköltség}. Then at least one of the constants $\beta_{\alpha_{\min}}$ and $\delta_{\alpha_{\min}}$ corresponding to \eqref{ojlerlagrandzs} is non-negative. 
\end{claim}
\begin{proof}
Let us assume that \eqref{hosszúinf} is attained by $\nu_{b,c}$. First, let us assume for contradiction that both $\beta_{\alpha_{\min}}$ and $\delta_{\alpha_{\min}}$ are negative. Then the interference of the origin for the minimizing measure $\nu_{b,c}$ with density $h(s,u)$ is smaller than with respect to $\mu'$, and also $\mu'(\mathbf D_\alpha)>\nu_{b,c}(\mathbf D_\alpha)=b$. Hence, there are more users in the frustration area $\mathbf D_\alpha$ w.r.t. $\mu'$ than w.r.t. $\nu_{b,c}$, and they are all under the direct uplink SIR level $c$ w.r.t. $\mu'$. This indicates $G(\mu',~\tau_c,~\text{up-dir})(\mathbf W)>b$, and also $0=h(\mu' \vert \mu') < h(\nu_{b,c} \vert \mu')$, which contradicts the fact that \eqref{hosszúinf} is attained by $\nu_{b,c}$.
\end{proof}
Note that if $c \in (0,\tilde{c}_+)$ and $b>0$ are such that $G(\mu',~\tau_c,~\text{up-dir})(\mathbf W)<b$, then $h(\nu_{b,c} \vert \mu')$ equals \eqref{hosszúinf}. Now, by the definition of SIR for the direct uplink, if $c \in (0,\tilde{c}_+)$ and $b>0$ are such that $G(\mu',~\tau_c,~\text{up-dir})(\mathbf W)=b$, then there exists $\alpha_{\min} \in [\ell_{\min} \Fmin, \ell_{\max} \Fmax]$ such that $\SIR((s,u),~(o,F_o),~\mu')$ $=c$ holds for $(s,u)$ such that $\ell(\vert s \vert) u=\alpha_{\min}$. Therefore in this special case, it follows that $\nu_{b,c}=\mu'$, which implies that $\beta_{\alpha_{\min}}=\delta_{\alpha_{\min}}=0$. An easy computation shows that we have $\alpha_{\min}=c \int_{\mathbf W'} \ell(s)u ~q(s)~ f(u)  \d s\d u$. Thus, the only case when the condition of Claim \ref{legyenmínusz} is not satisfied is when $c \in (0,\tilde{c}_+)$ and $b>0$ are such that $G(\mu',~\tau_c,~\text{up-dir})(\mathbf W)>b$. Then our optimization yields $G(\nu_{b,c},~\tau_c,\text{up-dir})(\mathbf W)=b$, hence \eqref{hosszúinf} is not equal to $h(\nu_{b,c} \vert \mu')$ but to $h(\mu' \vert \mu')=0$. In this case, one has to decrease the number of users in the a priori distribution in order to obtain a configuration which exhibits exactly $b$ users under SIR level $c$. Hence it is necessary that at least one of the constants $\beta_{\alpha_{\min}}$ and $\delta_{\alpha_{\min}}$ be negative.

We will discuss the cases $b=0$ and $b \downarrow 0$ in Section \ref{bélus}. There we first set up some general results about the infima of relative entropies as $b \downarrow 0$, and using these we conclude more concretely for the direct uplink in this special rotationally invariant setting.
\section{The path-loss-free setting for direct downlink communication} \label{ocsú}
As we noted in Section \ref{ráadás}, even in the fading-free case, the minimizers of relative entropy for direct downlink communication may break the radial symmetry, see \cite[Section 7.2]{cikk}. Nevertheless, there is a special case of our setting from Section \ref{Anfang} where we can pove that symmetry breaking does not occur, handling large deviation properties of direct downlink frustration probabilities analytically is particularly simple.
In this case, there are no spatial effects that influence the number of frustrated users, only the random number of users and the randomness of their fadings plays a r\^{o}le. However, relative entropies still depend on spatial configurations of the corresponding measures.

\begin{exam}[The path-loss-free setting]\index{path-loss!-free setting}\index{means of communication!downlink!direct}
Let $r>0$ be arbitrary, and let us use the observation from Section \ref{ráadás} that the approach of Chapter \ref{eleje} also works if $r$ is not an integer. Let us consider our marked Poisson point process $\mathbf X^\lambda$ on $W \times [\Fmin, \Fmax]$, under Assumption \ref{szamár}. Then, we note that e.g. if the path-loss function is given by $\ell(s)=\min \lbrace K, s^{-\beta} \rbrace$ for some $\beta>0$, then for sufficiently small $r>0$ we have that $(x,y) \mapsto \ell(\vert x-y \vert)$ is constant on $[-r,r]^d \times [-r,r]^d$. Hence, it is not pathological to consider the case when the path-loss function on is constant equal to $\ellmax=\ellmin=K$ on $W$. Then the SIR w.r.t. direct downlink communication is given by
\[ \SIR((o,F_o),X_i,~\mathbf L_\lambda)=\frac{F_{o}}{\frac{1}{\lambda}\sum_{X_j \in X^\lambda} F_{X_j}}=\frac{F_o}{\frac{1}{\lambda}\sum_{j=1}^{N(\lambda)} F_{X_j}}. \] 
Here $N(\lambda)=\lambda L_\lambda(W)$ is Poisson distributed with parameter $\lambda \mu(W)$ and independent of the fadings of the users. Now, let $c \in (0,\tilde{c}_+)$. Then the numerator and also the denominator of this SIR quantity is independent of the receiver $X_i$, and for all $X_i \in X^\lambda$ we have \[ \lbrace \SIR((o,F_o),X_i,~\mathbf L_\lambda) < c \rbrace = \lbrace \SIR((o,F_o),X_k,~\mathbf L_\lambda) < c, \forall X_k \in X^\lambda \rbrace=\left\lbrace \frac{1}{\lambda} \sum\limits_{j=1}^{N(\lambda)} F_{X_j} > F_o/c \right\rbrace. \] In particular, either all users or neither of them are strictly under SIR level $c$. In the latter case, we have $G(\mathbf L_\lambda,~\tau_c,~\text{do-dir})=\frac{N(\lambda)}{\lambda}$. 

Let us consider the fading-free direct downlink case with all the special assumptions of Section \ref{ráadás}, i.e. that $W=B_r(o) \subset \mathbb R^2$, the intensity measure $\mu$ is rotationally invariant with radial density $q$, and $\zeta=\mathbb P \circ F_0^{-1}$ is absolutely continuous with density $f$. We show that in this case, all minimizers of the rate function are rotationally invariant. Indeed, let $\nu \in \mathcal{M}(\mathbf W)$ with density $g$. For $c \in (0,\tilde{c}_+)$, we observe that $G(\nu,~\tau_c,~\text{do-dir})$ does not depend on the spatial configuration of the users in $\nu$. Hence, we conclude that the measure $\nu' \in \mathcal{M}(\mathbf W)$ defined by its density $g'$ as
\[ g'(x,u)=\frac{1}{2 \pi \vert x \vert} \int_{\partial B_{\vert x \vert} (o)} g(y,u) \mathcal{H}_1(\d y), \]
where $\mathcal H_1$ denotes one-dimensional Hausdorff measure (curve length in $\mathbb R^2$), exhibits the property $G(\nu,~\tau_c,~\text{do-dir})=G(\nu',~\tau_c,~\text{do-dir})$. Moreover, $\nu'$ is rotationally symmetric. The arguments of the proof of \cite[Proposition 7.1]{cikk} imply that $h(\nu' \vert \mu') \geq h(\nu \vert \mu')$. Here equality holds if and only if $\lbrace g' \neq g \rbrace$ has zero Lebesgue measure, i.e. if $g$ is rotationally invariant.

Therefore, similarly to Section \ref{ráadás}, the form of the minimizers in $\lbrace h(\nu \vert \mu') \vert~G(\nu,~\tau_c,~\text{do-dir})(\mathbf W) \geq b \rbrace$ can be found via solving the corresponding Euler--Lagrange equations. For $c \in (0,\tilde{c}_+)$ and $b>0$, we look for the measure $\nu_c$ with density $h_c$ which minimizes the relative entropy function
\begin{equation} \label{kislagrandzs} (s,u) \mapsto \int_{\mathbf W'} h_c(s,u) \left[ \log \frac{h(s,u)}{q(s)f(u)} -1 \right] \d s \d u \end{equation}
under the two constraints 

\begin{equation} \label{ebura} \int_{\mathbf W'} h_{b,c} (s,u) \d s \d u \geq b, \end{equation} 
\begin{equation} \label{fakó} \int_{\mathbf W'} h_{b,c} (s,u) u \d s \d u \geq \frac{F_o}{c K}. \end{equation}
I.e., the interference has to be large enough to push all users under the direct downlink SIR level $c$, but also the total number of users has to be at least $b$. Our observation that the number of frustrated users w.r.t. $\nu \in \mathcal{M}(\mathbf W)$ depends only on the total number of users $\nu(\mathbf W)$, and it is equal to $\nu(\mathbf W)$ if $\nu(\mathbf W)$ is sufficiently large, implies the following. \emph{Any} measure $\nu \in \mathcal{M}(\mathbf W)$ that exhibits at least $b$ users in total, also has the property that all these users are on SIR level less than or equal to $c$. Thus, we can find the minimizer $h_{b,c}$ via minimizing the relative entropy \eqref{kislagrandzs} under the constraints \eqref{ebura} and \eqref{fakó} but with equality instead of inequality in both ones. This is true for $b$ large enough to ensure that these two equations have a common solution. For small $b$, the equation coming from \eqref{ebura} will disappear and the optimization only under the constraint \eqref{fakó} with equality instead of inequality will yield the minimizer; this is certainly true for the limiting case $b=0$. Using \cite[Section 12, Theorem 1]{gelfand}\index{variational calculus!Lagrange multiplier} leads to a minimizing density of the form
\begin{equation} \label{cöcöcö} h_{b,c}(s,u)=q(s) f(u) \exp(\gamma(b,c) u+\delta(b,c)), \end{equation}
where the second term in the exponential tilting is zero for small non-negative $b$.

Note also that in this path-loss-free case, the behaviour of direct uplink frustration probabilities may be more complex than the behaviour described for the direct downlink. This is true because in the direct uplink case, the numerator of SIR also depends on the fading of the transmitter. In particular, for $c \in (0,\tilde{c}_+)$ it may happen that some but not all users have QoS less than $c$. For the same reason, the relayed downlink cannot be handled as easily as the direct downlink, even in this special, fading-free case. 
\end{exam}
The conclusion of this example is the following. On the one hand, similarly to the direct uplink case in Section \ref{ráadás}, we see that also in the direct downlink case, fading effects are in general not negligible reasons for bad connection. This is shown by the fading-dependence of the exponential tilting in \eqref{cöcöcö}. On the other hand, if there is no path-loss in the system, then effects coming from the number of users are stronger than fading effects. Indeed, let us note that $N(\lambda)=X^\lambda_\delta(W)$ holds for all $\delta \in \mathbb B$, and thus we see that $\lbrace N(\lambda) \rbrace_{\lambda>0}$ is a homogeneous Poisson process on $(0,\infty)$ with intensity $\mu(W)$. Therefore by the Poisson law of large numbers (Theorem \ref{poissonlln}\index{Poisson law of large numbers}) we have that for any $\varepsilon>0$, $\mathbb P_2( N(\lambda) \in ((1-\varepsilon) \lambda\mu(W), (1+\varepsilon) \lambda\mu(W)))$ converges to 1 as $\lambda \to \infty$. Under the event $\lbrace N(\lambda) \in ((1-\varepsilon) \lambda\mu(W), (1+\varepsilon) \lambda\mu(W)) \rbrace$, the fact that $\Finite$ implies that the SIR quantities are also bounded from above and bounded away from zero. In other words, arbitrarily low positive SIR values can only occur in the system if the \emph{number of users is unusually large}. 

This is also true if the path-loss is non-constant, since by assumption $0<\ell_{\min} < \ell_{\max} < \infty$. If there is also path-loss involved, the path-loss properties and the number of the users together correspond to one source of randomness, while the other source of randomness is the realization of the fadings of the users. Then, also for the direct downlink, the interference is different at each receiver in the system, and therefore the large deviation behaviour of frustration probabilities is more complex. The next section shows the importance of having unexpectedly many users in total in order to obtain a rare frustration event, in a special case of the setting of Section \ref{ráadás}, i.e. for direct uplink communication with non-constant path-loss.
\section{Loudness of users in the entropy minimizing settings (simulation results)}\label{I'm a simulant}
As we sketched in the Introduction, we are particularly interested in the following question. In the minimizer density of relative entropy \eqref{minike} for \emph{unlikely} frustration events, are users in average louder than the expectation of the fading variable $F_0$? If the answer were yes, then this would show that unexpectedly large average loudness is one of the reasons for bad connection in the system. In particular, not only the number and the spatial configuration of users determines these most likely rare events, but also untypical behaviour of the fadings of the users an important r\^{o}le.

However, the simulations in this section show that such effect most likely does not occur. The reason for this is that as long as the number of users is held fixed, the SIR is constant under multiplying the path-losses and fadings of \emph{all} users by the same constant. Instead, it turns out that the most likely configurations corresponding to the rare frustration events exhibit an \emph{unlikely large number of users}, which was also seen for path-loss-free direct downlink communication in Section \ref{ocsú}.

Let us consider the special case of Section \ref{ráadás} for direct uplink communication. Let us normalize the minimizer $\nu_{b,c}$, which has density \eqref{ojlerlagrandzs}, into a probability measure with density $h_0$, via rescaling the factors $\beta_{\alpha_{\min}}$ and $\delta_{\alpha_{\min}}$, let $\beta^0_{\alpha_{\min}}$ and $\delta^0_{\alpha_{\min}}$ denote the rescaled factors. I.e., we write
\begin{equation}  \label{ojlerlagrandzs0} h_0(s,u)=f(u)~q(s)~\exp\left({\beta^0_{\alpha_{\min}} \ell(s) u  +  {\delta^0_{\alpha_{\min}}} \mathds 1 {\lbrace \ell(s)u <\alpha_{\min} \rbrace}}\right). \end{equation}\index{relative entropy!minimizers} Then (conditional on the 1-set that $\lim_{\lambda \to \infty} \lambda L_\lambda(W)>0$) the average loudness of users in $\nu_{b,c}$ can be expressed as follows
\[ \int_{\mathbf W} h_0(s,u) u \d u = \mathbb E_2 \left[ F_{X_1} \exp(\beta^0_{\alpha_{\min}} F_{X_1}+\delta^0_{\alpha_{\min}} \mathds 1 \lbrace {\ell(\vert X_1 \vert) F_{X_1}}<\alpha_{\min} \rbrace \right)].\]
Now, e.g. if $\beta^0_{\alpha_{\min}}$ is positive, then by convexity we have \begin{equation} \label{kacsamajom} \mathbb E_2 \left[ F_{X_1} \exp(\beta^0_{\alpha_{\min}} F_{X_1}) \right] \geq \mathbb E_2 \left[ F_{0}] \exp(\beta^0_{\alpha_{\min}} \mathbb E[F_{0}]) \right] > \mathbb E[F_{0}], \end{equation}
i.e., the first term accounts for an increase in the average loudness. On the other hand, if we know that $\alpha_{\min}$ is strictly less than $\Fmax \ell_{\max}$ and $\delta^0_{\alpha_{\min}}$ is not too large, then the average loudness of the \emph{additional users} corresponding to the second term in the exponential of \eqref{kacsamajom} is less than $\mathbb E[F_0]$. This is true because they are situated on $\mathbf D_\alpha$, where the average random path-loss and thus the average fading is lower than in average. By Claim \ref{legyenmínusz}, in the case of our interest when $b>0$ and $c \in (0,\tilde{c}_+)$ are such that $G(\mu',~\tau_c,~\text{up-dir})(\mathbf W)<b$, we have that at least one of the constants $\beta_{\alpha_{\min}}$ and $\delta_{\alpha_{\min}}$ is nonnegative. Now, the question whether the average loudness in $\nu_{b,c}$ is larger than $\mathbb E[F_0]$ reduces to the question whether $\beta_{\alpha_{\min}}$ is large compared to $\delta_{\alpha_{\min}}$.

One cannot expect to be able to answer this question unless via numerical computations and simulations. The difficulty is that although we have the explicit formula \eqref{ojlerlagrandzs} for the density of the minimizer of relative entropy, we have to determine $\alpha_{\min}$ numerically. In order to perform this, for all possible $\alpha$ one has to compute the the parameters $\beta_\alpha$ and $\delta_\alpha$ numerically from the implicitly given integral equations \eqref{egyikkutya} and \eqref{másikeb}. This way, results about the average loudness of users in the minimizer density for fixed levels of number of frustrated users $b$ and QoS threshold $c$ does not give information about different values $b$ and $c$. Therefore instead of considering this minimal density and computing the average loudness of it numerically for certain $b$, $c$, we simulate our marked Poisson point process $\mathbf X^\lambda$ and consider the average fading in configurations corresponding to rare frustration events w.r.t. direct uplink communication. We note that these latter methods can also be extended to the other forms of communication, but the running time of the algorithms becomes particularly longer since the interference has to be computed at each user.

We used the restricted model of Section \ref{ráadás}, with the spatial intensity measure $\mu'$ being the restriction of the Lebesgue measure to $W=B_1(o) \subset \mathbb R^2$. Further we assumed that the path-loss function corresponds to ideal Hertzian propagation\index{Hertzian propagation}\index{path-loss!function!ideal Hertzian propagation} with $\beta=4$ and $K=5$, i.e. $\ell(s)=\min \lbrace s^{-4}, 5 \rbrace$, and the fadings of the users are uniformly distributed on the interval $[1,2]$. Then the first picture of Figure \ref{ábrázolás} shows that the a priori quantity $G(\mu',~\tau_{1.1},~\text{up-dir})(\mathbf W)$ of users under SIR level $c=1.1$ is not more than 0.6, in particular it follows from the continuity of $c \mapsto G(\mu',~\tau_c,~\text{up-dir})(\mathbf W)$ that there exists $\varepsilon>0$ such that $G((1+\varepsilon)\mu',~\tau_c,~\text{up-dir})(\mathbf W)$ is strictly less than $b=0.9875$. Thus, by Corollary \ref{1.3}, for $\lambda$ large, $\lbrace G(\mathbf L_\lambda,~\tau_{1.1},~\text{up-dir})(\mathbf W)>0.9875 \rbrace$ is a rare event. We generated $10^6$ samples of $X^\lambda$ with overall intensity $\lambda \mu(W)=50$, we considered those realizations which satisfied $G(\mathbf L_\lambda,~\tau_{1.1},~\text{up-dir})(\mathbf W)>0.9875$ and compute the average loudness of users in these configurations. In other words, we proceeded similarly to the simulations described in \cite[Section 7.1]{cikk}, but now we were interested rather in the average loudness of the users than their spatial positions. According to the numerically computed shape of $c \mapsto G(\mu',~\tau_{1.1},~\text{up-dir})(\mathbf W)$, Campbell's theorem (Theorem \ref{campbell}) implies that there exists users in the Poisson point process with SIR level significantly less than 1.1 with asymptotically positive probability as $\lambda \to \infty$. In particular, writing $N(\lambda)=\lambda L_\lambda (W)$, for $\lambda>0$ and $\alpha>0$ we have that the probability
\[ \mathbb P_2(G(\mathbf L_\lambda,~\tau_{1.1},~\text{up-dir})(\mathbf W)>0.9875,~N(\lambda) \leq (1+\alpha)\lambda \mu(W)) \]
is positive. In other words, the parameters $b=1.1$, $c=0.9875$ are chosen in such a way that it can happen with positive probability that the number of the users is not unexpectedly high, but their fadings and path-loss values still imply the frustration event $\lbrace G(\mathbf L_\lambda,~\tau_{1.1},~\text{up-dir})(\mathbf W)>0.9875 \rbrace$. However, we saw the following.

The empirical probability of the event $\lbrace G(\mathbf L_\lambda,~\tau_{1.1},~\text{up-dir})(\mathbf W)>0.9875 \rbrace$ after $10^6$ simulations was $1.1 \times 10^{-5}$, i.e. we found 11 configurations that are contained in these event. We computed the average fading of the users in these 11 configurations. The mean of these 11 averages was 1.47019, which is even lower than the expectation $\mathbb E[F_0]=1.5$. Out of the 11 configurations, the maximal average was 1.5378. There were larger deviations towards the other direction: the lowest average was 1.4221. This way, we do not see in the simulations that unexpectedly large average loudness would be the reason of the considered frustration event.

On the other hand, we see a interesting effect if we consider the \emph{number of users} in these rare configurations. In all the 11 ones, this number turned out to be at least 81 and at most 86. These are all very unlikely quantities. The number of users is Poisson distributed with parameter 50, which takes values greater than 80 with probability $3.436 \times 10^{-5}$, which is cca. 3 times as much as the empirical probability of $\lbrace G(\mathbf L_\lambda,~\tau_{1.1},~\text{up-dir})(\mathbf W)>0.9875 \rbrace$. Therefore it was interesting to run the simulation once again and check what proportion of the configurations with more than 80 users belongs to the ones where more than 98.75\% of the users is under SIR level 1.1. After the next round of simulations, we found 8 configurations belonging to $\lbrace G(\mathbf L_\lambda,~\tau_{1.1},~\text{up-dir})(\mathbf W)>0.9875 \rbrace$, each of which exhibiting more than 80 users. Now the mean of the average fadings in these 8 configurations was 1.49373, i.e. slightly below $\mathbb E[F_0]$. The minimal average loudness under this event was 1.46921, the maximal one was 1.51669. In total, there were 40 configurations with more than 80 users, which is even 5 times the number of the configurations with extremely many numbers of frustrated users under SIR level 1.1. Thus, empirically, the rare event $\lbrace G(\mathbf L_\lambda,~\tau_{1.1},~\text{up-dir})(\mathbf W)>0.9875 \rbrace$ is a proper subset of $\lbrace N(\lambda)>80 \rbrace$. The mean of the average loudnesses in all these 40 configurations was 1.47632, which is further away from $\mathbb E[F_0]$ than the mean for only the 8 configurations with extremely many frustrated users. 

We conclude the simulations as follows. For intensity $\lambda \mu=50$ and for $\alpha >1$, no positive empirical correlation can be seen between the occurrance of the events \[ \lbrace G(\mathbf L_\lambda,~\tau_{1.1},~\text{up-dir})(\mathbf W)>0.9875 \rbrace \quad \text{ and } \quad \lbrace \frac{1}{N(\lambda)} \sum_{X_i \in X^\lambda} F_{X_i} > \alpha \mathbb E[F_0] \rbrace. \] However, in our samples, each configuration belonging to $\lbrace G(\mathbf L_\lambda,~\tau_{1.1},~\text{up-dir})(\mathbf W)>0.9875 \rbrace$ exhibits an unlikely large number of users in the system, and this causes a large interference that makes unexpectedly many users frustrated. We note that this seems to be the same effect as the one we described in Section \ref{ocsú} for the fading-free direct downlink case. I.e., under Assumption \ref{szamár}, both the path-losses and the fadings in the system are bounded from above and bounded away from 0, and the only unbounded random quantity is the number of the users, which is Poisson distributed. We see that even under frustration events that do not have zero probability conditional on the constraint that the number of users is not unexpectedly large, i.e. where it would be possible to have a usual number of users and still too many frustrated users, we experience that these events occur only when there are unlikely many users. 

However, not all configurations with unexpectedly many users correspond to unlikely many frustrated users. Instead, it seems to be the case that compared to the unusually large number of users, fadings and path-losses have to be in some sense close the average behaviour, since this is entropically favourable. This is shown in the simulation by the fact that if the number of frustrated users is large, then the average fading of users is closer to $\mathbb E[F_0]$ than generally in the case when $N(\lambda)$ is too large. Also in the explicit minimizer \eqref{ojlerlagrandzs} from Section \ref{ráadás}, we see that the higher the fading density at a given fading value is, the more users have this loudness value in the minimizer. 

We conjecture that for any relevant frustration parameters $b \geq 0$ and $c \in (0,\tilde{c}_+)$, the average fading, and similarly the average path-loss, will not be much larger than the expected one. The reason for this is, as we have already mentioned in the beginning of this section, that given the number of users, the SIR is constant under multiplying all path-losses and fadings by the same constant. Hence, if the number of users is not unexpectedly large, then the increase of the interference caused by large path-loss and fading values is cancelled by the increase of the random path-loss\index{path-loss!expected}\index{path-loss!random} (numerator of SIR) of the users. The same cancellation does not occur if we increase the interference by increasing the number of users, since the denominator of SIR is $1/\lambda$ times and not $1/N(\lambda)$ times the sum of random path-losses. Once having sufficiently many users to cause interference, the average loudnesses and path-losses have to be as close to the expected behaviour as possible, in order not to cause additional costs on the exponential scale.

\begin{figure}
\includegraphics[scale=0.8]{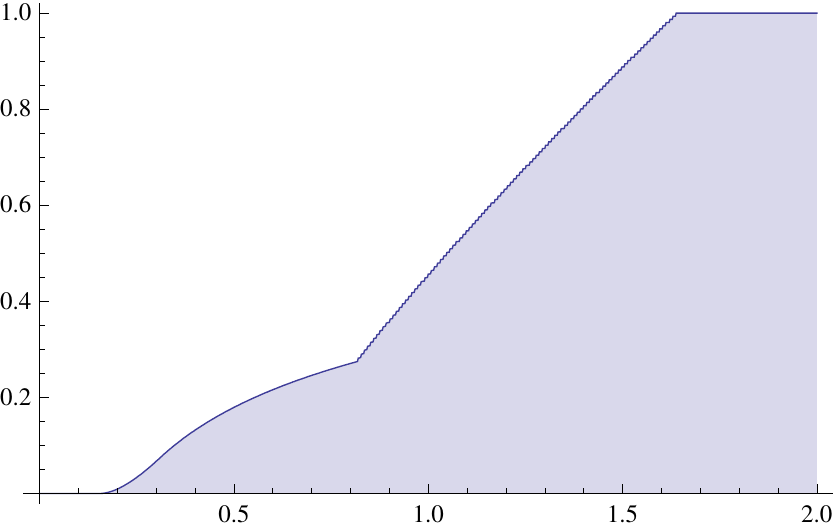}
\includegraphics[scale=0.8]{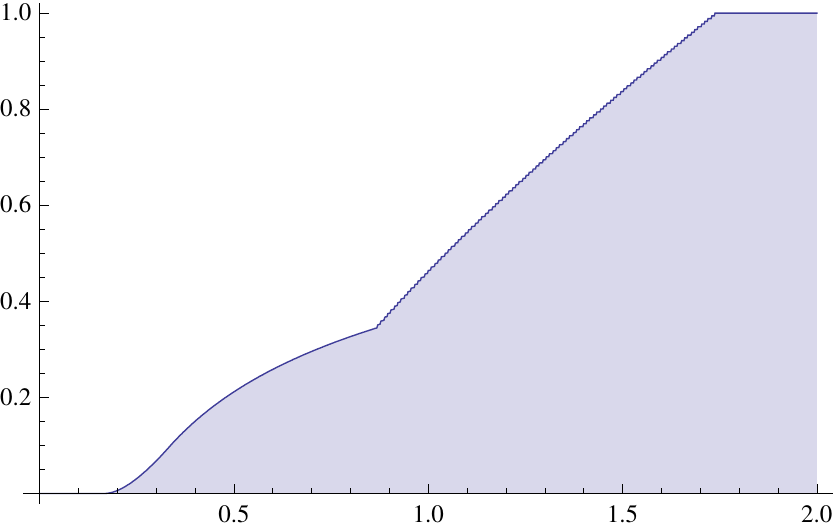} \\
\includegraphics[scale=0.8]{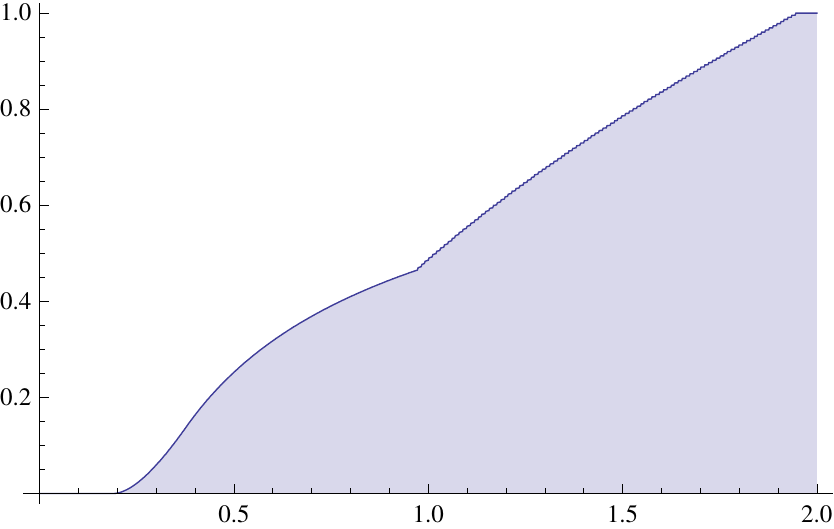}
\includegraphics[scale=0.8]{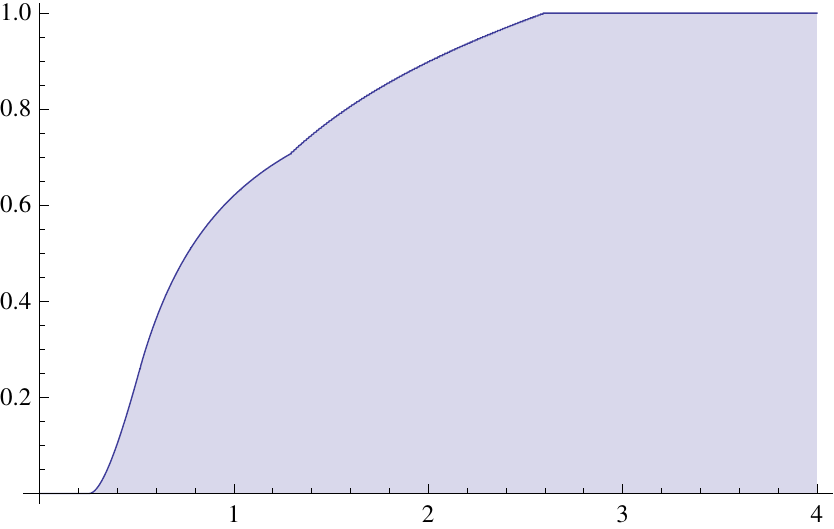} \\
\includegraphics[scale=0.8]{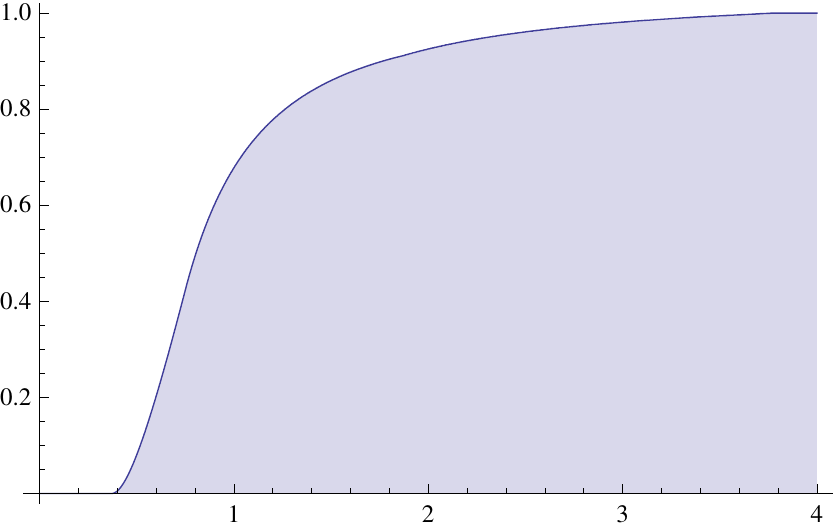}
\includegraphics[scale=0.8]{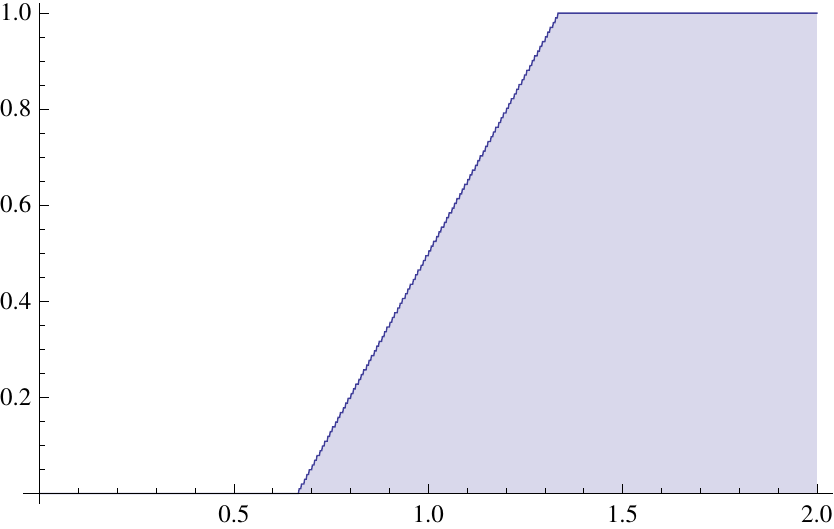}
\caption{The function $p:~c \mapsto G(\mu',~\tau_c,~\text{up-dir})(\mathbf W)/\mu'(\mathbf W)$ in the case of $\ell(| x |)=\min \lbrace 5, |x|^{-\beta} \rbrace$ with $\beta=4,~3,~2,~1,~1/2,~0$ (from the left to the right, row by row). In both axes, 1 unit equals 1, i.e. in the typical setting, there exist users with direct uplink SIR value larger than 1 in the system. \\ The upper almost linear segments of the functions appear because the highest path-loss value $\ellmax=5$ has a positive measure w.r.t. the spatial intensity $\mu$. Thus, close to the maximal SIR value, almost all users have maximal path-loss values, and since the cumulative distribution function of the fading variable is linear on $[1,2]$, this yields an almost linear increase of the function $p$. By the linearity of the fading distribution function, as the SIR level tends to the minimal level from above, the derivative of $p$ tends to a linear multiple of the slope $\ell$ at $\ellmin$, which is positive. Thus, $p$ is nondifferentiable at both the minimal and the maximal SIR level. For $\beta>0$, $p$ is also nondifferentiable and even its concavity breaks at the half of the maximal SIR value, the point where the first users with maximal path-loss value appear. However, for small positive $\beta$ the difference between the one-sided slopes is very little, see the case of $\beta=1/2$. 
Apart from these three points, the smoothness of the path-loss function $\ell$ and the cumulative distribution function of $F_0$ together with the fact that for the direct uplink, the denominator of SIR is the same for all users implies that the function $p$ is also smooth. \\ Thus, for $n \in \mathbb N$, choosing the fading distribution function to be $n$ times continuously differentiable also in $\Fmin$ respectively $\Fmax$, then also $p$ will be $n$ times continuosly differentiable and the minimal respectively maximal SIR level. The breaking point in the middle can be smoothened by choosing $\ell$ to be strictly monotone decreasing (also if the fading variable is uniform). } \label{ábrázolás}
\end{figure}
\section{Existence of frustrated users} \label{bélus}\index{frustration events!existence of frustrated users}
According to Corollaries \ref{1.2} and \ref{1.3}, we obtained large deviation limits for the frustration probabilities $\mathbb P_2(G(\mathbf L_\lambda,~\boldsymbol \tau_{\mathbf c})(\mathbf W)>\mathbf b)$ also in the case when $b_i=0$ for some $i \in \four$. As we have seen, in this case the frustration probability $\mathbb P_2(G(\mathbf L_\lambda,~\boldsymbol \tau_{c_i},m_i)(\mathbf W)>b_i)$ means that there exists at least a single user who is under QoS $b_i$. Here $m_i$ is the corresponding means of communication, using the notation of Section \ref{kalsszikus}. We want to investigate under which conditions the equality
\begin{equation} \label{béla} \lim_{\mathbf b \downarrow 0} \lim_{\lambda \to \infty} \frac{1}{\lambda} \log \mathbb P_2(G(\mathbf L_\lambda,~\boldsymbol \tau_{\mathbf c})(\mathbf W)>\mathbf b)=\lim_{\lambda \to \infty} \frac{1}{\lambda} \log \mathbb P_2(G(\mathbf L_\lambda,~\boldsymbol \tau_{\mathbf c})(\mathbf W)> 0) \end{equation} 
holds, and similarly for cases when some coordinates of $\mathbf b$ converge to 0. Moreover, in specific cases for specific means of communication, such as the one described in Section \ref{ráadás}, this will imply by continuity for all $i \in \four$ that
\begin{equation} \label{bélaegyenlőis} \lim_{b_i \downarrow 0} \lim_{\lambda \to \infty} \frac{1}{\lambda} \log \mathbb P_2(G(\mathbf L_\lambda,~\boldsymbol \tau_{c_i},~m_i)(\mathbf W) \geq b_i)=\lim_{\lambda \to \infty} \frac{1}{\lambda} \log \mathbb P_2(G(\mathbf L_\lambda,~\boldsymbol \tau_{c_i},~m_i)(\mathbf W)> 0). \end{equation} We also want to determine whether the corresponding infimum of relative entropy in \eqref{bélaegyenlőis} according to \eqref{minike} is attained, and if in some sense the minimizing configuration for $b_i=0$ is the limit of the minimizing configurations corresponding to $b_i$ as $b_i \downarrow 0$. 

First, we prove that \eqref{béla} holds in full generality, without having any information about the forms of the corresponding minimizers of relative entropy. 
\begin{prop} \label{Nagy Béla}
Under Assumption \ref{szamár}, let $\mathbf c \in (0,\tilde{\mathbf c}_+)$. Then \eqref{béla} holds. In addition, for all $i \in \four$ we have \begin{equation} \label{bélai} \lim_{b_i \downarrow 0} \lim_{\lambda \to \infty} \frac{1}{\lambda} \log \mathbb P_2(G(\mathbf L_\lambda,~\tau_{c_i},~m_i)(\mathbf W)>b_i)=\lim_{\lambda \to \infty} \frac{1}{\lambda} \log \mathbb P_2(G(\mathbf L_\lambda,~ \tau_{c_i},~m_i)(\mathbf W)>0). \end{equation} 
\end{prop}
\begin{proof}
It suffices to prove the second, coordinate-wise statement, which implies the first one.\footnote{ Recall from Section \ref{egykettőegyhárom} that for all $\lambda >0$, the event $\lbrace G(\mathbf L_\lambda,~\tau_{c_i},~m_i)(\mathbf W)>b_i \rbrace$ equals $\lbrace G(\mathbf L_\lambda,~\boldsymbol \tau_{\mathbf c})(\mathbf W)>\mathbf b \rbrace$, where $\mathbf b \in \mathbb R^4$ is such that its $i$th coordinate equals $b_i$, and its remaining coordinates are negative. This ensures that the large deviation limits of frustration probabilities in \emph{single coordinates} are also given by infima of relative entropies w.r.t. $\mu'$.} Let $c \in (0,\tilde{c}_+)$. Using Corollary \ref{1.2}, we merely have to show that for all $i \in \four$ we have
\begin{equation} \label{szám} \lim_{b_i \downarrow 0} \inf_{\nu \in \mathcal{M}(\mathbf W):~ G(\nu,~\tau_{c_i},~m_i)(\mathbf W)>b_i} h(\nu \vert \mu')= \inf_{\nu \in \mathcal{M}(\mathbf W): ~ G(\nu,~\tau_{c_i},~m_i)(\mathbf W)>0} h(\nu \vert \mu'). \end{equation} It is clear that the l.h.s. of \eqref{szám} is greater than or equal to the r.h.s. In order to prove the opposite inequality, one can use a standard argument for limits of infima, which does not involve any specific property of the relative entropies, and therefore we leave it for the reader.
\end{proof}
The proof indicates that this result is also formally analogous to the fading-free setting $F_0 \equiv 1$. Now we show an interesting effect coming from the case distinction of Section \ref{kalsszikus}, where we also used fading-related quantities to define the minimal SIR vector $\mathbf K$.
\begin{claim} \label{nincsenpathloss}
Under Assumption \ref{szamár}, let $\mathbf c \in (0,\tilde{\mathbf c}_+)$ and let us assume that $G(\mu',~\boldsymbol{\tau}_{\mathbf c})(\mathbf W)=0$. Then,
\[ \inf \lbrace h(\nu | \mu') | ~ \nu \in\mathcal{M}(\mathbf W), ~G(\nu,~\boldsymbol \tau_{\mathbf c})(\mathbf W)>0 \rbrace=0 \] if and only if there exists $i \in \four$ such that $c_i=K_i$, where $\mathbf K$ is the corresponding minimal $\SIR$ vector defined in \eqref{minsir}, but there exists no $j \in \four$ such that $c_j$ is strictly less than $K_j$.\index{SIR!minimal SIR vector}
\end{claim}
\begin{proof}
First, let us assume that $c_i=K_i$ for some $i$ but $c_j \geq K_j$ for all $j \in \four$. Then, by Corollary \ref{új1.4} we have $\inf \lbrace h(\nu | \mu') | ~ \nu \in\mathcal{M}(\mathbf W), ~G(\nu,~\boldsymbol \tau_{\mathbf c})(\mathbf W)>0 \rbrace>0$. Second, if $c_i < K_i$ for some $i \in \four$, then, again by Corollary \ref{új1.4}, $\inf \lbrace h(\nu | \mu') | ~ \nu \in\mathcal{M}(\mathbf W), ~G(\nu,~\boldsymbol \tau_{\mathbf c})(\mathbf W)>0 \rbrace=0$.
\end{proof}
The claim means that even if the a priori measure $\mu'$ has no users with QoS vector strictly less than $\mathbf c$, the number of users in the marked Poisson point process under QoS level $\mathbf c$ will not decay exponentially if $\mu'$ has a positive quantity of users at QoS level arbitrarily close to $c_i$ for at least one $i \in \four$.

Knowing these, let us let us reconsider the specific case of Section \ref{ráadás}, where for $b>0$, $c \in (0,\tilde{c}_+)$, the density of the minimizer $\nu_{b,c}$ of $\lbrace h(\nu \vert \mu') \vert~ G(\nu, ~\tau_c,~\text{up-dir})(\mathbf W) \geq b \rbrace$ is given by \eqref{ojlerlagrandzs}. Let us fix this $c$ and denote the density \eqref{ojlerlagrandzs} by $h_{b}(s,u)$ instead of $h(s,u)$, and also let us emphasize the dependence of the Lagrange multipliers on $b$, i.e.
\begin{equation} \label{ojlerkelagrandzska} h_{b}(s,u)= f(u)~q(s)~\exp\left({\beta_{\alpha_{\min}}(b) \ell(s) u  +  {\delta_{\alpha_{\min}}}(b) \mathds 1 {\lbrace \ell(s)u <\alpha_{\min} \rbrace}(s,u)}\right). \end{equation}
Now we compare this to the case $b=0$ and then conclude for the limit $b \downarrow 0$. If we want to find the minimizer $\nu_{0,c}$ of $\lbrace h(\nu \vert \mu') \vert~ G(\nu, ~\tau_c,~\text{up-dir})(\mathbf W)>0 \rbrace$, i.e. the minimal rate for finite measures that exhibit at least a single frustrated user, then the SIR level has to decrease to $c$ at least at the outer boundary $\partial B_r(o) \times \lbrace \Fmin \rbrace$ of $\mathbf W$. This way, the minimizer $\nu_{0,c}$ is the extremal point of the relative entropy function
\[ \alpha \mapsto \iint_{\mathbf W} h^\alpha(s,u) \log \left[ \frac{h^\alpha(s,u)}{q(s) f(u)}-1 \right] \d s \d u\]
under the constraint
\[ \iint_{\mathbf W} \ell(s) u h^\alpha(s,u) \d s \d u = \frac{\alpha}{c}.~ \footnote{ Here $\alpha$ is not a power, it just indicates that the dependence on $\alpha$ here is not the same as the dependence on $b$ in $h_b$ in \eqref{ojlerkelagrandzska} and also that $h_{\alpha}$ is not the same as $h_\alpha$ in Section \ref{ráadás}.}\]\index{relative entropy!minimizers}
Hence, again by \cite[Section 12, Theorem 1]{gelfand}\index{variational calculus!Lagrange multiplier}, we conclude that there exists a constant $\gamma_0$ such that the minimizer density is given by
\begin{equation} \label{ojler0lagrandzs0} h_0 (s,u)=q(s) f(u) \exp(\gamma_0 \ell(s) u), \end{equation}
i.e. the form of this minimizer density is analogous to the first term of the minimizer density for $b>0$, but the second term is missing.

Now, the argumentation of Claim \ref{nincsenpathloss} applied for only the direct uplink coordinate implies that if $c=K_2=\frac{\ell_{\min} \Fmin}{\int_{\mathbf W} q(s) \d s \mathbb E[F_0]}$ is the minimal direct uplink SIR level, then $\inf \lbrace h(\nu \vert \mu') \vert G(\nu, ~\tau_c,~\text{up-dir})(\mathbf W)>0 \rbrace$ is zero, in particular it equals $h(\mu' \vert \mu')$. This implies that in this case, $h_0$ has to equal the density of $\mu'$, i.e. it follows that $\gamma_0=0$. If $c<K_2$, then the infimum is greater than 0, which together with the fact that there exist no users $(x,u) \in \mathbf W$ with $\SIR((x,u),~(o,F_o),~\mu') \leq c$ implies that $\gamma_0>0$. We also note that if $c>K_2$, then $G(\mu', ~\tau_c,~\text{up-dir})(\mathbf W)>0$.

Hence, using Proposition \ref{Nagy Béla}, we conclude that if $c \in (0,\tilde{c}_+)$ is such that $G(\mu', ~\tau_c,~\text{up-dir})(\mathbf W)=0$, then, using that for all $b \geq 0$, $\inf \lbrace h(\nu \vert \mu') \vert~ G(\nu, ~\tau_c,~\text{up-dir})(\mathbf W) \geq b \rbrace$ is attained by $\nu_{b,c}$, we have
\begin{multline} \label{taláneltünökhirtelen} \lim_{b \downarrow 0} \iint_{\mathbf W} h_b(s,u) \log \left[ \frac{h_b(s,u)}{q(s) f(u)}-1 \right] \d s \d u=\lim_{b \downarrow 0} \inf_{\nu \in \mathcal{M}(\mathbf W):~ G(\nu,~\tau_{c},~\text{up-dir})(\mathbf W) \geq b} h(\nu \vert \mu') \\ = \inf_{\nu \in \mathcal{M}(\mathbf W): ~ G(\nu,~\tau_{c},~\text{up-dir})(\mathbf W)>0} h(\nu \vert \mu')=\iint_{\mathbf W} h_0(s,u) \log \left[ \frac{h_0(s,u)}{q(s) f(u)}-1 \right] \d s \d u. \end{multline}
In particular, this implies the following for the Lagrange multipliers
\[ \lim_{b \downarrow 0} \beta_{\alpha_{\min}}(b)=\gamma_0 \geq 0 \quad \text{and} \quad \lim_{b \downarrow 0} \delta_{\alpha_{\min}}(b) =0, \]
because \eqref{taláneltünökhirtelen} implies that in the limit $\beta \downarrow 0$ the second term in the exponential of \eqref{ojlerkelagrandzska} has to vanish and the whole exponential tilting has to converge to the one in \eqref{ojler0lagrandzs0}.

Finally, if $c \in (0,\tilde{c}_+)$ is such that $G(\mu', ~\tau_c,~\text{up-dir})(\mathbf W)=b_0>0$, then, analogously the proof of Proposition \ref{Nagy Béla}, we have that for $b \in [0,b_0]$, $\nu_{b,c}$ equals $\mu'$. In particular, already for some positive values $b$, the corresponding constants $\beta_{\alpha_{\min}}$ and $\delta_{\alpha_{\min}}$ equal 0.

The overall conclusion for the setting of \eqref{ráadás} is that for fixed $c$, not only the infima of the rates $\inf \lbrace h(\nu \vert \mu') \vert~ G(\mu', ~\tau_c,~\text{up-dir})(\mathbf W)>b \rbrace$ converge to $\inf \lbrace h(\nu \vert \mu') \vert~ G(\mu', ~\tau_c,~\text{up-dir})(\mathbf W)>0 \rbrace$, but also the minimizer densities $h_b$ converge to $h_0$ on $\mathbf W$. By the boundedness of $\ell$ and $[\Fmin, \Fmax]$, this convergence must not be only pointwise but also uniform on $\mathbf W$.

\section{Loudness depending on space} \label{kernel}\index{fading!depending on space}
In this section, we extend the model defined in Section \ref{Anfang} to the case when the distribution of the fading of a user also depends on the spatial position of the user. After describing the new model in general, we provide two examples how to construct it in such a way that the marked Poisson point process of user-fading pairs takes values in a separable space. Using a special case of the second example, we conclude that in the case of space-dependent fading, spatial effects and fading effects in the system cannot be separated any more.

In Chapter \ref{eleje}, under Assumption \ref{szamár}, our marked Poisson point process $\mathbf X^\lambda=(X_i, F_{X_i})_{X_i \in X^\lambda}$ had identically distributed marks. The conditional distribution of the fading $F_{X_i}$ knowing the position of the user $X_i=x \in W$ was given by
\[ p(x,\d u)=\mathbb P(F_0 \in \d u)=\zeta(\d u),\quad x \in W,~ u \in [\Fmin, \Fmax]. \]
Using the Marking Theorem (Theorem \ref{marking})\index{Marking Theorem}, one can also consider different probability kernels $p(x,\d u)$ which also depend on the spatial position $x$ of the users, and define the conditional distribution of the fading of the user at spatial position $x$ as $p(x,\cdot)$. This dependence can have the interpretation that e.g. in some areas the users have less modern devices, which have lower fading values than the ones of users in other regions. Also it can happen for some geographical or meteorological reasons that the signal coming from users of certain areas is partially reflected or absorbed. This will result an additional decay in the signal strength, which is not explained by the path-loss function $\ell$. In these latter cases, we actually face spatial effects which in principle do not depend on the fading properties of the users. But since the path-loss function $\ell$ is assumed to depend only on distance, we cannot explain these effects in the fading-free model, where fadings of the users are constant equal to 1.

By the Marking Theorem\index{Poisson point process!marked!separable construction}, for any probability kernel $p(\cdot,\centerdot)$ describing the distribution of the fading $F_{X_i}$ of the user $X_i$ at spatial position $x$, $\mathbf X^\lambda=\lbrace (X_i, F_{X_i})\rbrace_{X_i \in X^\lambda}$ will be a (marked) Poisson point process in the product space $\mathbf W$, and the fadings of two different users $X_i$ and $X_j$ will remain independent. However, in order to be able to do large deviation analysis, it is convenient to ensure that this marked Poisson point process takes values in a Polish space. We have already seen in Section \ref{kingmanke} that this holds in the case of i.i.d. fadings $p(\cdot, \centerdot)=\zeta(\centerdot)$, and that by the Ionescu-Tulcea construction it is sufficient to ensure that the product probability space is separable. Then, as before, the marked Poisson process takes values in the Polish product space $ \mathbf W^{\mathbb N}$. 

In the following, we present two general ways to construct models with space-dependent fadings. One can easily show that in these cases the process $\mathbf X^\lambda$ also takes values in a separable space.
\vspace{3pt} 

\emph{1. Countably many areas.}\index{fading!depending on space!countably many areas} Let us consider a countable partition \[ \lbrace W^1, W^2, \ldots \rbrace \] of $W$, such that $W^i$ is a measurable subset of $W$ for all $i \in \mathbb N$, $W^i$ is disjoint from $W^j$ for all $i \neq j$, and $\cup_{i \in \mathbb N} W^i=W$. Let $F_0^1, F_0^2, \ldots$ be fading distributions such that there exists $\Finite$ such that $\inf_{n \in \mathbb N} \essinf F_0^i > 0$ and $\sup_{n \in \mathbb N} \esssup F_0^i < \infty$. Writing \[ \zeta^i(\d u)=\mathbb P(F_0^i \in \d u), \quad u \in [\Fmin, \Fmax], \]
we define the probability kernel
\[ p(x,\d u)=\sum_{i=1}^{\infty} \zeta^i(\d u) \mathds 1 {\lbrace x \in W^i \rbrace}. \] 
Then the marked Poisson point process $\mathbf X^\lambda$ has intensity \begin{equation} \label{újmértékernel} \mu'(\d x, \d u)=\mu(\d x) p(x,\d u), \quad x \in W,~ u \in [\Fmin, \Fmax], \end{equation} and it can be defined on a Polish probability space, and the fadings of all users are bounded between $\Fmin$ and $\Fmax$. Therefore all the estimations from Chapter \ref{eleje} involving the essential bounds of fading variables remain true. Our sprinkling construction from Section \ref{sprinkle} did not use the i.i.d. nature of the fadings, only the fact that the empirical measures of users
\[ \mathbf L_{\lambda}^{\varrho'} ((x,u))=\frac{1}{\lambda} \sum_{X_i \in X^\lambda} \mathds 1 {\lbrace \varrho_1(X_i)=x \rbrace} \mathds 1 {\lbrace \varrho_2(F_{X_i})=u \rbrace} \] are independent and Poisson distributed random variables. For the same reason, Proposition \ref{2.2} also stays true, and this leads us to a generalization of Theorem \ref{1.1} and Corollaries \ref{1.2} and \ref{1.3} to this case. 

Note that in this case the distributions of the fading variables $F^1, F^2, \ldots$ can be both discrete and continuous, even it is possible to have some regions where it is continuous and some where it is discrete (e.g. constant).
\vspace{3pt} 

\emph{2. Continuously varying fading distribution.}\index{fading!depending on space!continuously varying fading} In these examples, the fading distribution will depend continuously on the spatial position. Here we focus on the case when the fading variables are absolutely continuous. There are various ways to construct a similar model with discrete fadings as well.

Let $\Finite$ and \[ p(x,\d u)=f(x, u) \d u, \quad x \in W,~ u \in [\Fmin, \Fmax],  \]
where $f: W \times [\Fmin, \Fmax] \to [\Fmin, \Fmax]$ is a map such that for all $x$, we have that $u \mapsto f(x,u)$ is a continuous probability density supported on $[\Fmin, \Fmax]$, and $x \mapsto f(x,\cdot)$ is continuous in the uniform norm, i.e.
\[ \forall \varepsilon>0~\exists \delta>0~ \forall x,~y \in W: ~\vert x-y \vert < \delta ~\Rightarrow~\sup_{u \in [\Fmin, \Fmax]} \vert f(x,u)-f(y,u) \vert < \varepsilon. \]
By continuity, it suffices to define $f(x,u)$ for rational points $(x,u) \in (W \cap \mathbb Q^d) \times ([\Fmin, \Fmax] \cap \mathbb Q)$ to determine the law of the marked Poisson point process uniquely. Therefore the process can be again defined on a Polish probability space. Similarly to the previous case, the results of Chapter \ref{eleje} can also be generalized to this setting.\footnote{ The setting described in Section \ref{Anfang} corresponds to the case $f(x,u)=h(u)$, where $h$ is the density of the fading variable $F_0$, under the assumptions that $\Finite$ and $F_0$ has a continuous density $h$ on $[\Fmin, \Fmax]$.}

To see an example for such $f$, let $k: W \to (0,\infty)$ be continuous. Then the compactness of $W$ implies that $k$ is also bounded and bounded away from 0. Let $F_0^0$ be a fading variable distributed on $[a,b] \subset (0,\infty)$, with continuous density $f_0^0$. Then for $x \in W$, let us define the law of the space-dependent fading variable $F_0^x$, as follows \begin{equation} \label{szépia} F_0^x\overset{d}{=}k(x)F_0^0, \end{equation} 
i.e. with density
\[ f_0^x(v)=f_0^0(v/k(x)),~v \in [k(x) a, k(x) b]. \]
The probability kernel is given as
\[ p(x, \d u)=f(x,u) \d u=f_0^x(u)\d u, \quad x \in W, u \in [k(x) a, k(x) b], \]
and we have that \[ a \min_{x \in W} k(x) \Fmin=\inf_{x \in W} \essinf F_0^x>0, \quad b \max_{x \in [a,b]} k(x)=\sup_{x \in W} \esssup F_0^x<\infty,  \]
and our continuity assumptions are also satisfied.

Finally we show a simple example of $k$ where it is mathematically clear that in the case of space-dependent loudness, spatial effects and fading effects cannot be separated any more. Let $k(x)=\ell(\vert x \vert)$ be our path-loss function defined in Section \ref{explanation}; this is continuous. In the case of i.i.d. fadings $(F^0_{X_i})_{X_i \in X^\lambda}$, the SIR for e.g. direct uplink communication is given by
\[ \SIR_{\lambda}((X_i,F_{X_i}),~o,~\mathbf L_\lambda)=\frac{\ell(\vert X_i \vert) F^0_{X_i}}{\frac{1}{\lambda} \sum_{X_j \in X^\lambda} \ell(\vert X_j \vert) F^0_{X_j}}. \]
On the other hand, using the trivial path-loss function $\ell_0(r)=1$\index{path-loss!-free setting} and defining the distribution of $F_0^x$ according to \eqref{szépia}, we obtain the same SIR value for each direct uplink transmission as in the previous model with i.i.d. fadings, since for all $i$ we have
\[ \frac{\ell_0(\vert X_i \vert) F^{X_i}_{X_i}}{\frac{1}{\lambda} \sum_{X_j \in X^\lambda} \ell_0(\vert X_j \vert) F^{X_j}_{X_j}}=\frac{\ell(\vert X_i \vert) F^0_{X_i}}{\frac{1}{\lambda} \sum_{X_j \in X^\lambda} \ell(\vert X_j \vert) F^0_{X_j}}. \] The same holds for relayed uplink, and after transforming $F_o$ analogously to \eqref{szépia}, also for downlink communication.

This means that one can consider our model with i.i.d. fadings from Chapter \ref{eleje} from two different perspectives. On the one hand, one can see the fadings as identically distributed ones, and the differences among the average behaviour of signal strengths of users as a consequence of an effect caused entirely by spatial properties. On the other hand, one can say that there is no decay in signal strength caused by path-loss, but the loudness of users decays as their distance from the origin increases. This latter interpretation is also not unnatural. E.g., the following may hold on a larger scale. The vicinity of the base station is the most developed part of the communication area, users located here have better devices, which have therefore higher fading values, and the fading value of the devices of users decays over distance from $o$.
\section{Random fading at the base station}\index{random fading!of the base station} \label{Fernsehturm}
In this section, we extend our model from Section \ref{Anfang} via a random fading value at $o$. First, we give heuristics about the how the choice of the \emph{deterministic} fading value $F_o$ of the origin effects the behaviour of the system.\index{fading!of the base station} Then, using these observations, we define the new model with a random fading at the base station, we compare it to the setting of Chapter \ref{eleje}, and finally we derive a necessary and sufficient condition for the exponential decay of frustration probabilities in this extended setting.

\emph{1. The effects of the deterministic value $F_o$.} Although changing the value of the positive $F_o$ only multiplies $(x,u) \mapsto \SIR((o,F_o)~,(x,u),~\nu)$ by a constant for fixed $\nu \in \mathcal{M}(\mathbf W)$, it may make a substantial change in the path-loss landscape if $g$ is not just the identity truncated at a large value $K$. Here largeness would mean that for all reasonable choices of $F_o$, $g$ acts as the identity on all direct downlink SIR values. If $F_o$ is much larger than the fading variables $\lbrace F_{X_i} \vert~X_i \in X^\lambda \rbrace$ with high probability, and the latter ones are typically small, the QoS for the direct downlink may be equal to its maximal value $\tilde{c}_+$ with very high probability, while all transmissions where the transmitter is not $o$ have a lower QoS. Hence in the vicinity of the origin, where the QoS of direct downlink communication is sufficiently high to make the transmission useful, we may only see direct downlink communication. On the other hand, if $F_o$ is substantially less then the $F_{X_i}$'s with high probability, let us consider a user who is spatially in a large distance from $o$. Then, under the realistic assumption that the path-loss function $\ell: ~[0,\infty) \mapsto (0,\infty)$ is monotone decreasing\footnote{ This assumption is not explicitly written in the model definition of the fading-free setting in \cite[Section 1.1]{cikk}, and therefore we also have not taken it for granted so far, but in concrete examples of this chapter we regularly use it, such as \cite[Section 7]{cikk} does.}, this user as receiver experiences a much lower QoS when receiving a message directly from $o$ than when receiving a message from a $X_i \in X^\lambda$ with a high fading value $F_{X_i} \gg F_o$ who is close to the origin. Therefore, in this case, for users situated further away from the origin, the optimization decision \eqref{gamma} will often lead to relaying via loud users close to the origin, instead of using direct communication. Finally, when $\mathbb E [F_0]$ and $F_o$ are close to one and the variance of $F_0$ is sufficiently small, the statistical behaviour of the system will be similar to the one of the system with the same spatial intensity $\mu$ and with constant fadings equal to 1.

In general, the SIR for direct downlink communication w.r.t. the intensity measure $\mu'$ is given by\index{means of communication!downlink!direct}
\[ \SIR((o,F_o),~x,~\mu')=\frac{\ell(\vert x \vert) F_o}{\int\limits_{\Fmin}^{\Fmax}\int\limits_{W} \ell(\vert y-x \vert) u \mu(\d y)\zeta(\d u)}=\frac{\ell(\vert x \vert) F_o}{\mathbb E[F_0]\int_{W} \ell(\vert y-x \vert) \mu(\d y)}=\frac{F_o}{\mathbb E[F_0]} \SIR(o,~x,~\mu),\]
where on the r.h.s. we used the notation of Section \ref{explanation}. I.e., in the case of constant fading in the origin, the SIR for direct downlink communication w.r.t. the a priori measure is just a constant multiple of the one from the fading-free setting. In particular, if $\mathbb E[F_0]=F_o$ and $g(x)=\min \lbrace x, K \rbrace$ for $K$ sufficiently large, then by the definition of $\Gamma$, the same holds for all direct downlink QoS quantities\footnote{ An easy computation shows that if $\mu(W)>0$, then it suffices to put $K=\sup_{x \in W} \frac{\ellmax F_o}{\int_W \ell(\vert y-x \vert) u ~\mu(\d y)\mathbb E[F_0]}$. The definition of $\Gamma$ implies that this is the maximal SIR level for relayed downlink communication. If $\mu(W)=0$, we can let $K$ be equal to any positive value.}, and hence also for the number of users under SIR level $c \in (0,\tilde{c}_+)=(0,K)$ w.r.t. direct downlink communication. By Corollary \ref{1.3}, the number of users under the direct downlink QoS level $c$ decays exponentially in the fading-free case if and only if it decays exponentially in this special fading configuration, and similarly for direct downlink communication. But in view of Corollary \ref{1.2}, the minimizers of relative entropy and the exponential rates of decay may be different in the two settings. In particular, if $\zeta$ is non-constant, one cannot expect that in the minimizers all users have fading value equal to $\mathbb E[F_0]$.

Looking back to the proofs in Chapter \ref{eleje}, we also see that if we set $F_o$ to be deterministic, then instead of the assumption $F_o=\frac{\Fmin+\Fmax}{2}$ we can let $F_o$ to be any other value in $[\Fmin, \Fmax]$, and our approach will still work. The choice $F_o=\frac{\Fmin+\Fmax}{2}$ is the most convenient one because this is the only point in $[\Fmin, \Fmax]$ which is a sub-cube centre w.r.t. $\delta$-discretization for all $\delta \in \mathbb B \cup \lbrace 1 \rbrace$. In the other cases, in the discretized setting one has to consider $\varrho_2(F_o)$ instead of $F_o$, but this still allows our proofs to work as before. We also note that our approach in Sections \ref{Anfang}--\ref{egykettőegyhárom} does not require $\Fmin$ to be the essential infimum of $F_0$ and $\Fmax$ to be the essential supremum of $F_0$, it just needs that the fadings of the users and also the fading of the origin be essentially bounded between these two positive values. We use this observation for the model definition of this section.

\emph{2. New model definition with random $F_o$.} It is nevertheless plausible to construct a model where, similarly to the uplink scenarios, there is a more substantial difference from the fading-free setting also in the downlink case. Let us consider the case of i.i.d. fadings; similarly to Section \ref{kernel} it can be shown that the case of space-dependent fading is not much more difficult. When fadings are i.i.d., then the interference term w.r.t. the intensity measure $\mu'$ is always equal to $\mathbb E[F_0]$ times of the interference w.r.t. the spatial intensity $\mu$. Additional randomness can be obtained via letting $F_o$ also be a random variable, and assuming the following instead of Assumption \ref{szamár}:
\begin{ass} \label{csacsi}\index{Assumption \ref{csacsi} (i.i.d. fadings, random $F_o$)}
There exists $0<\Fmin$ and $\Fmax<\infty$ such that the fading variable $F_0$ corresponding to the fading distribution of the users satisfies $\mathbb P(\Fmin \leq F_0 \leq \Fmax)=1$. Moreover, let $F_\ast: (\Omega, \mathcal F, \mathbb P) \to \mathbb R$ be a random variable also satisfying $\mathbb P(\Fmin \leq F_o \leq \Fmax)=1$. \\ We assume that for all $\lambda>0$, $\lbrace (X_i, F_{X_i}) \vert X_i \in X^\lambda \rbrace$ is a marked Poisson point process with i.i.d. fadings distributed as $F_0$, and that the fading of origin $F_o$ is a random variable that equals $F_\ast$ in distribution and that is independent from the marked Poisson point process $\mathbf X^\lambda$. 
\end{ass}

The marked Poisson process $\mathbf X^\lambda$ was defined on the separable probability space $(\Omega_2, \mathcal F_2, \mathbb P_2)$ in Section \ref{Anfang}. Now, the one-dimensional extension \[ (\Omega_3, \mathcal F_3, \mathbb P_3)=(\Omega_2 \times \Omega, \mathcal F_2 \tensor \mathcal F, \mathbb P_2 \tensor \mathbb P) \]
is a separable probability space where our marked Poisson point process of user-fading pairs is defined together with the random fading $F_o$ of the origin, which is independent of the marked process. Similarly, the state space of the new process consisting of the marked Poisson point process of the users and the fading of the origin is $\mathbf W^{\mathbb N} \times [\Fmin, \Fmax]$, which is a Polish space.\index{Poisson point process!marked!separable construction} Randomized fading at the origin can be necessary e.g. because the loudness of the base station may vary in time, and then at one fixed time it may have a certain distribution $F_\ast$. Then, our setting from Section \ref{Anfang} views this system at this fixed time, conditional on the value of $F_o$.

\emph{3. Quenched and annealed setting w.r.t. $F_o$.} The behaviour of the new system w.r.t. uplink communication is entirely the same as the one of the system described in Chapter \ref{eleje}, since the uplink QoS quantities do not depend on the fading at the origin. For downlink communication, we obtain the results of Chapter \ref{eleje} \emph{conditional on $F_o$}.

When one encounters two different sources of randomness in a probabilistic model, one can investigate the \emph{quenched}\index{large deviation principle!quenched and annealed} (or: \emph{almost sure}) setting and the \emph{annealed} (or: \emph{averaged}) setting, see e.g. \cite[Section 10]{varadhan}. Annealed setting means taking expectations w.r.t. the two different sources of randomness simultaneously. For instance, we did annealed large deviation analysis in Chapter \ref{eleje}. We had two sources of randomness: the spatial locations of the users and the fadings, where the latter ones are independent conditional on the first ones. We considered expectations with respect to the product measure $\mathbb E_2$, which averages out both sources of randomness. Quenched setting means conditioning on one source of randomness and deriving results about the system this way. Usually one conditions on realizations which satisfy some general conditions that are true almost surely, the name "almost sure setting" originates from here. Now we see that the model defined in Section \ref{Anfang} under $\mathbb P_2$ is in fact a quenched version of the model that we defined here with Assumption \ref{csacsi} under $\mathbb P_3$, conditional on $F_o$. The 1-set to which our typical realizations of $F_o$ belong is the event $\lbrace \Fmin \leq F_o \leq \Fmax \rbrace$. While for the uplink it is indifferent whether we consider the system under $\mathbb P_2$ or under $\mathbb P_3$, considering downlink communication under $\mathbb P_3$ gives rise to a simple annealed analogue of the setting of Chapter \ref{eleje}. I.e., now we consider $F_o$ as one source of randomness, the entire marked Poisson process $\mathbf X^\lambda$, $\lambda>0$ as another source, and instead of conditioning on $F_o$, we analyze the large deviation asymptotics of frustration probabilities under the product measure $\mathbb P_3$.

\emph{4. Exponential decay of unlikely frustration probabilities under the annealed measure.} For $b, c \geq 0$, $\nu \in \mathcal{M}(\mathbf W)$ and $u \in [\Fmin, \Fmax]$ we introduce the notation
\[ G(\nu,~ \tau_c,~ \text{do},~ u)=G(\nu,~ \tau_c,~ \text{do})\vert_{F_o=u}, \] where $G(\nu, \tau_c, \text{do})$ was defined in \eqref{do}. Analogously we put $G(\nu,~ \tau_c,~\text{do-dir},~ u)=G(\nu,~\tau_c,~ \text{do-dir})\vert_{F_o=u}$. Moreover, we write $G(\nu,~ \tau_c,~ \text{do-dir},~ F_\ast)$ and $G(\nu,~ \tau_c,~ \text{do},~ F_\ast)$ as a random variable defined on $(\Omega, \mathcal F, \mathbb P)$ in the case of fixed $\nu$ and $c$.

Then we have, e.g. for direct downlink communication, for any $\lambda>0$
\begin{align*} 
    \mathbb P_3(G(\mathbf L_\lambda,~ \tau_c,~\text{do-dir})(\mathbf W)>b)&=\int_{\Fmin}^{\Fmax} \mathbb P_2( G(\mathbf L_\lambda,~ \tau_c,~\text{do-dir},~u)(\mathbf W)>b) (\mathbb P \circ F_\ast^{-1})(\d u) \\&= \mathbb E(\mathbb P_2( G(\mathbf L_\lambda,~ \tau_c,~\text{do-dir},~F_\ast)(\mathbf W) >b)). \numberthis \label{várhatóértékbácsi}
\end{align*}
Thus, if $\limsup\limits_{\lambda \to \infty} \frac{1}{\lambda} \log \mathbb P_2( G(\mathbf L_\lambda,~ \tau_c,~\text{do-dir},~\Fmin)(\mathbf W)>b)<0$, then $\mathbb P_3(G(\mathbf L_\lambda,~ \tau_c,~\text{do-dir})(\mathbf W) >b)$ decays exponentially as $\lambda \to \infty$. On the other hand, if $\limsup\limits_{\lambda \to \infty} \frac{1}{\lambda} \log \mathbb P_2( G(\mathbf L_\lambda,~ \tau_c,~\text{do-dir},~\Fmax)(\mathbf W)>b)=0$, then $\mathbb P_3(G(\mathbf L_\lambda,~ \tau_c,~\text{do-dir})(\mathbf W)>b)$ does not decay at an exponential speed as $\lambda \to \infty$. These conditions are far from optimal. Instead, the following corollary gives a necessary and sufficient condition for exponential decay of the annealed frustration probabilities. 
\begin{cor} \label{bigfadinginthebasis}\index{frustration probabilities!exponential decay}
Under Assumption \ref{csacsi}, we have for all $\mathbf b \in \mathbb R$ and $\mathbf c \in (0,\tilde{\mathbf c}_+)$ that
\begin{enumerate}[(i)]
    \item $\limsup\limits_{\lambda \to \infty} \frac{1}{\lambda} \log \mathbb P_3( G(\mathbf L_\lambda,~ \tau_c,~\mathrm{do-dir})(\mathbf W)>b_4)<0$ holds if and only if \\ $F_\ast$ has zero mass on $A=\lbrace u:~\limsup\limits_{\lambda \to \infty} \frac{1}{\lambda} \log \mathbb P_2( G(\mathbf L_\lambda,~ \tau_c,~\mathrm{do-dir},~u)(\mathbf W)>b_4)=0 \rbrace$,
    \item $\limsup\limits_{\lambda \to \infty} \frac{1}{\lambda} \log \mathbb P_3( G(\mathbf L_\lambda,~ \tau_c,~\mathrm{do})(\mathbf W)>b_3)<0$ holds if and only if \\ $F_\ast$ has zero mass on $B=\lbrace u:~\limsup\limits_{\lambda \to \infty} \frac{1}{\lambda} \log \mathbb P_2( G(\mathbf L_\lambda,~ \tau_c,~\mathrm{do},~u)(\mathbf W)>b_3)=0 \rbrace$, 
    \item $\limsup\limits_{\lambda \to \infty} \frac{1}{\lambda} \log \mathbb P_3(G(\mathbf L_\lambda,~ \boldsymbol \tau_{\mathbf c})(\mathbf W)>\mathbf b)<0$ holds if and only if \\ $F_\ast$ has zero mass on $C=\lbrace u:~\limsup\limits_{\lambda \to \infty} \frac{1}{\lambda} \log \mathbb P_2( G(\mathbf L_\lambda,~ \tau_c,~u)(\mathbf W)>\mathbf b)=0 \rbrace$.\footnote{ Note that by construction, in the case when $F_\ast$ has an absolutely continuous distribution, $A$, $B$ and $C$ are always either empty or sub-intervals of the support of $F_\ast$ containing the minimal value of $F_\ast$.}
\end{enumerate}
\end{cor}

\begin{proof}
First, we consider (i). We start with proving that if $F_\ast$ has a positive mass on $A$, then $\mathbb P_3( G(\mathbf L_\lambda,~ \tau_c,~\mathrm{do-dir})(\mathbf W)>b_4)$ does not decay at an exponential speed. Equivalently, we show that the exponential decay of $\mathbb P_3( G(\mathbf L_\lambda,~ \tau_c,~\mathrm{do-dir})(\mathbf W)>b_4)$ implies (iii).

For the first, let us consider the simple sub-case when $F_\ast$ is a discrete random variable with finitely many values $F^1 < \ldots < F^n$, $\mathbb P(F_\ast =F^i)=p_i$, where $p_i>0$, $\forall i=1,\ldots,n$ and $\sum_{i=1}^n p_i=1$. Assume $\limsup\limits_{\lambda \to \infty} \frac{1}{\lambda} \log \mathbb P_2( G(\mathbf L_\lambda,~ \tau_c,~\text{do-dir},~F^1)(\mathbf W) >b)=0$, i.e., at least in the case of the the smallest possible fading value of the origin, in the setting of Chapter \ref{eleje} associated to this fixed fading value, the considered frustration probability does not decay exponentially. Then we can estimate
\begin{align*} \limsup_{\lambda \to \infty} \frac{1}{\lambda} \log \mathbb P_3( G(\mathbf L_\lambda,~ \tau_c,~\text{do-dir})(\mathbf W)>b) &= \limsup_{\lambda \to \infty} \frac{1}{\lambda} \log \mathbb E(\mathbb P_2( G(\mathbf L_\lambda,~ \tau_c,~\text{do-dir},~F_o)(\mathbf W)>b)) \\ &=\limsup_{\lambda \to \infty} \frac{1}{\lambda} \log \sum_{i=1}^n p_i \mathbb P_2( G(\mathbf L_\lambda,~ \tau_c,~\text{do-dir},~F^i)(\mathbf W)>b) \\ &\geq \limsup_{\lambda \to \infty} \frac{1}{\lambda} \log p_1 \mathbb P_2( G(\mathbf L_\lambda,~ \tau_c,~\text{do-dir},~F^1)(\mathbf W)>b) \\ &=\limsup_{\lambda \to \infty} \frac{1}{\lambda} \log \mathbb P_2( G(\mathbf L_\lambda,~ \tau_c,~\text{do-dir},~F^1)(\mathbf W)>b)=0. \numberthis \label{ldp0}  \end{align*}
We can generalize this to the case when the distribution of $F_\ast$ has a positive mass on $A$. Then there exists $u_1 \leq u_2 \in [\Fmin, \Fmax]$ such that $\limsup\limits_{\lambda \to \infty} \frac{1}{\lambda} \log \mathbb P_2( G(\mathbf L_\lambda,~ \tau_c,~\text{do-dir},~u_2)(\mathbf W)>b)=0$ and $F_\ast$ has a positive mass on $[u_1,u_2]$, which we denote as $p=\mathbb P(F_\ast \in [u_1,u_2])$. Then for all $\lambda>0$, $\mathbb E(\mathbb P_2( G(\mathbf L_\lambda,~ \tau_c,~\text{do-dir},~F_\ast)(\mathbf W)>b))$ can be estimated from below by $\mathbb E(\mathbb P_2( G(\mathbf L_\lambda,~ \tau_c,~\text{do-dir},~F_\sharp)(\mathbf W)>b))$. Here $F_\sharp$ is a discrete fading variable with weights $p=\mathbb P(F_\sharp=u_2)=1-\mathbb P(F_\sharp=\Fmax)$ if it is possible to choose $u_2$ such that $u_2 \neq \Fmax$, otherwise $P(F_\sharp=u_2=\Fmax)=1$. It is clear from \eqref{ldp0} that $\mathbb E(\mathbb P_2( G(\mathbf L_\lambda,~ \tau_c,~\text{do-dir},~F_\ast)(\mathbf W)>b))$ does not decay exponentially as $\lambda \to \infty$, and therefore also $\mathbb E(\mathbb P_2( G(\mathbf L_\lambda,~ \tau_c,~\text{do-dir},~F_o)(\mathbf W)>b))$ does not decay at an exponential speed. Hence, noting that the equality in \eqref{várhatóértékbácsi} implies that $\mathbb P \circ F_\ast^{-1}$ nullsets do not interplay in determining large deviation behaviour of frustration probabilities w.r.t. $\mathbb P_3$, we have proven our statement that the exponential decay $\mathbb P_3( G(\mathbf L_\lambda,~ \tau_c,~\mathrm{do-dir})(\mathbf W)>b_3)$ implies (iii). The converse of his statement follows from \eqref{várhatóértékbácsi} and Fubini's theorem. 

Second, we note that the equation \eqref{várhatóértékbácsi} and the estimation \eqref{ldp0} uses only the definition of $\mathbb P,~ \mathbb P_2$ and $\mathbb P_3$ and no specific property of the direct downlink\index{means of communication!downlink!relayed} that is not true for the general downlink. Using also the observation that the uplink communication is independent of the value of $F_o$, (ii) and (iii) can be proven analogously to (i).


\end{proof}


In particular, Corollary \ref{bigfadinginthebasis} implies that if any of the following conditions holds:
\begin{enumerate}
\item $F_\ast$ has zero mass on $A$,
\item $F_\ast$ has zero mass on $B$,
\item $\mathbb P_2(G(\mathbf L_\lambda,~\tau_{c_1},~\text{up})(\mathbf W)>b_1)$ decays exponentially,
\item $\mathbb P_2(G(\mathbf L_\lambda,~\tau_{c_2},~\text{up-dir})(\mathbf W)>b_2)$ decays exponentially,
\end{enumerate}
then we have that $\mathbb P_3(G(\mathbf L_\lambda,~\boldsymbol \tau_{\mathbf c})(\mathbf W)>\mathbf b)$ decays exponentially. However, it is in general not true that if none of these conditions holds, then $ \mathbb P_3(G(\mathbf L_\lambda,~\boldsymbol \tau_{\mathbf c})(\mathbf W)>\mathbf b)$ does not decay exponentially. In this case, we only know that $\lbrace G(\mathbf L_\lambda,~\boldsymbol \tau_{\mathbf c})(\mathbf W)>\mathbf b \rbrace$ is an intersection of four events which themselves have asymptotically subexponential probabilities, but the probability of the intersection may decay at an exponential speed. Instead, a correct equivalent condition to the subexponentional behaviour of $\mathbb P_3(G(\mathbf L_\lambda,~\boldsymbol \tau_{\mathbf c})(\mathbf W)>\mathbf b)$ is (iii) of Corollary \ref{bigfadinginthebasis}.

For more precise description about when the conditions of this corollary hold, we refer to Section \ref{kalsszikus}. We note that there, unlike in Sections \ref{Anfang}--\ref{egykettőegyhárom}, there was a delicate dependence on the minimal SIR vector, where it was assumed that $\Fmin$ was indeed the essential minimum and $\Fmax$ the essential maximum of $F_0$. Hence, in order to formulate an analogous statement to Corollary \ref{új1.4} under Assumption \ref{csacsi}, one has to redefine the coordinates of the minimal SIR vector as $K_{\text{up-dir}}=g \left( \frac{\ellmin \mathbb P\text{-}\essinf F_0}{\int_{\mathbf W} \ell(\vert x \vert) \mu'(\d x, \d u)} \right)$ and $K_{\text{do-dir}}=g \left( \mu\text{-}\essinf_{y \in W} \frac{\ellmin \mathbb P\text{-}\essinf F_\ast}{\int_{\mathbf W} \ell(\vert x-y \vert) \mu'(\d x, \d u)} \right)$,
\[ K_{\text{up}}=\mu\text{-}\essinf_{(x,u) \in \mathbf W} \mu\text{-}\esssup_{(y,v) \in \mathbf W} \Gamma((x,u)~(y,v),~(o,\mathbb P\text{-}\essinf F_\ast),~\mu'), \text{ and} \] \[  K_{\text{do}}=\mu\text{-}\essinf_{(x,u) \in \mathbf W} \mu\text{-}\esssup_{(y,v) \in \mathbf W} \Gamma((o,\mathbb P\text{-}\essinf F_\ast)~(y,v),~(x,u),~\mu'). \] 
Defining the minimal SIR vector\index{SIR!minimal SIR vector} as $\mathbf K=(K_{\text{up}},~K_{\text{up-dir}},~K_{\text{do}},~K_{\text{do-dir}})$, one can argue exactly the same way as in Section \ref{kalsszikus} to obtain a full characterization of exponential decay of frustration probabilities in terms of the a priori measure $\mu'$.\footnote{ Note that the quantities $K_{\text{up}}$ and $K_{\text{up-dir}}$ are exactly the same as in Section \ref{kalsszikus}, since the SIR does not depend on the fading value of the receiver.} This simply means expressing the cases of Corollary \ref{bigfadinginthebasis} in terms of the cases of Corollary \ref{új1.4}, which we therefore leave for the reader.
\chapter{Summary} \label{summary}
In this Master's thesis, we generalized the results of the paper \cite{cikk} about large deviation properties of frustration probabilities in a wireless network on a compact communication area. The individual task of the thesis was to incorporate random fadings in the model, which are interpreted as the loudnesses of the users. The pairs of spatial positions and loudnesses $\mathbf X^\lambda=\lbrace (X_i, F_{X_i}) \rbrace_{X_i \in X^\lambda}$ of the users form a marked Poisson point process. In the simplest scenario, defined in Sections \ref{Anfang} and \ref{discretization}, we had the following assumptions on the fadings. Conditional on the spatial positions of the users, their fadings are i.i.d., they equal the fading variable $F_0$ in distribution, and the base station has a fixed fading value $F_o$. Moreover, the fadings are bounded from above and bounded away from 0, i.e. for the essential bounds of $F_0$ we have $\Finite$.

This allowed us to consider uplink and downlink communication with the base station, also with relaying with one hop at maximum. Hence, we regarded all means of communication described in \cite[Section 1.1]{cikk}, in a more realistic setting. The assumption $\Finite$ was necessary in order to be able to generalize the arguments in \cite{cikk}. This way we proved Theorem \ref{1.1}, Corollaries \ref{1.2} and \ref{1.3}, which are the random-fading analogues of \cite[Theorem 1.1, Corollaries 1.2 and 1.3]{cikk} respectively. I.e, with the general technical result Theorem \ref{1.1}, we proved that in our setting the exponential rate of decay of the frustration probability $\mathbb P_2(G(\mathbf L_\lambda,~\boldsymbol \tau_{\mathbf c})>\mathbf b)$ is given by the infimum of relative entropy w.r.t. the a priori measure $\mu'$, taken among finite measures on $\mathbf W$ that also exhibit more than $\mathbf b$ frustrated users under SIR level $\mathbf c$. In particular, we have seen that if the frustration probability is unlikely w.r.t. $\mu'$, then it decays exponentially. 

The proof techniques in Chapter \ref{eleje} are similar to the ones in \cite[Sections 3--6]{cikk}. Analogously to the fading-free case, it follows from Sanov's theorem (Theorem \ref{nagysanov}) that the empirical measures $\mathbf L_\lambda =\frac{1}{\lambda} \sum_{X_i \in X^\lambda} \delta_{(X_i, F_{X_i})}$, $\lambda>0$, satisfy a large deviation principle with good rate function $\nu \mapsto h(\nu \vert \mu')$ on $\mathcal{M}(\mathbf W)$. The main technical difficulty in this chapter was that the number of frustrated users
\[ \nu \mapsto G(\nu,~\tau_{c},~\text{up})(\mathbf W)=\int_{\mathbf W} \mathds 1 {\lbrace D((x,u),~o,~\nu)<c \rbrace}~ \nu(\d s, \d u) \]
is a functional depending on $\nu$ integrated w.r.t. the same measure $\nu$, and therefore it may be discontinuous. We have seen that the maps $\nu \mapsto G(\nu,~\tau_c,~\text{up})$ and $\nu \mapsto G(\nu,~\tau_c,~\text{do})$ are u.s.c. but not l.s.c. This is caused purely by the effect of relaying: disappearence of potential relays may cause a sudden decrease in the QoS for users who could only obtain a satisfactory QoS using these relays. In particular, combining Lemma \ref{3.6} and \cite[Lemma 3.7]{cikk} yields that the number of frustrated users for the direct communication cases $\nu \mapsto G(\nu,~\tau_c,\text{up-dir})$ and $\nu \mapsto G(\nu,~\tau_c,\text{do-dir})$ are continuous. The discontinuities prevented us from applying the contraction principle or Varadhan's lemmas directly to conclude Theorem \ref{1.1} and its two corollaries. Instead, analogously to the fading-free case, we discretized the spatial dimension, and here also the fading dimension. The assumption $\Finite$ implies that the discretized range of fading values $\Fdelta$ and the discretized space-fading landscape $\mathbf W_\delta$ are finite. This property is inevitable if one wants to follow the approach of \cite{cikk}. 

The extra dimension provided by the fadings made several estimations longer, e.g. in Section \ref{kettőegy}, where one could observe the spatial effects and the fading effects separately. Using the finiteness of $\mathbf W_\delta$ and the fact that the marked Poisson point process is still Poissonian according to the Marking Theorem (Theorem \ref{marking}), we were also able to work out an analogue of the sprinkling construction described in \cite[Section 3.2]{cikk}, which is the most essential tool of the original paper. Here, using thinning manipulations on the discretized Poisson point process, one proves that configurations that cause discontinuity of our functionals of interest are negligible on the exponential scale. 
We exploited the sprinkling arguments in several proofs: the ones of Proposition \ref{2.2}, Corollary \ref{1.2} and Corollary \ref{1.3}.

In Section \ref{kalsszikus}, we determined a necessary and sufficient condition for the exponential decay of the frustration probabilities $\mathbb P_2( G(\mathbf L_\lambda,~\boldsymbol \tau_{\mathbf c})(\mathbf W)>\mathbf b)$ for fixed $\mathbf b, \mathbf c$, in terms of the empirical measure $\mu'$. Such a classification does not appear in \cite{cikk}, but for most of the cases the result follows easily from Corollary \ref{1.2}. The case that requires additional work is when $G(\mu',~\boldsymbol \tau_{\mathbf c})(\mathbf W)=0$. Then we experience exponential decay, unless no coordinate of $\mathbf c$ is strictly under the QoS level corresponding to the minimal SIR level that can be experienced in the system. In this latter subcase, at least one coordinate $c_i$ of $\mathbf c$ equals the corresponding minimal SIR level, and typically the number of users under SIR level $c_i$ w.r.t. the type of communication $m_i$ becomes positive from time to time as $\lambda \to \infty$. Apart from this case, it is necessary that $G((1+\varepsilon)\mu',~\boldsymbol \tau_{\mathbf c})(\mathbf W)\leq \mathbf b$ hold for some $\varepsilon>0$ in order to obtain exponential decay of $\mathbb P_2( G(\mathbf L_\lambda,~\boldsymbol \tau_{\mathbf c})(\mathbf W)>\mathbf b)$. In other words, not surprisingly, non-unlikely frustration probabilities do not decay exponentially.

Having the general results of Chapter \ref{eleje}, in Chapter \ref{effect} we were investigating the effects coming from the randomness of the fadings more closely. We were looking for formulas which show explicit dependence on the fading distribution, and to compute the average loudness of users in the measures that minimize the relative entropy \eqref{minike} corresponding to Corollary \ref{1.2}. Handling these problems required different levels of specifications in the model of Section \ref{Anfang}, and in Section \ref{I'm a simulant} also some numerical computations and simulations in a very special case. First, in Section \ref{ráadás}, following the variational calculus approach of \cite[Section 7.1]{cikk}, we derived a formula for the minimizers of relative entropy for direct uplink communication in a special case of the setting of Section \ref{Anfang}. Here the spatial allocation of the users was two-dimensional and rotationally invariant, but the fading variable could have a quite general absolutely continuous distribution. We have seen that the minimizer \eqref{ojlerlagrandzs} is rotationally symmetric, and it depends exponentially on the density of the fading variable. However, there is no known way to find explicit minimizers of relative entropy for the other three means of communication, because in these cases, symmetry breaking may occur, as seen in \cite[Section 7]{cikk}. 

As a new development in the thesis, in the rest of Chapter \ref{effect} we investigated various fading-dependent features of the system. In Section \ref{ocsú}, we introduced a very special but realistic setting, the \emph{path-loss-free} model. Here the large deviation properties of frustration probabilities can easily be handled analytically, which is not true for the direct downlink for general. In particular, the minimizers of the rate function are again rotationally invariant, and they can also be determined using variational calculus. Here we saw that the minimizers depend on the fading distribution, but as long as the fadings are bounded and bounded away from zero, very low SIR levels can only occur in the system if the \emph{total number of users is unusually large}. 

This observation is supported by the simulation results of Section \ref{I'm a simulant}. There we considered direct uplink communication with uniformly distributed fadings on the interval $[1,2]$, constant spatial intensity and path-loss corresponding to Hertzian propagation. Also through this empirical approach, we have seen that the main characteristic of the rare frustration events is the very large number of users. We have not experienced that in these rare setting the average loudness is unexpectedly large. The reason for this is that whenever the number of users is not larger than on average but their fadings or path-losses (or both) are too large, the increase in the numerator of SIR is cancelled by the increase in the denominator of SIR. The same cancellation effect does not interplay when the number of users is too large, because the cancellation factor $\frac{1}{\lambda}$ in the interference depends on the intensity and not on the number of users in the corresponding realization of the marked Poisson point process. Another effect that may contribute here is that the number of users in $W$ is Poisson distributed, hence unbounded, while the path-losses and fadings of users are essentially bounded. Hence, in order to obtain the most likely frustration configuration, it is entropically favourable to increase the number of users and keep their fadings and path-losses as close to the average behaviour as possible.

In Section \ref{bélus}, we returned to the analytical approach and proved that generally in the model of Section \ref{Anfang}, the limit $\lim_{\mathbf b \downarrow 0} \lim_{\lambda \to \infty} \frac{1}{\lambda} \mathbb P_2(G(\mathbf L_\lambda,~\boldsymbol \tau_{\mathbf c})(\mathbf W)>\mathbf b)$ equals $\lim_{\lambda \to \infty} \frac{1}{\lambda} \mathbb P_2(G(\mathbf L_\lambda,~\boldsymbol \tau_{\mathbf c})(\mathbf W)>0)$, which is the exponential decay rate of the probability of having at least a single user under QoS vector $\mathbf c$. In particular, we interpreted this result in the special setting of Section \ref{ráadás} for direct uplink communication. There we saw that the variational calculus approach provides a minimizer also for the case $b=0$, and for fixed $c$, the density of this minimizer is the uniform limit of the densities corresponding to $b$ as $b \downarrow 0$. Again, the case distinction according to whether $c$ equals the minimal SIR level or not plays an important r\^{o}le. If yes, then the minimizer for $b=0$ is the a priori measure $\mu'$. If not, then the minimizer for $b=0$ can be obtained by a nontrivial exponential tilting from $\mu'$. 

The last two sections of Chapter \ref{effect} describe modifications of our model defined in Section \ref{Anfang}. These modified settings are more realistic, and we showed that some of the main results of Chapter \ref{eleje} also hold in these cases. Section \ref{kernel} describes space-dependent fading. There, the distribution of the fading of a user may depend of the spatial location of this user, but it has to be independent of the spatial positions and fadings of all other users. $\mathbf X^\lambda=\lbrace (X_i, F_{X_i}) \rbrace_{X_i \in X^\lambda}$ will therefore still be a marked Poisson point process, and as long all the fadings are uniformly bounded, Theorem \ref{1.1} and Corollaries \ref{1.2}, \ref{1.3} will maintain to hold. We showed two examples of such models: the case of finitely many areas (finitely many possible fading distributions), and the case of in space continuously varying fading distribution. We interpreted by an example that in the case of space-dependent fading, spatial effects and fading effects cannot be separated any more.

Section \ref{Fernsehturm} describes the extension of our model with a random fading variable at the base station $o$, which was assumed to be constant in Chapter \ref{eleje}. This only changes the model w.r.t. downlink communication. We assumed that the fading distribution at the base station is also bounded and bounded away from 0, and the fading of the origin is independent of the entire marked Poisson process of user-fading locations. We showed that in this new setting, analogues of the results of Chapter \ref{eleje} hold on an extended probability space. 

\emph{Summa summarum}, the main characteristics of our model with random fadings are the following. According to Chapter \ref{eleje}, as long as we assume that $\Finite$, an analogue of the approach of \cite[Sections 2--6]{cikk} works and we obtain the same general results. In Chapter \ref{effect}, we have seen several properties of the system that depend delicately on the fading distribution, and which do not follow from the fading-free setting. However, it seems to be the case that the most important characteristic of the rare frustration events is not unlikely large (or little) average fading or average path-loss, but unlikely large number of users.

In the following, we enumerate several imporant open questions along which related research can be continued.
\section{Open questions} \label{folyt}
\begin{itemize}
    \item Can one extend the large deviation results of this Master's thesis to settings where instead of our essential boundedness assumption for the fadings $\Finite$, we merely require $\mathbb E[F_0]<\infty$? \\ It is clear that the approach of \cite{cikk} cannot be used in this more general case without substantial modifications. The difficulty is that the continuity lemmas described in our Section \ref{pajti} and the sprinkling construction described in Section \ref{sprinkle} fail if we do not have positive and finite essential bounds of $F_0$. Also, the estimations in the proof of Proposition \ref{2.1} do not work any more.
    \item If it is possible to extend the model to only moment assumptions on $F_0$, can one also generalize the model to the case where the path-loss function is unbounded? I.e., when instead of Lipschitz continuity of $\ell$, we have $\lim_{r \downarrow 0} \ell(r)=\infty$? \\
    For instance, the setting of the paper \cite{KB} uses the path-loss function $\ell(r)=K r^{-\beta}$, $\beta>0$, which has this property. The model of this article differs from ours for several reasons, e.g. the spatial communication area is not bounded, and only one Poisson point process of users is considered, without any $\lambda$-dependence. But there are also random fadings in this setting. Here, by \cite[Proposition 3]{KB}, the sequence of the SIR values of the users in decreasing order is a two-parameter Poisson--Dirichlet process. If the results of our model were true for the case of explosion of $\ell$ at 0, one could consider the SIR sequence from this perspective, and try to connect our large deviation results with the large deviation results on the Poisson--Dirichlet process, e.g. \cite[Theorem 3.4]{fengshui}. In the setting of \cite{KB}, this latter LDP corresponds to the limit $\beta \to 0$, i.e. when the path-loss function converges pointwise to a positive constant function.
    \item Can one extend the setting of this Master's thesis to the case of more than one hop (i.e., with relaying through more than one relay possible)? \\
    This question belongs to the current research topic of the Leibniz Group "Probabilistic methods for mobile ad-hoc networks" lead by Wolfgang König at WIAS Berlin. \\
    In this thesis, we have been able to generalize the formal results of \cite{cikk} conveniently to the case of random fadings as long as $\Finite$. Also we could take the fading of $o$ to be random, but also bounded between $\Fmin$ and $\Fmax$. Since this is true for relayed communication with at most one hop, one can conjecture that a similar generalization is possible for the case of maximum $k<\infty$ hops. In more complex relaying models where the number of hops is not bounded, different results may appear.
    \item Can one find the minimizers for the setting of Section \ref{ráadás} analytically, for the means of communication apart from the direct uplink and the path-loss-free direct downlink? Can one prove (not only conjecture by simulations) that there exist anisotropic (not radially symmetric) minimizers of relative entropy for these ways of communication? \\
    According to \cite[Section 7.2--7.3]{cikk}, it would also be a substantial development to achieve this in the fading-free model, and it is also not clear that it would imply analogous results in the case of random fadings.
    \item \emph{The quenched case}: what is the large deviation behaviour of the frustration probabilities conditional on the fading variables? \\
    In our model from Section \ref{Anfang}, we have two sources of randomness: the spatial configurations (and in particular: the number) of users and the values of the fadings. Throughout this thesis, we considered frustration probabilities under the product probability measure $\mathbb P_2$, which averages out these two sources of randomness simultaneously. Using the terminology introduced in Section \ref{Fernsehturm}, our setting is annealed. It is also interesting to consider the quenched version of the setting, when one conditions on the sequence of all fading variables and investigates large deviation properties of the system that are true for almost every fading sequence. These may turn out to be significantly different from the annealed ones detailed in this thesis.\index{large deviation principle!quenched and annealed}
\end{itemize}
\appendix
\newpage
\pagenumbering{Roman}
\pagestyle{plain}
\chapter{Appendix}


\section{Zusammenfassung in deutscher Sprache}
\begin{center} \textbf{Dichte Mobilfunknetzwerke mit zufälligen Fadings} \end{center}
In dieser Masterarbeit analysieren wir das asymptotische Verhalten großer Abweichungen der Interferenz in einem zufälligen Netzwerk, wenn die Nummer der Benutzer gegen Unendlich geht. Wir verallgemeinern das Modell des Artikels \cite{cikk}. Die Aufgabe meiner Masterarbeit war zufällige Fadings, die als Lautstärken der Benutzer interpretiert werden, in das Modell einzuarbeiten. Wir interessieren uns dafür, was der Grund für schlechte Verbindungen ist. \emph{A priori} gibt es zwei Möglichkeiten: dass sich zu viele Benutzer direkt nebeneinander sammeln und damit zu große Interferenz verursachen, oder dass einige Benutzer zu laut sind, und deshalb die Interferenz unerwartet groß wird.

Um diese Fragen zu beantworten, organisieren wir die Masterarbeit wie folgt. Im Kapitel \ref{előzés} erklären wir grundlegende Resultate über große Abweichungen und Poisson-Punktprozesse, die wir brauchen, um die Denkansätze des Artikels \cite{cikk} verallgemeinern zu können. Danach stellen wir das Modell und die Hauptresultate aus \cite{cikk} vor.

Im Kapitel \ref{eleje} definieren wir ein neues Modell mit zufälligen Fadings und beweisen die Analoga der Sätze aus \cite[Section 1.2]{cikk}. Sei $\lambda>0$ die Dichte der Benutzer. Die Orte der Benutzer werden durch einen Poisson-Punktprozess $X^\lambda$ mit der Intensität $\lambda \mu(W)$ auf der kompakten Teilmenge $W \subset \mathbb R^d$ mit $r,d \in \mathbb N$ definiert. Hier ist $\mu$ ein endliches, absolut stetiges Borelmaß auf der Menge $W$. Sei $X_i \in X^\lambda$, dann bezeichnet $F_{X_i}$ das Fading von $X_i$, und für ein gegebenes $X^\lambda$ sind $\lbrace F_{X_i} \rbrace_{X_i \in X^\lambda}$ unabhängige und identisch verteilte (i.i.d.), positive Zufallsvariablen. Wir nehmen auch an, dass $\mathbb P(F_{X_i} \in [\Fmin, \Fmax])=1$, wobei $0<\Fmin<\Fmax<\infty$ gilt. Deshalb ist $\mathbf X^\lambda = \lbrace (X_i, F_{X_i})_{X_i \in X^\lambda} \rbrace$ ein markierter Poisson-Punktprozess mit Werten in $\mathbf W=W \times [\Fmin, \Fmax]$ und der Intensität $\mu'(\d x, \d u)=\mu(\d x) \mathbb P(F_{X_i} \in \d u)$. Das empirische Maß der Benutzer ist $ \mathbf L_\lambda=\frac{1}{\lambda} \sum_{X_i \in X^\lambda} \delta_{(X_i,F_{X_i})};$ das ist ein zufälliges Element der Menge $\mathcal{M}(\mathbf W)$ der endlichen Maße auf $\mathbf W$. Weiterhin ist $\ell: [0,\infty) \to (0,\infty)$ die Pfadverlust-Funktion, die global Lipschitz-stetig ist. Die Empfangsqualität der Nachricht, die von $(X_i,F_{X_i}) \in \mathbf X^\lambda$ gesendet und gleichzeitig von $x \in W$ empfangen wird, ist gegeben durch die \emph{signal-to-interference ratio (SIR)}
\[ \SIR((X_i,F_{X_i}),~x,~\mathbf L_\lambda)=\frac{\ell(\vert X_i-x \vert) F_{X_i}}{\frac{1}{\lambda} \sum_{X_j \in X^\lambda} \ell(\vert X_j-x \vert) F_{X_j}}.\] Die Benutzer wollen mit der Basisstation $o$ kommunizieren, die sich im Ursprung von $\mathbb R^d$ befindet und einen konstanten Fadingwert $F_o$ hat. Alle Benutzer können Nachrichten zu $o$ schicken (\emph{Uplink}-Szenario) oder Nachrichten von $o$ bekommen (\emph{Downlink}-Szenario). Im Modell ist es auch möglich, die Nachricht über maximal einen anderen Benutzer weiterzuleiten. Die SIR der weitergeleiteten Kommunikation ist das Minimum der zwei dazugehörigen SIR-Werte. Das heißt, alle Benutzer können mit $o$ direkt kommunizieren, aber falls es einen anderen Benutzer gibt, durch den die SIR verbessert werden kann, dann wählt der Benutzer diesen Umweg.

Falls $\SIR((X_i,F_{X_i}),~o,~\mathbf L_\lambda)$ kleiner als ein kritischer Wert $c$ ist, dann ist die Verbindung nicht gut und der Benutzer $(X_i,F_{X_i})$ \emph{frustriert}. Es ist nützlich, die Definition der SIR für andere endliche Maße $\nu \in \mathbf W$ wie folgt zu verallgemeinern
\[ \SIR((x,u),~(y,v),~\nu)=\frac{\ell(\vert x \vert) u}{\int_{\mathbf W} \ell(\vert x-y \vert) v \nu(\d y, \d v)}. \]
Dann definiert man z.~B. bei direkter Uplink-Kommunikation
\begin{equation}\label{csún} G(\nu,~\tau_c,~\text{up-dir})(\cdot)=\int_{\cdot} \mathds 1 \lbrace \SIR((x,u),~(o,F_o),~\nu) < c \rbrace \nu(\d s, \d u), \end{equation}
das ist das empirische Maß der Benutzer, deren Empfangsqualität unter dem SIR-Niveau $c$ liegt. Deshalb ist $G(\nu,~\tau_c,~\text{up-dir})(\mathbf W)$ die Anzahl der Benutzer unter SIR-Niveau $c$ hinsichtlich des Maßes $\nu$. Die Anzahl der frustrierten Benutzer bei den anderen Kommunikationstypen wird analog bestimmt. Wir interessieren uns für das Verhalten des Frustrationereignis bei großen Abweichungen
$G(\mathbf L_\lambda,~\boldsymbol \tau_{\mathbf c})(\mathbf W)>\mathbf b,$
wobei $\mathbf c=(c_i)_{i \in \four}$ und $\mathbf b=(b_i)_{i \in \four}$ Elemente von $\mathbb R^4$. Die Bedeutung dieses Ereignisses ist, dass mehr als $b_i$ Benutzer entsprechend des Kommunikationstyps $m_i$ einen SIR kleiner als $c_i$ haben $\forall i \in \four$. Hier heißt $m_1$ (möglicherweise weitergeleitete) Uplink-Kommunikation, $m_2=\text{up-dir}$ nur direkte Uplink-Kommunikation, $m_3$ Downlink-Kommunikation und $m_4$ direkte Downlink-Kommunikation. 

Die Hauptergebnisse dieser Masterarbeit beschreiben die exponentiellen Abfallraten der Frustrationswahrscheinlichkeiten. Sie benutzen den Begriff der \emph{relativen Entropie}. Sei $\nu$, $\nu' \in \mathcal{M}(\mathbf W)$, dann ist die relative Entropie des Maßes $\nu$ hinsichtlich des Maßes $\nu'$ als
\[ h(\nu \vert \nu')=\begin{cases} \int_{\mathbf W} \frac{\d \nu}{\d \nu'} \log \frac{\d \nu}{\d \nu'} \d \nu'-\nu(\mathbf W)+\nu'(\mathbf W), \quad \text{ falls } \frac{\d \nu}{\d \nu'}\text{ existiert,} \\ \infty \quad \quad \quad \quad \quad \quad \quad \quad \quad \quad \quad \quad \quad \quad \quad \quad \text{   sonst} \end{cases}\]
definiert. Sei $\mathbf b,\mathbf c \in \mathbb R^4$, sodass für jedes $i \in \four$ $c_i$ positiv und kleiner als der maximale mögliche SIR-Wert $\tilde{c}_+$ ist. Dann impliziert Korollar \ref{1.2}, dass
\begin{equation} \label{megint} \lim_{\lambda \to \infty} \frac{1}{\lambda} \log \mathbb P_2 (G(\mathbf L_\lambda, ~\boldsymbol \tau_{\mathbf c})(\mathbf W)>\mathbf b)= - \inf\limits_{\nu \in \mathcal{M}(\mathbf W):~G(\nu, ~\boldsymbol \tau_{\mathbf c}) (\mathbf W) > \mathbf b} h(\nu \vert \mu') \end{equation}
gilt. Das unter dem Infimum geschriebene Maß ist das wahrscheinlichste Konfiguration in Abhängigkeit von den seltenen Frustrationsereignis. 

Weiterhin zeigt Korollar \ref{1.3}, dass falls ein Frustrationsereignis hinsichtlich der Intensität $\mu'$ unwahrscheinlich ist, d.~h. falls ein $\varepsilon>0$ mit $G((1+\varepsilon)\mu',\boldsymbol \tau_{\mathbf c})(\mathbf W) \leq \mathbf b$ existiert, wobei $b_i \geq 0$ und $0<c_i <\tilde{c_i}_+$ für alle $i \in \four$, dann fallen die Frustrationswahrscheinlichkeiten exponentiell schnell ab:
\[ \limsup_{\lambda \to \infty} \frac{1}{\lambda} \log \mathbb P_2 (G(\mathbf L_\lambda, ~\boldsymbol \tau_{\mathbf c})(\mathbf W)>\mathbf b)<0. \]

Die Beweisideen dieses Satzes lauten wie folgt. Der Satz von Sanov impliziert, dass die empirischen Maße $\lbrace \mathbf L_\lambda \rbrace_{\lambda>0}$ einem Prinzip der großen Abweichungen mit guter Ratenfunktion $\nu \mapsto h(\nu \vert \mu')$ genügen. Bei Beweis von \eqref{megint} ist die Schwierigkeit, dass $\nu \mapsto G(\nu,~\tau_c)$ unstetig ist und deshalb können die Lemmas von Varadhan (Lemma \ref{varadhanupper} und Lemma \ref{varadhanlower} in dieser Masterarbeit) nicht direkt angewandt werden. Deswegen diskretisieren wir $\mathbf W$ und benutzen die Stetigkeit, um Rückschlüsse auf das originale Modell zu ziehen. Einige von unseren Hilfsätzen benutzen die Ausdünnung (\emph{thinning}) der diskretisierten Poisson-Punktprozesse. 

In Abschnitt \ref{kalsszikus} am Ende des dritten Kapitels befindet sich eine vollständige Einteilung, die zeigt, unter welchen Bedingungen $\mathbb P_2 (G(\mathbf L_\lambda, ~\boldsymbol \tau_{\mathbf c})(\mathbf W)>b)$ exponentiell schnell abfällt. Wenn die Bedingungen von Korollar \ref{1.3} nicht erüllt sind, gibt es exponentiellen Abfall nur falls $\mathbf b=0$ und mindestens eine Koordinate von $\mathbf c$ kleiner als der minimale SIR-Wert des entsprechenden Kommunikationstyp ist.

Im Kapitel \ref{effect} analysieren wir die Wirkung von zufälligen Fadings. Deshalb sind viele Teile dieses Kapitels keine Analoga von entsprechenden Teilen des Artikels \cite{cikk}. In Abschnitt \ref{ráadás} wird die direkte Uplink-Kommunikation in einem Spezialfall unseres Modells von Kapitel \ref{eleje} betrachtet. Dabei gehen wir von einem zweidimensionalem Raum und rotationssymmetrischen Intensitäten aus und die Fadings können eine absolut stetige Verteilung auf $[\Fmin, \Fmax]$ haben. In diesem Fall kann man die Dichte des minimalisierenden Maßes von \eqref{megint} mit Variationsrechnung finden für alle $b >0$ und $c \in (0,\tilde{c}_+)$. In Abschnitt \ref{ocsú} analysieren wir den Sonderfall, in dem der Pfadverlust konstant ist. Bei direkter Downlink-Kommunikation kann man sehen, dass die minimalisierenden Maße auch rotationssymmetrisch sind, und man kann diese Maße auch mit Variationsrechnung finden. Wir zeigen auch, dass die Konfigurationen mit unerwartet viel frustrierten Benutzern vorlegen auch insgesamt im System unerwartet viel Benutzer. Diese Betrachtung wird von den Simulationen des Abschnitts \ref{I'm a simulant} bekräftigt. Diese Simulationen des markierten Poisson-Punktprozesses $\mathbf X^\lambda$ zeigen weiterhin, dass die durchschnitt\-liche Lautstärke der Benutzer in den minimalisierenden Konfigurationen nicht unerwartet groß sind. In Abschnitt \ref{bélus} wird bewiesen, dass das Ergebnis in \eqref{megint} mit $\mathbf b= 0$ den Limes $\mathbf b \downarrow 0$ angleicht. Danach kann man im Spezialfall aus Abschnitt \ref{ráadás} die Form des minimalisierenden Maßes mit $b=0$ ableiten, auch durch Variationsrechnung. 
In den Abscnitten \ref{kernel} und \ref{Fernsehturm} lockern wir die Voraussetzungen des Modells im Kapitel \ref{eleje}. Wir verallgemeinern das originale Modell mit ortsabhaängingen Fadings bzw. mit zufälligem Fadingswert an der Basisstation $o$ und erhalten Resultate, die analog zu denen in Abschnitt \ref{kijelentés} sind.
\newpage
\section{Index of notations} \label{indexofnotations} 
\begin{flushleft}
General notations and abbreviations independent of our model: \\
\begin{tabular}{ll}
$\mathcal{M}(Y)$ & set of finite measures on the space $Y$ \\
$\mathcal{M}_1(Y)$ & set of probability measures on the space $Y$ \\
$h(\nu \vert \mu)$ & relative entropy of $\nu \in \mathcal{M}(Y)$ w.r.t. $\mu \in \mathcal{M}(Y)$ \\
$A^B$ & set of functions with domain $B$ mapping to $A$ (for arbitrary sets $A,B$) \\
$\Lambda_{X}(\cdot)$ & logarithmic moment generating function of the random variable $X$ \\ 
$\Lambda^{\ast}_{X_1}(\cdot)$ & Fenchel--Legendre transform of the logarithmic moment generating function $\Lambda_{X_1}(\cdot)$ \\
$\sharp$ & cardinality of countable set (we write $\infty$ for countably infinite) \\
$\mathbb N$ & $=\lbrace 1,2,\ldots \rbrace$ \\
$\mathbb N_0$ & $=\lbrace 0,1,2,\ldots \rbrace$ \\
$a \vee b$ & $=\max \lbrace a,b \rbrace $ ($a,b \in \mathbb R$) \\
$a \wedge b$ & $=\min \lbrace a, b \rbrace$ ($a, b \in \mathbb R$) \\
$\nu(f(\cdot))$ & $= \int_X f(x) \nu(\d x)$ ($\nu \in \mathcal{M}(X)$, $f: X \to \mathbb R$ measurable) \\
$\mathcal{B}(X)$ & Borel $\sigma$-algebra of the topological space $X$ \\
$A^o$ & interior of the set $A$ \\
$\overline{A}$ & closure of the set $A$ \\
$A^c$ & complement of the set $A$ \\
$\partial A$ & boundary of the set $A$ \\
$\mathds 1 \lbrace A \rbrace$ & indicator function of the measurable set $A$ \\
l.s.c. & lower semicontinuous \\
u.s.c. & upper semicontinuous \\
PRM & Poisson random measure \\
LD & large deviation(s) \\
LDP & large deviation principle \\
i.i.d. & independent and identically distributed \\
a.s. & almost surely \\
l.h.s. & left hand side \\
r.h.s. & right hand side \\
w.r.t. & with respect to \\

\end{tabular}
\end{flushleft}
\begin{flushleft}
Model definition \textbf{without fadings and mobility}:\\
\begin{tabular}{ll}
$d$ & dimension of the model \\
$W$ & communication area: $[-r,r]^d \subset \mathbb R^d$, with an integer $r \geq 1$ \\
$\mu$ & finite Borel measure on $W$, absolutely continuous w.r.t. the Lebesgue measure \\
$X^\lambda$ & Poisson point process on $W$ with intensity $\lambda \mu$ \\
$\ell$ & path-loss function \\
$\ellmin$ & minimal value of $\ell(\vert x-y \vert)$ with $x,y \in W$ \\
$\ellmax$ & maximal value of $\ell(\vert x-y \vert)$ with $x,y \in W$ \\
$J_2$ & Lipschitz continuity parameter of $\ell$ \\
$L_\lambda$ & rescaled empirical measure of $X^\lambda$ \\
$o$ & base station: origin of $\mathbb R^d$
\end{tabular}
\end{flushleft}
\begin{flushleft}
Quantities corresponding fadings and user-fading pairs: \\
\begin{tabular}{ll}
$F_0$ & fading variable (in the case of i.i.d. fadings) \\
$\Fmin$ & essential infimum of $F_0$ \\
$\Fmax$ & essential supremum of $F_0$ \\
$F_o$ & fading value of $o$ (preliminary given, but in Section \ref{Fernsehturm} random) \\
$\mathbf W$ & $=W \times (0, \infty)$. Under Assumption \ref{szamár}, we write $\mathbf W=W \times [\Fmin, \Fmax]$ \\
$(\Omega, \mathcal F, \mathbb P)$ & probability space on which $F_0$ is defined \\
$\zeta$ & = $\mathbb P \circ F_0^{-1}$ (fading distribution) \\
$(\Omega_2, \mathcal F_2, \mathbb P_2)$ & $=(\Omega_0 \times \Omega_1, \mathcal F_1 \tensor \mathcal F_2, \mathbb P_1 \tensor \mathbb P_2)$ \\
$\mathbf X^\lambda$ & $=\lbrace (X_i, F_{X_i}) \rbrace_{X_i \in X^\lambda}$ (marked Poisson process of user-fading pairs) \\
$\mathbf L_\lambda$ & rescaled empirical measure of $\mathbf X^\lambda$ \\
$\mu'$ & $=\mu \times \zeta$ (intensity measure of $\mathbf X^1$) \\
\end{tabular}
\end{flushleft}
\begin{flushleft}
SIR-related quantities:
\begin{tabular}{ll}
$L((X_i,F_{X_i}),~x)$ & $=\ell(|X_i-x |) F_{X_i}$ random path-loss from $(X_i,F_{X_i}) \in \mathbf X^\lambda$ to $x \in W$ \\
$\SIR(\cdot,\cdot,\nu)$ & signal-to-interference ratio w.r.t. $\nu \in \mathcal{M}(\mathbf W)$ \\
$\SIR_\lambda(\cdot,\cdot,\mathbf L_\lambda)$ & $=\frac{1}{\lambda}\SIR(\cdot,\cdot,\mathbf L_\lambda)$ \\
$\mathcal{I}_\lambda$ & interference w.r.t. $\SIR_\lambda$ \\
$c$ & SIR threshold \\
QoS & quality of service \\
$g$ & increasing function describing $\SIR$ perception \\
$D(\cdot)$ & $=g(\SIR(\cdot))$: QoS for direct communication \\
$\tilde{c}_+$ & maximal value of $g$; thus also the maximal $\SIR$ value \\
$[0,\tilde{\mathbf c}_+)$ & $=[0,\tilde{c}_+)^4$ \\
$\tilde{\varrho}_+$ & minimal value that $g$ takes to $\tilde{c}_+$ \\
$\beta'_o$ & $=\min \lbrace 1, \frac{\ellmin \Fmin}{\tilde{\varrho}'_+ \ellmax \Fmax} \rbrace$ \\ & (threshold value of $L_\lambda(W)$, under this all $D$ quantities equal $\tilde{c}_+$) \\
$\Gamma$ & maximum of direct and relayed SIR value w.r.t. a given relay \\
$R((x,u),~(y,v),~\nu)$ & QoS for arbitrary (possibly relayed) communication \\ & for a transmission between $(x,u)$ and $(y,v)$, w.r.t. $\nu$ \\
$\tau$ & decreasing function on $[0, \infty)$ \\
$F$ & increasing function on $\mathcal{M}(\mathbf W)^4$ \\
$G(\nu,~\tau,~\text{up})$ & rescaled measure given by $\frac{\d G(\nu,~\tau,~\text{up})}{\d \nu}(\cdot)=\tau(R(\cdot,~o,~\nu))$ \\ 
$G(\nu,~\tau,~\text{up-dir})$ & rescaled measure given by $\frac{\d G(\nu,~\tau,~\text{up-dir})}{\d \nu}(\cdot)=\tau(D(\cdot,~o,~\nu))$ \\ 
$G(\nu,~\tau,~\text{do})$ & rescaled measure given by $\frac{\d G(\nu,~\tau,~\text{do})}{\d \nu}(\cdot)=\tau(R(o,~\cdot,~\nu))$ \\ 
$G(\nu,~\tau,~\text{do-dir})$ & rescaled measure given by $\frac{\d G(\nu,~\tau,~\text{do-dir})}{\d \nu}(\cdot)=\tau(D(o,~\cdot,~\nu))$ \\ 
$\tau_c(\cdot)$ &$=\mathds 1 {\lbrace \cdot < c \rbrace}$ \\
$b$ & proportion of users under QoS level $c$ \\
$F_b$ & $= (\nu \mapsto -\infty \mathds 1 {(\nu(\mathbf W)) >b)})$ \\
\end{tabular}
\end{flushleft}
\begin{flushleft}
Some mobility-related quantities (from \cite{cikk}, used in Section \ref{explanation}, but not in the model of Chapter \ref{eleje}): \\
\begin{tabular}{ll}
$J_1$ & Lipschitz continuity parameter of the trajectories of $X^\lambda$ (which take values in $W$) \\
$\mathcal{L}$ & space of Lipschitz continuity trajectories on $W$ parameter $J_1$ \\
$I=[0,T]$ & time horizon of the process with mobility \\
$\pi_t$ & projection at time $t \in I$: $\mathcal{L} \to W$, $x \mapsto x_t$ \\
$\overline{\nu}^{\text{up}}[\tau]$ & analogue of $G(\nu,~\tau,\text{up})$ with mobility and constant fadings \\
$\overline{\nu}[\boldsymbol \tau]$ & analogue of $G(\nu,~\boldsymbol{\tau})$ \\
\end{tabular}
\end{flushleft}
\begin{flushleft}
Discretization of spatial and fading quantities:\\
\begin{tabular}{ll}
$\mathbb B$ & $=\lbrace 3^{-n} \vert n \in \mathbb N \rbrace$ (set of discretization parameters) \\
$\delta$ & $\in \mathbb B$ discretization parameter \\
$W_\delta$ & discretized version of the communication area $W$ \\
$\varrho'=(\varrho_1,\varrho_2)$ & discretization operator: \\ & $\varrho_1$ acts on the space coordinate, $\varrho_2$ on the fading coordinate \\
$\nu^{\varrho_1}$ & $=\nu \circ \varrho_1^{-1} \in \mathcal{M}(W_\delta)$, where $\nu \in \mathcal{M}(W)$ \\ & (measure induced from $\nu$ in the discretized setting) \\ 
$\varrho_2(F_0)$ & discretized fading variable \\
$\Fdelta$ & support of $F_0^{\varrho_2}$ \\
$\mathbf W_\delta$ & $=W_\delta \times \Fdelta$ \\
$\nu^{\varrho'}$ & $=\nu \circ \varrho'^{-1} \in \mathcal{M}(\mathbf W_\delta)$, where $\nu \in \mathcal{M}(\mathbf W)$ \\ & (measure induced from $\nu$ in the discretized setting) \\
$\imath'$ & $W_\delta \mapsto W$ inclusion operator \\
$\nu^{\imath'}$ & $=\nu \circ \imath'^{-1} \in \mathcal{M}(\mathbf W)$, where $\nu \in \mathcal{M}(\mathbf W_\delta)$ \\ & (measure induced from $\nu$ in the continuous setting) \\
Assumption \ref{szamár} & $\Finite$ and $F_o=(\Fmin+\Fmax)/2$ \\
$X^\lambda_\delta$ & $=\lambda L_\lambda^{\delta}$ (discretized spatial Poisson point process) \\
$\mathbf X^\lambda_\delta$ & $=\lambda \mathbf L_\lambda^{\delta}$ (discretized marked Poisson point process of space-fading pairs) \\
$\Lambda_\delta(\upsilon,s)$ & $=(\upsilon,s) + \left([-\delta r, \delta r]^d \times \left[ -\delta (\Fmax-\Fmin), \delta(\Fmax-\Fmin) \right] \right)$ \\ & (discretization sub-cubes) \\ 
$\kappa_\delta$ & $=\min_{(x,u) \in \mathbf W_\delta:~\mu'^{\varrho'}((x,u))>0} \mu'^{\varrho'}((x,u))$ \\
$N(\lambda)$ & $=\lambda L_\lambda(W)$ (number of users in the whole communication area $W$)
\end{tabular}
\end{flushleft}
\begin{flushleft}
Notation used in proofs of Chapters \ref{eleje} and \ref{effect} generally:\\
\begin{tabular}{ll}
$V(\nu)$ & zero set of the measure $\nu \in \mathcal{M}(\mathbf W_\delta)$ \\
$ \Vert  (x,u) \Vert $ & $=\vert x \vert + \vert u \vert = \Vert x \Vert_2 + \vert u \vert \quad$ ($\ell^1$ product norm of $(x,u) \in \mathbb R^d \times \mathbb R$) \\
$\pi_i$ & projection to one communication type: $\mathbb R^4 \to \mathbb R$, $(a_j)_{j \in \four} \mapsto a_i$ \\
$m_i$ & name of the means of communication: \\ & $m_1$=up, $m_2$=up-dir, $m_3$=do, $m_4$=do-dir \\
$\mathbf K=(K_1,~K_2,~K_3,~K_4)$ & minimal SIR level (vector of essential infima of QoS values, \\ $=(K_{\text{up}},K_{\text{up-dir}},K_{\text{do}},K_{\text{do-dir}})$ & $K_i$ corresponds to the communication type $m_i$)
\end{tabular}
\end{flushleft}
\addcontentsline{toc}{chapter}{Index}
\printindex



\end{document}